\documentclass[11pt]{article}
\usepackage{amsmath,amsthm,amssymb}
\usepackage[usenames,dvipsnames]{xcolor}
\usepackage{enumerate}
\usepackage{graphicx}
\usepackage{cite}
\usepackage{comment}
\usepackage{oands}
\usepackage{tikz}
\usepackage{changepage}
\usepackage{bbm}
\usepackage{mathtools}
\usepackage[margin=1in]{geometry}
\usepackage[pagewise,mathlines]{lineno}
\usepackage{appendix}
\usepackage{stmaryrd} 
\usepackage{multicol}
\usepackage{microtype}
\usepackage[colorinlistoftodos,textsize=footnotesize]{todonotes}

\definecolor{dgreen}{RGB}{0,180,0}

\usepackage[pdftitle={Harmonic functions on mated-CRT maps},
  pdfauthor={Ewain Gwynne, Jason Miller, and Scott Sheffield},
colorlinks=true,linkcolor=NavyBlue,urlcolor=RoyalBlue,citecolor=PineGreen,bookmarks=true,bookmarksopen=true,bookmarksopenlevel=2,unicode=true,linktocpage]{hyperref}

\theoremstyle{plain}
\newtheorem{thm}{Theorem}[section]
\newtheorem{cor}[thm]{Corollary}
\newtheorem{lem}[thm]{Lemma}
\newtheorem{prop}[thm]{Proposition}

\def\@rst #1 #2other{#1}
\newcommand\MR[1]{\relax\ifhmode\unskip\spacefactor3000 \space\fi
  \MRhref{\expandafter\@rst #1 other}{#1}}
\newcommand{\MRhref}[2]{\href{http://www.ams.org/mathscinet-getitem?mr=#1}{MR#2}}

\theoremstyle{definition}
\newtheorem{defn}[thm]{Definition}
\newtheorem{remark}[thm]{Remark}

\numberwithin{equation}{section}

\newcommand{\dsb}{\begin{adjustwidth}{2.5em}{0pt}
\begin{footnotesize}}
\newcommand{\dse}{\end{footnotesize}
\end{adjustwidth}}

\newcommand{\ssb}{\begin{adjustwidth}{2.5em}{0pt}}
\newcommand{\sse}{\end{adjustwidth}}

\newcommand{\aryb}{\begin{eqnarray*}}
\newcommand{\arye}{\end{eqnarray*}}
\def\alb#1\ale{\begin{align*}#1\end{align*}}
\def\allb#1\alle{\begin{align}#1\end{align}}
\newcommand{\eqb}{\begin{equation}}
\newcommand{\eqe}{\end{equation}}
\newcommand{\eqbn}{\begin{equation*}}
\newcommand{\eqen}{\end{equation*}}

\newcommand{\BB}{\mathbbm}
\newcommand{\ol}{\overline}

\newcommand{\op}{\operatorname}

\newcommand{\bd}{\mathbf}
\newcommand{\im}{\operatorname{Im}}
\newcommand{\re}{\operatorname{Re}}
\newcommand{\frk}{\mathfrak}
\newcommand{\eqD}{\overset{d}{=}}
\newcommand{\ep}{\varepsilon}
\newcommand{\rta}{\rightarrow}

\newcommand{\wt}{\widetilde}
\newcommand{\wh}{\widehat} 
\newcommand{\mcl}{\mathcal}

\newcommand{\bdy}{\partial}
\newcommand{\rng}{\mathring}

\let\originalleft\left
\let\originalright\right
\renewcommand{\left}{\mathopen{}\mathclose\bgroup\originalleft}
\renewcommand{\right}{\aftergroup\egroup\originalright}

\title{Harmonic functions on mated-CRT maps}

\date{  }
\author{
\begin{tabular}{c} Ewain Gwynne\\[-5pt]\small MIT \end{tabular}
\begin{tabular}{c} Jason Miller\\[-5pt]\small Cambridge \end{tabular}
\begin{tabular}{c} Scott Sheffield\\[-5pt]\small MIT \end{tabular}
}

\begin{document}

\maketitle

\begin{abstract}
A \emph{mated-CRT map} is a random planar map obtained as a discretized mating of correlated continuum random trees.
Mated-CRT maps provide a coarse-grained approximation of many other natural random planar map models (e.g., uniform triangulations and spanning tree-weighted maps), and are closely related to $\gamma$-Liouville quantum gravity (LQG) for $\gamma \in (0,2)$ if we take the correlation to be $-\cos(\pi\gamma^2/4)$. 
We prove estimates for the Dirichlet energy and the modulus of continuity of a large class of discrete harmonic functions on mated-CRT maps, which provide a general toolbox for the study of the quantitative properties of random walk and discrete conformal embeddings for these maps.

For example, our results give an independent proof that the simple random walk on the mated-CRT map is recurrent, and a polynomial upper bound for the maximum length of the edges of the mated-CRT map under a version of the Tutte embedding.  
Our results are also used in other work by the first two authors which shows that for a class of random planar maps --- including mated-CRT maps and the UIPT --- the spectral dimension is two (i.e., the return probability of the simple random walk to its starting point after $n$ steps is $n^{-1+o_n(1)}$) and the typical exit time of the walk from a graph-distance ball is bounded below by the volume of the ball, up to a polylogarithmic factor. 
\end{abstract}

\tableofcontents

\section{Introduction}
\label{sec-intro}

\subsection{Overview} 
\label{sec-overview}

There has been substantial interest in random planar maps in recent years.  
One reason for this is that random planar maps are the discrete analogs of \emph{$\gamma$-Liouville quantum gravity} (LQG) surfaces for $\gamma \in (0,2)$. 
Such surfaces have been studied in the physics literature since the 1980's~\cite{polyakov-qg1,polyakov-qg2}, and can be rigorously defined as metric measure spaces with a conformal structure~\cite{shef-kpz,lqg-tbm1,lqg-tbm2,lqg-tbm3,gm-uniqueness}. The parameter $\gamma$ depends on the particular type of random planar map model under consideration. For example, $\gamma=\sqrt{8/3}$ for uniform random planar maps, $\gamma = \sqrt 2$ for spanning-tree weighted maps, and $\gamma = \sqrt{4/3}$ for bipolar-oriented maps. 

Central problems in the study of random planar maps include describing the large-scale behavior of graph distances;
analyzing statistical mechanics models on the map; and understanding the \emph{conformal structure} of the map, which involves studying the simple random walk on the map and various ways of embedding the map into $\BB C$. Here we will focus on this last type of question for a particular family of random planar maps called \emph{mated-CRT maps}, which (as we will discuss more just below) are directly connected to many other random planar map models and to LQG. 

To define mated-CRT maps, fix $\gamma \in (0,2)$ (which corresponds to the LQG parameter) and let $(L,R) : \BB R\rta \BB R^2$ be a pair of correlated, two-sided standard linear Brownian motions normalized so that $L_0 = R_0 = 0$ with correlation $-\cos(\pi\gamma^2/4)$, i.e., $\op{corr}(L_t , R_t) = -\cos(\pi\gamma^2/4)$ for each $t\in\BB R  \setminus \{0\}$ and $(L_t,R_t)$ can be obtained from a standard planar Brownian motion by applying an appropriate linear transformation. We note that the correlation ranges from $-1$ to 1 as $\gamma$ ranges from $0$ to 2. 
The mated CRT map is the random planar map obtained by mating, i.e., gluing together, discretized versions of the continuum random trees (CRT's) constructed from $L$ and $R$~\cite{aldous-crt1,aldous-crt2,aldous-crt3}. 
More precisely, the \emph{$\ep$-mated-CRT map}\footnote{In this paper we only consider the mated-CRT map with the plane topology. Mated-CRT maps with the disk and sphere topology are studied in~\cite{gms-tutte}.} associated with $(L,R)$ is the random graph with vertex set $\ep\BB Z$, with two vertices $x_1,x_2\in\ep\BB Z$ with $x_1 < x_2$ connected by an edge if and only if either
\eqb \label{eqn-inf-adjacency}
\left( \inf_{t\in [x_1-\ep ,x_1]} L_t\right) \vee \left( \inf_{t\in [x_2-\ep ,x_2]} L_t \right)  \leq \inf_{t\in [x_1,x_2-\ep]} L_t; 
\eqe
or the same holds with $R$ in place of $L$. If $|x_2-x_1| > \ep$ and~\eqref{eqn-inf-adjacency} holds for both $L$ and $R$, then $x_1$ and $x_2$ are connected by two edges.
We note that the law of the planar map $\mcl G^\ep$ does not depend on $\ep$ due to Brownian scaling, but for reasons which will become apparent just below it is convenient to think of the whole collection of maps $\{\mcl G^\ep\}_{\ep > 0}$ coupled together with the same Brownian motion $(L,R)$. 
See Figure~\ref{fig-mated-crt-map} for an illustration of the definition of $\mcl G^\ep$ and an explanation of how to endow it with a canonical planar map structure under which it is a triangulation.

\begin{figure}[ht!]
 \begin{center}
\includegraphics[scale=.75]{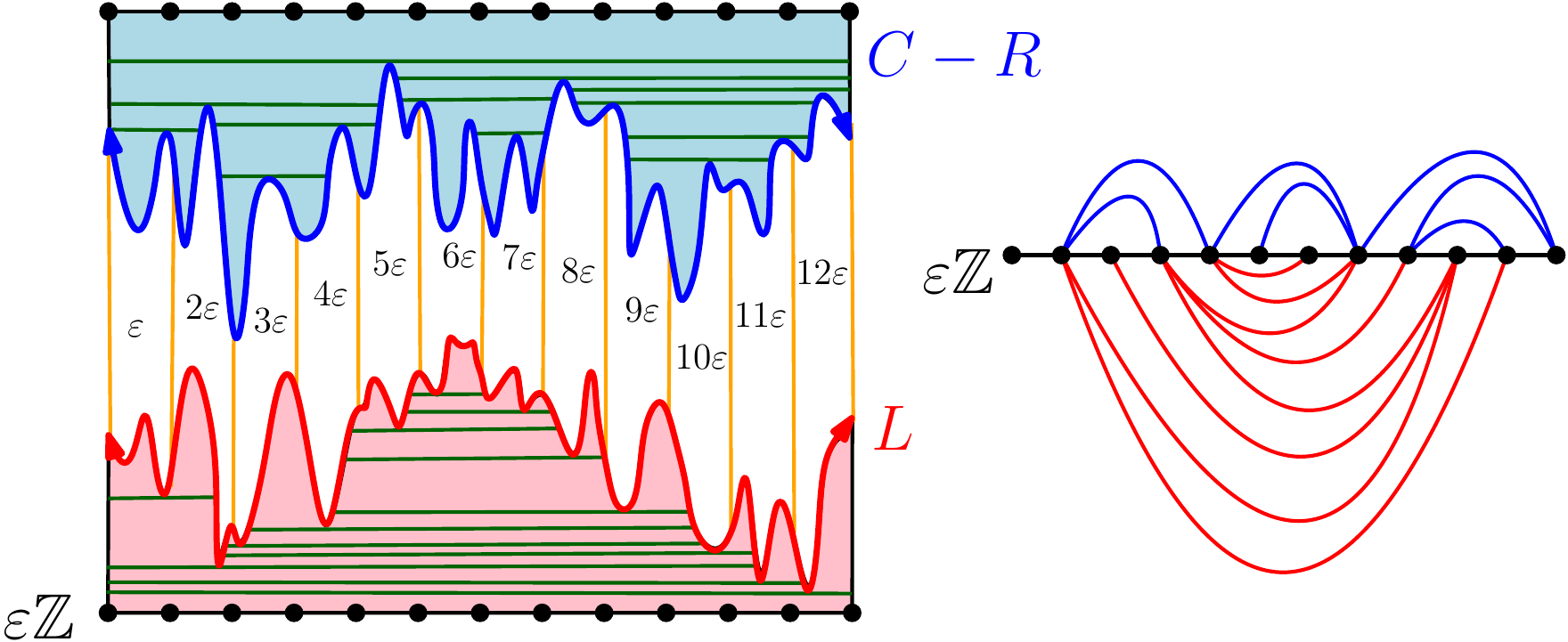} 
\caption{\textbf{Left:} To construct the mated-CRT map $\mcl G^\ep$ geometrically, one can draw the graph of $L$ (red) and the graph of $C-R$ (blue) for some large constant $C > 0$ chosen so that the parts of the graphs over some time interval of interest do not intersect. Here, this time interval is $[0,12\ep]$. One then divides the region between the graphs into vertical strips (boundaries shown in orange). Each vertical strip corresponds to the vertex $x\in \ep \BB Z$ which is the horizontal coordinate of its rightmost points. Vertices $x_1,x_2\in \ep\BB Z$ are connected by an edge if and only if the corresponding vertical strips are connected by a horizontal line segment which lies under the graph of $L$ or above the graph of $C-R$. For each pair of vertices for which the condition holds for $L$ (resp.\ $C-R$), we have drawn the lowest (resp.\ highest) segment for which joins the corresponding vertical strips in green. Equivalently, for each $x\in \ep\BB Z$, we let $t_x$ be the time in $[x-\ep,x]$ at which $L$ attains its minimum value and we draw in green the longest horizontal segment under the graph of $L$ which contains $(t_x , L_{t_x})$; and we perform a similar procedure for $R$. 
Note that consecutive vertices are always joined by an edge.
\textbf{Right:} One can draw the graph $\mcl G^\ep$ in the plane by connecting two vertices $x_1,x_2 \in  \ep \BB Z$ by an arc above (resp.\ below) the real line if the corresponding vertical strips are connected by a horizontal segment above (resp.\ below) the graph of $L$ (resp.\ $C-R$); and connecting each pair of consecutive vertices of $\ep \BB Z$ by an edge. This gives $\mcl G^\ep$ a planar map structure. 
With this planar map structure, each face of $\mcl G^\ep$ corresponds to a horizontal strip below the graph of $L$ or above the graph of $C-R$ which is bounded by two horizontal green segments and two segments of either the graph of $L$ or the graph of $C-R$. 
Almost surely, neither $L$ nor $R$ attains a local minimum at any point in $\ep\BB Z$ and neither $L$ nor $R$ has two local minima where it attains the same value.
From this, it follows that a.s.\ the boundary of each horizontal strip intersects the boundaries of exactly three vertical strips (two of these intersections each consist of a segment of the graph of $L$ or $C-R$, and one is a single point). 
This means that a.s.\ each face of $\mcl G^\ep$ has exactly three vertices on its boundary, so $\mcl G^\ep$ is a triangulation.
}\label{fig-mated-crt-map}
\end{center}
\end{figure}

Mated-CRT maps are an especially natural family of random planar maps to study. One reason for this is that these maps provide a bridge between many other interesting random planar map models and their continuum analogs: LQG surfaces. 
Let us now explain the precise sense in which this is the case, starting with the link between mated-CRT maps and other random planar map models; see Figure~\ref{fig-bijections} for an illustration. 

A number of random planar maps can be bijectively encoded by pairs of discrete random trees (equivalently, two-dimensional random walks) by discrete versions of the above definition of the mated-CRT map. Consequently, the mated-CRT map (with $\gamma$ depending on the particular model) can be viewed as a coarse-grained approximation of any of these random planar maps. For example, Mullin's bijection~\cite{mullin-maps} (see~\cite{bernardi-maps,shef-burger,chen-fk} for more explicit expositions) shows that if we replace $(L,R)$ by a two-sided simple random walk on $\BB Z^2$ and construct a graph with adjacency defined by a direct discrete analog of~\eqref{eqn-inf-adjacency}, then we obtain the infinite-volume local limit of random planar maps sampled with probability proportional to the number of spanning trees they admit. The left-right ordering of the vertices corresponds to the depth-first ordering of the spanning tree. There are similar bijective constructions, with different laws for the random walk, which produce the uniform infinite planar triangulation (UIPT)~\cite{bernardi-dfs-bijection,bhs-site-perc} as well as a number of natural random planar maps decorated by statistical mechanics models~\cite{shef-burger,gkmw-burger,kmsw-bipolar,lsw-schnyder-wood}.

At least in the case when the encoding walk has i.i.d.\ increments, one can use a strong coupling result for random walk and Brownian motion~\cite{kmt,zaitsev-kmt}, which says that the random walk $\mcl Z$ and the Brownian motion $Z$ can be coupled together so that $\max_{-n\leq j \leq n} |\mcl Z_j - Z_j| = O(\log n)$ with high probability, to couple one of these other random planar maps with the mated-CRT map. This allows us to compare the maps directly. 
This approach is used in~\cite{ghs-map-dist} to couple the maps in such a way that graph distances differ by at most a polylogarithmic factor, which allows one to transfer the estimates for graph distances in the mated-CRT map from~\cite{ghs-dist-exponent} to a larger class of random planar map models. A similar approach is used in~\cite{gm-spec-dim,gh-displacement} to prove estimates for random walk on these same random planar map models.

On the other hand, the mated-CRT map possesses an a priori relationship with SLE-decorated Liouville quantum gravity. 
We will describe this relationship in more detail in Section~\ref{sec-peanosphere}, but let us briefly mention it here. 
Suppose $h$ is the random distribution on $\BB C$ which describes a \emph{$\gamma$-quantum cone}, a particular type of $\gamma$-LQG surface. 
Let $\eta$ be a whole-plane space-filling SLE$_{\kappa'}$ from $\infty$ to $\infty$ with\footnote{
Here we follow the imaginary geometry~\cite{ig1,ig2,ig3,ig4} convention of writing $\kappa'$ instead of $\kappa$ for the SLE parameter when it is bigger than 4.
} $\kappa' =16/\gamma^2 > 4$,
sampled independently of $h$ and then parameterized by $\gamma$-LQG mass with respect to $h$ (we recall the definition and basic properties of space-filling SLE in Section~\ref{sec-wpsf-prelim}). It follows from~\cite[Theorem 1.9]{wedges} that if we let $\mcl G^\ep$ for $\ep > 0$ be the graph whose vertex set is $\ep\BB Z$, with two vertices $x_1,x_2 \in \ep\BB Z$ connected by an edge if and only if 
the corresponding cells $\eta([x_1-\ep , x_1])$ and $\eta([x_2-\ep ,x_2])$ share a non-trivial boundary arc, then $\{\mcl G^\ep\}_{\ep > 0}$ has the same law as the family of mated-CRT maps defined above.
 
The above construction gives us an embedding of the mated-CRT map into $\BB C$ by mapping each vertex to the corresponding space-filling SLE cell. 
It is shown in~\cite{gms-tutte} that the simple random walk on $\mcl G^\ep$ under this embedding converges in law to Brownian motion modulo time parameterization (which implies that the above SLE/LQG embedding is close when $\ep $ is small to the so-called \emph{Tutte embedding}). 
The main theorem of~\cite{gms-tutte} is proven using a general scaling limit result for random walk in certain random environments~\cite{gms-random-walk}, which in turn is proven using ergodic theory. 
The theorem gives us control on the large-scale behavior of random walk and harmonic functions on $\mcl G^\ep$ under the SLE/LQG embedding, but provides very little information about their behavior at smaller scales and no quantitative bounds for rates of convergence.
 
The goal of this paper is to prove quantitative estimates for discrete harmonic functions on $\mcl G^\ep$, which can be applied at mesoscopic scales and which include polynomial bounds for the rate of convergence of the probabilities that the estimates hold. In particular, we obtain estimates for the Dirichlet energy and the modulus of continuity of a large class of such discrete harmonic functions. See Section~\ref{sec-main-results} for precise statements. 
We will not use the main theorem of~\cite{gms-tutte} in our proofs. Instead, we will rely on a quantitative law-of-large-numbers type bound for integrals of functions defined on $\BB C$ against certain quantities associated with the cells $\eta([x-\ep,x])$. 

Our results provide a general toolbox for the study of random walk on mated-CRT maps, and thereby random walk on other random planar maps thanks to the coupling results discussed above.  
For example, our results give an independent proof that the random walk on the mated-CRT map is recurrent (Theorem~\ref{thm-green-lower}; this can also be deduced from the general criterion of Gurel-Gurevich and Nachmias \cite{gn-recurrence}, see Section~\ref{sec-mated-crt-map}).
We also obtain a polynomial (in $\ep$) upper bound for the maximum length of the edges of the mated-CRT map under the so-called Tutte embedding with identity boundary data (Corollary~\ref{cor-max-edge}). We note that~\cite{gms-tutte} shows only that the maximum length of these embedded edges tends to zero as $\ep\rta 0$, but does not give any quantitative bound for the rate of convergence.

The results of this paper will also play a crucial role in the subsequent work~\cite{gm-spec-dim}, which proves that the spectral dimension of a large class of random planar maps --- including mated-CRT maps, spanning-tree weighted maps, and the UIPT --- is two (i.e., the return probability after $n$ steps is $n^{-1+o_n(1)}$) and also proves a lower bound for the graph distance displacement of the random walk on these maps which is correct up to polylogarithmic errors (the complementary upper bound is proven in~\cite{gh-displacement}). 
We expect that our results may also have eventual applications to the study of discrete conformal embedddings of random planar maps, e.g., to the problem of showing that the maximal size of the faces of certain random planar maps --- like uniform triangulations and spanning tree-weighted maps --- under the Tutte embedding tends to 0. See the discussion just after Corollary~\ref{cor-max-edge}. 
 
One way to think about the approach used in this paper is as follows. A powerful technique for studying random walk and harmonic functions on random planar maps is to embed the map into $\BB C$ in some way, then consider how the embedded map interacts with paths and functions in $\BB C$. A number of recent works have used this technique with the embedding given by the circle packing of the map~\cite{stephenson-circle-packing}; see, e.g.,~\cite{benjamini-schramm-topology,gn-recurrence,abgn-bdy,gill-rohde-type,angel-hyperbolic,lee-conformal-growth,lee-uniformizing}. Here, we study random walk and harmonic functions on the mated-CRT map using the embedding of this map coming from SLE/LQG instead of the circle packing. For many quantities of interest, one can get stronger estimates using this embedding than using circle packing since we have good estimates for the behavior of space-filling SLE and the $\gamma$-LQG measure.

\begin{figure}[ht!]
 \begin{center}
\includegraphics[scale=.8]{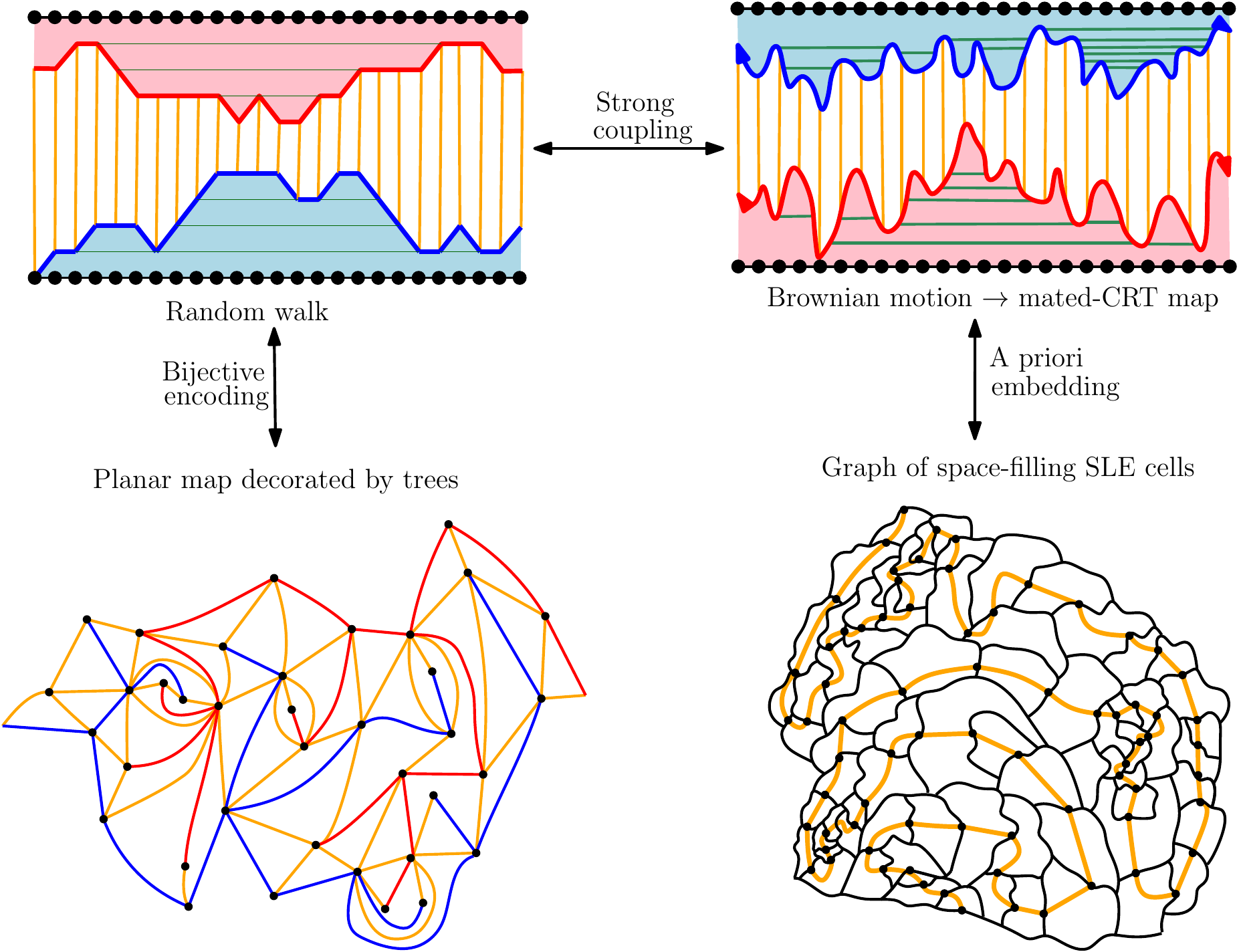} 
\caption{A visual representation of the relationship between mated-CRT maps and other objects. \textbf{Left:} Various random planar maps (e.g., the UIPT or spanning-tree weighted maps) can be encoded by means of a two-dimensional random walk via a discrete version of the construction of the mated-CRT map (we will not use these bijections in this paper). \textbf{Right:} The mated-CRT map is defined using a pair of Brownian motions and has an embedding into $\BB C$ as the adjacency graph on the ``cells" $\eta([x-\ep,x])$ for $x\in\ep\BB Z$ of a space-filling SLE parameterized by $\gamma$-LQG mass. This paper proves estimates for the mated-CRT map under this embedding. One can transfer these estimates to other random planar maps (up to a polylogarithmic error) using a strong coupling of the encoding walk for the other planar map and the Brownian motion used to define the mated-CRT map; see~\cite{ghs-map-dist,gm-spec-dim}.
}\label{fig-bijections}
\end{center}
\end{figure}

\subsection{Mated-CRT maps and SLE-decorated Liouville quantum gravity} 
\label{sec-peanosphere}

We now describe the connection between mated-CRT maps and SLE-decorated LQG, as alluded to at the end of Section~\ref{sec-overview}. 
This connection gives an embedding of the mated-CRT map into $\BB C$, which will be our main tool for analyzing mated-CRT maps. Moreover, most of our main results will be stated in terms of this embedding. 
See Section~\ref{sec-prelim} for additional background on the objects involved.

Heuristically speaking, for $\gamma \in (0,2)$ a $\gamma$-LQG surface parameterized by a domain $D\subset \BB C$ is the random two-dimensional Riemannian manifold with metric tensor $  e^{\gamma h}\, dx\otimes dy$, where $h$ is some variant of the Gaussian free field (GFF) on $D$~\cite{shef-gff,ss-contour,ig1,ig4} and $dx\otimes dy$ is the Euclidean metric tensor. This does not make literal sense since $h$ is a random distribution, not a pointwise-defined function. Nevertheless, one can make literal sense of $\gamma$-LQG in various ways. Duplantier and Sheffield~\cite{shef-kpz} constructed the volume form associated with a $\gamma$-LQG surface, a measure $\mu_h$ which is the limit of regularized versions of $e^{\gamma h(z)} \,dz$, where $dz$ denotes Lebesgue measure. One can similarly define a $\gamma$-LQG boundary length measure $\nu_h$ on certain curves in $D$, including $\bdy D$ and SLE$_\kappa$-type curves for $\kappa = \gamma^2$~\cite{shef-zipper}.
These measures are a special case of a more general theory called \emph{Gaussian multiplicative chaos}; see~\cite{kahane,rhodes-vargas-review,berestycki-gmt-elementary}.

Mated-CRT maps are related to SLE-decorated LQG via the peanosphere (or mating-of-trees) construction of~\cite[Theorem 1.9]{wedges}, which we now describe. 
Suppose $h$ is the random distribution on $\BB C$ corresponding to the particular type of $\gamma$-LQG surface called a \emph{$\gamma$-quantum cone}. Then $h$ is a slight modification of a whole-plane GFF plus $-\gamma \log |\cdot|$ (see Section~\ref{sec-lqg-prelim} for more on this field). 
Also let $\kappa' = 16/\gamma^2  > 4$ and let $\eta$ be a whole-plane space-filling SLE$_{\kappa'}$ curve from $\infty$ to $\infty$ sampled independently from $h$ and then parameterized in such a way that $\eta(0) = 0$ and the $\gamma$-LQG mass satisfies $\mu_h(\eta([t_1,t_2])) = t_2 - t_1$ whenever $t_1 , t_2 \in \BB R$ with $t_1 < t_2$ (see Section~\ref{sec-wpsf-prelim} and the references therein more on space-filling SLE).

\begin{figure}[ht!]
 \begin{center}
\includegraphics[scale=.8]{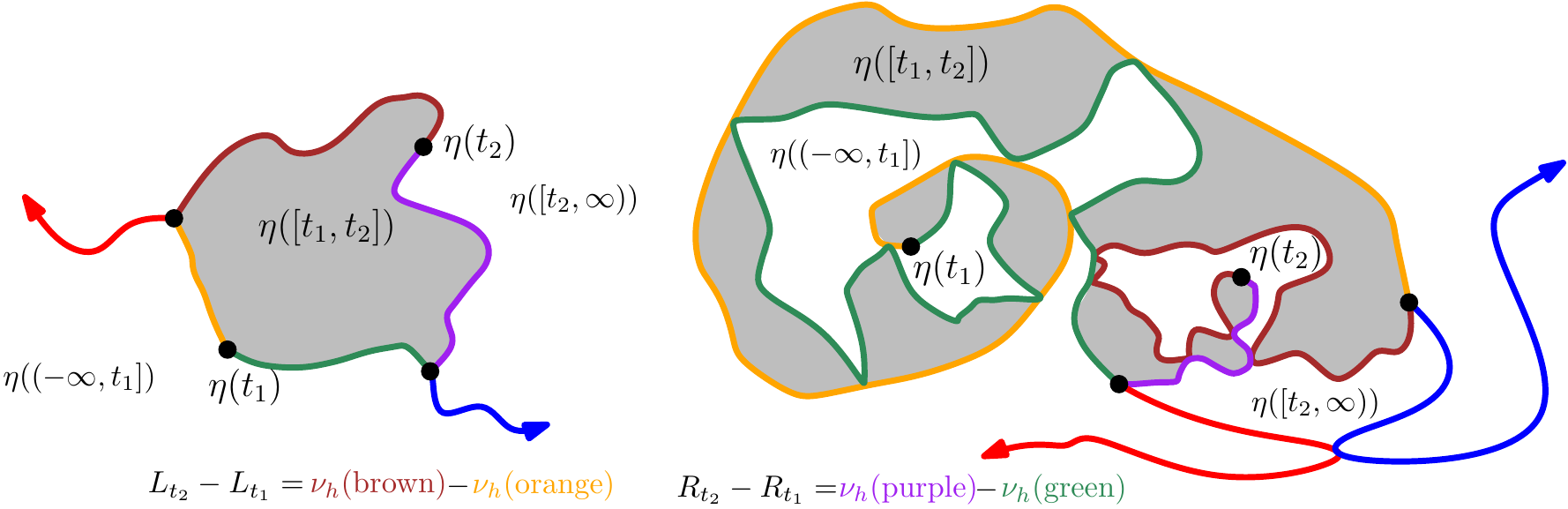} 
\caption{Illustration of the definition of the left/right boundary length process $(L,R)$ for space-filling SLE on a $\gamma$-quantum cone. The left figure corresponds to the case when $\kappa' \geq 8$, so that $\eta([t_1,t_2])$ is simply connected. The right figure corresponds to the case when $\kappa' \in (4,8)$, in which case the topology is more complicated since the left and right boundaries of the curve can intersect each other, but the definition of the left/right boundary length process is the same. In both cases, the intersection of the left (resp.\ right) outer boundaries of $\eta((-\infty,t_1])$ and $\eta([t_2,\infty))$ is shown in red (resp.\ blue). The black dots on the boundary correspond to the endpoints $\eta(t_1)$ and $\eta(t_2)$ and the points where $\eta((-\infty,t_1]) \cap \eta([t_1,t_2])$ and $\eta([t_2,\infty))\cap\eta([t_1,t_2])$ meet. These latter two points are hit by $\eta$ at the times when $L$ and $R$, respectively, attain their minima on $[t_1,t_2]$.
}\label{fig-peanosphere-bm-def}
\end{center}
\end{figure}

Let $\nu_h$ be the $\gamma$-LQG length measure associated with $h$ and define a process $L : \BB R \rta \BB R$ in such way that $L_0 = 0$ and for $t_1,t_2 \in \BB R$ with $t_1<t_2$, 
\allb \label{eqn-peanosphere-bm}
L_{t_2}- L_{t_1} &= \nu_h\left( \text{left boundary of $\eta([t_1,t_2]) \cap \eta([t_2,\infty))$} \right) \notag \\
&\qquad - \nu_h\left( \text{left boundary of $\eta([t_1,t_2]) \cap \eta((-\infty , t_1])$} \right) .
\alle 
Define $R_t$ similarly but with ``right" in place of ``left" and set $Z_t = (L_t , R_t)$. See Figure~\ref{fig-peanosphere-bm-def} for an illustration. It is shown in~\cite[Theorem 1.9]{wedges} that $Z$ evolves as a correlated two-dimensional Brownian motion with correlation $-\cos(\pi\gamma^2/4)$, i.e., $Z$ has the same law as the Brownian motion used to construct the mated-CRT map with parameter $\gamma$ (up to multiplication by a deterministic constant, which does not affect the definition of the mated-CRT map). Moreover, by~\cite[Theorem 1.11]{wedges}, $Z$ a.s.\ determines $(h,\eta)$ modulo rotation and scaling. 

We can re-phrase the adjacency condition~\eqref{eqn-inf-adjacency} in terms of $(h,\eta)$. In particular, for $x_1,x_2 \in \ep\BB Z$ with $x_1 < x_2$,~\eqref{eqn-inf-adjacency} is satisfied if and only if the \emph{cells} $\eta([x_1-\ep , x_1])$ and $\eta([x_2-\ep , x_2])$ intersect along a non-trivial connected arc of their left outer boundaries; and similarly with ``$R$" in place of ``$L$" and ``left" in place of ``right". Indeed, this follows from the explicit description of the curve-decorated topological space $(\BB C , \eta)$ in terms of $(L,R)$ given in~\cite[Section 8.2]{wedges}.

Consequently, the mated-CRT map $\mcl G^\ep$ is precisely the graph with vertex set $\ep\BB Z$, with two vertices connected by an edge if and only if the corresponding cells $\eta([x_1-\ep , x_1])$ and $\eta([x_2-\ep , x_2])$ share a non-trivial connected boundary arc.
The graph on cells is sometimes called the \emph{$\ep$-structure graph} of the curve $\eta$ since it encodes the topological structure of the cells. 
The identification of $\mcl G^\ep$ with the $\ep$-structure graph of $\eta$ gives us an embedding of $\mcl G^\ep$ into $\BB C$ by sending each vertex $x\in\ep\BB Z$ to the point $\eta(x)$. See Figure~\ref{fig-structure-graph-def} for an illustration.

\begin{figure}[ht!]
\begin{center}
\includegraphics[scale=.8]{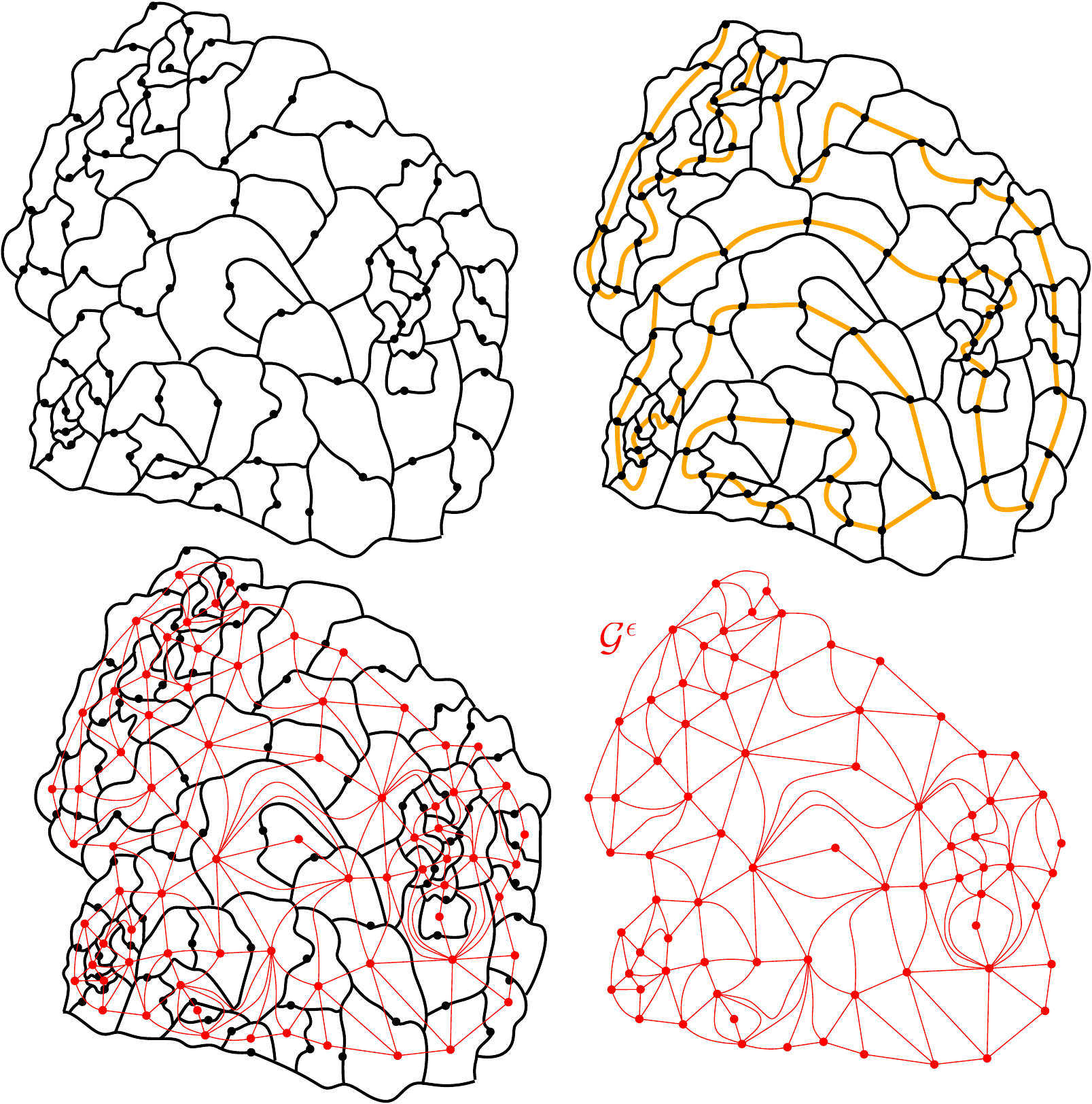} 
\caption{\label{fig-structure-graph-def} \textbf{Top left:} A segment of a space-filling curve $\eta : \BB R\rta \BB C$, divided into cells $  \eta([x-\ep ,x])$ for $x\in\ep\BB Z $. This figure looks like what we would expect to see for a space-filling SLE$_{\kappa'}$ parameterized by $\gamma$-quantum mass when $\kappa' \geq 8$, since this is the range when the curve does not make and fill in bubbles (see Section~\ref{sec-wpsf-prelim}). \textbf{Top right:} Same as top-left but with an orange path showing the order in which cells are hit by $\eta$. \textbf{Bottom left:} A point in each cell is shown in red, and is connected to each adjacent cell by a red edge. 
As explained in Figure~\ref{fig-mated-crt-map}, $\mcl G^\ep$ can be viewed as a planar triangulation.
In the present picture, the faces correspond to the points where three of the black curves meet.
Note that we cannot have more than three black curves meeting at a single point or we would have a face of degree greater than three (this can also be seen directly from the geometry of space-filling SLE; see~\cite[Section 8.2]{wedges}).
\textbf{Bottom right:} If we forget the original cells $\eta([x-\ep,x])$ but keep the red edges we get an embedding of $\mcl G^\ep$ into $\BB C$.  
}
\end{center}
\end{figure}


\subsection{Basic notation} 
\label{sec-basic-notation}

\noindent
We write $\BB N$ for the set of positive integers and $\BB N_0 = \BB N\cup \{0\}$. 
\vspace{6pt}

\noindent
For $a,b \in \BB R$ with $a<b$ and $r > 0$, we define the discrete intervals $[a,b]_{r \BB Z} := [a, b]\cap (r \BB Z)$ and $(a,b)_{r \BB Z} := (a,b)\cap (r\BB Z)$.
\vspace{6pt}
 
\noindent
For $K\subset \BB C$, we write $\op{area}(K)$ for the Lebesgue measure of $K$ and $\op{diam}(K)$ for its Euclidean diameter.
For $r > 0$ and $z\in\BB C$ we write $B_r(z)$ be the open disk of radius $r$ centered at $z$. 
For $K\subset\BB C$, we also write $B_r(K)$ for the (open) set of points $z\in\BB C$ which lie at Euclidean distance less than $r$ from $K$. 
\vspace{6pt}

\noindent
If $a$ and $b$ are two ``quantities'' (i.e., functions from any sort of ``configuration space'' to the real numbers) we write $a\preceq b$ (resp.\ $a \succeq b$) if there is a constant $C > 0$ (independent of the values of $a$ or $b$ and certain other parameters of interest) such that $a \leq C b$ (resp.\ $a \geq C b$). We write $a \asymp b$ if $a\preceq b$ and $a \succeq b$. We typically describe dependence of implicit constants in lemma/proposition statements and require constants in the proof to satisfy the same dependencies.
\vspace{6pt}

\noindent
If $a$ and $b$ are two quantities depending on a variable $x$, we write $a = O_x(b)$ (resp.\ $a = o_x(b)$) if $a/b$ remains bounded (resp.\ tends to 0) as $x\rta 0$ or as $x\rta\infty$ (the regime we are considering will be clear from the context). We write $a = o_x^\infty(b)$ if $a = o_x(b^s)$ for every $s\in\BB R$.  
\vspace{6pt}

\noindent
For a graph $G$, we write $\mcl V(G)$ and $\mcl E(G)$, respectively, for the set of vertices and edges of $G$, respectively. We sometimes omit the parentheses and write $\mcl VG = \mcl V(G)$ and $\mcl EG = \mcl E(G)$. For $v\in\mcl V(G)$, we write $\op{deg}(v;G)$ for the degree of $v$ (i.e., the number of edges with $v$ as an endpoint). 
\vspace{6pt}

\subsection{Setup}
\label{sec-standard-setup}

In this subsection, we describe the setup we consider throughout most of the paper and introduce some relevant notation.  
Let $h$ be a random distribution on $\BB C$ whose $\gamma$-quantum measure $\mu_h$ is well-defined and has infinite total mass. 
We will most frequently consider the case when $h$ is the distribution corresponding to a $\gamma$-quantum cone, since this is the case for which the corresponding structure graph coincides with a mated-CRT map. However, we will also have occasion to consider a choice of $h$ which does not have a $\gamma$-log singularity at the origin---in particular, we will sometimes take $h$ to be either a whole-plane GFF or the distribution corresponding to a 0-quantum cone.

Let $\eta$ be a whole-plane space-filling SLE$_{\kappa'}$ sampled independently from $h$ and then parameterized by $\gamma$-quantum mass with respect to $h$.
For $\ep > 0$, we let $\mcl G^\ep$ be the graph with vertex set $\ep\BB Z  $, with two vertices $x_1,x_2\in \ep\BB Z $ connected by an edge if and only if the cells $\eta([x_1-\ep , x_1])$ and $\eta([x_2-\ep,x_2])$ share a non-trivial boundary arc. 

We abbreviate cells by
\eqb \label{eqn-cell-def}
H_x^\ep := \eta([x-\ep , x]) ,\quad \forall x \in \ep\BB Z  
\eqe 
and for $z\in\BB C$, we define the vertex
\eqb \label{eqn-pt-vertex}
x_z^\ep := \min\left\{ x\in\ep\BB Z :    z\in H_x^\ep \right\} ,
\eqe
so that $H_{x_z^\ep}^\ep$ is the (a.s.\ unique) structure graph cell containing $z$. 
 
For a set $D\subset \BB C$, we write $\mcl G^\ep(D)$ for the sub-graph of $\mcl G^\ep$ with vertex set
\eqb \label{eqn-structure-graph-domain}
\mcl V\mcl G^\ep(D) := \left\{ x_z^\ep : z\in D \right\} =  \left\{ x\in\ep\BB Z  : H_x^\ep \cap D\not=\emptyset\right\} 
\eqe 
with two vertices connected by an edge if and only if they are connected by an edge in $\mcl G^\ep$. 
See Figure~\ref{fig-structure-graph-restrict} for an illustration of the above definitions.

\begin{figure}[ht!]
\begin{center}
\includegraphics[scale=.8]{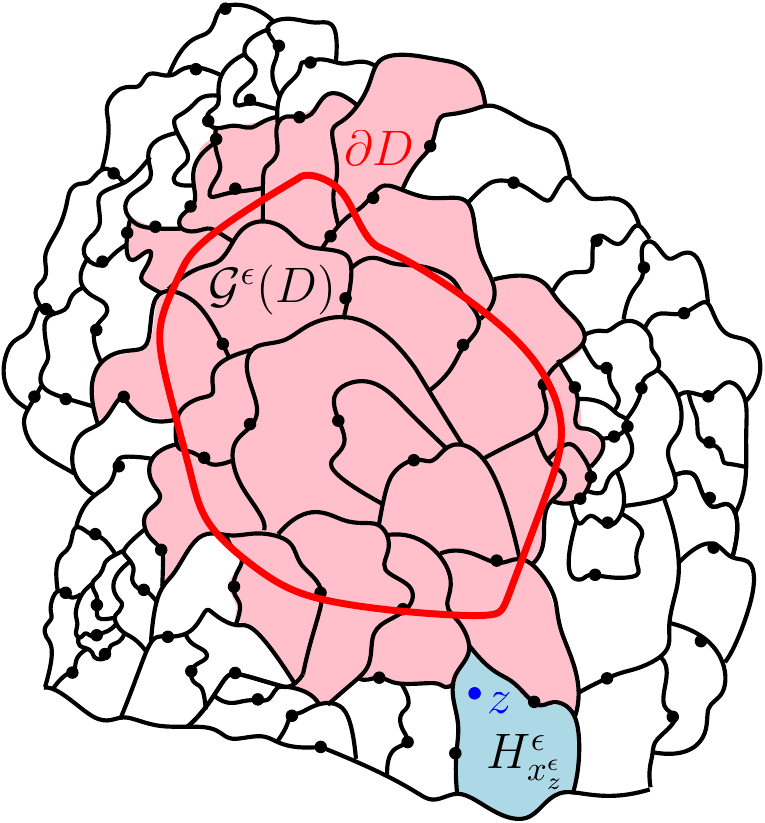} 
\caption{\label{fig-structure-graph-restrict} Illustration of the definitions in Section~\ref{sec-standard-setup}. A collection of cells of $\mcl G^\ep$ is shown with black boundaries and a domain $D$ is shown with red boundary. The pink cells are the those of the form $H_x^\ep$ for vertices $x  \in \mcl V\mcl G^\ep(D)$. Also shown is a point $z\in \BB C$ and the cell $H_{x_z^\ep}^\ep$ containing it (light blue). 
}
\end{center}
\end{figure}

\subsection{Main results}
\label{sec-main-results}

Suppose we are in the setting of Section~\ref{sec-standard-setup} with $h$ equal to the circle-average embedding of a $\gamma$-quantum cone (i.e., $h$ is the random distribution from Definition~\ref{def-quantum-cone} with $\alpha = \gamma$), so that the graphs $\{\mcl G^\ep\}_{\ep > 0}$ are the same in law as the $\ep$-mated CRT maps defined in Section~\ref{sec-overview}. 

We will study discrete harmonic functions on sub-maps of $\mcl G^\ep$ corresponding to domains in $\BB C$. 
We want to work at positive distance from $\bdy \BB D$ to avoid complications arising from the choice of normalization of the field,\footnote{In particular, the law of $h|_{\BB D}$ agrees in law with the corresponding restriction of the whole-plane GFF plus $-\gamma\log|\cdot|$, but this property does not hold outside of $\BB D$; see Section~\ref{sec-lqg-prelim}.} so we fix $\rho \in (0,1)$ and restrict attention to $B_{\rho}(0)$. Let $D\subset B_{\rho}(0)$ be an open set and let $f : \ol D\rta \BB C$ be a continuous function. 

Recall the sub-graph $\mcl G^\ep(D) \subset\mcl G^\ep$ from~\eqref{eqn-structure-graph-domain} and let $\frk f^\ep : \mcl V\mcl G^\ep( D)  \rta \BB R$ be the function such that 
\eqb \label{eqn-bdy-data}
\frk f^\ep(x) = \sup_{z\in H_x^\ep \cap \bdy D} f(z) ,\quad \forall x\in \mcl V\mcl G^\ep(\bdy D) 
\eqe 
and $\frk f^\ep$ is discrete harmonic on $\mcl V\mcl G^\ep( D) \setminus \mcl V\mcl G^\ep( \bdy D) $. 
The first main result of this paper shows that the discrete Dirichlet energy of $\frk f^\ep$ can be bounded above by a constant times the Dirichlet energy of $f$.  

\begin{defn} \label{def-discrete-dirichlet}
For a graph $G$ and a function $g : \mcl V(G) \rta \BB R$, we define its \emph{Dirichlet energy} to be the sum over unoriented edges
\eqbn
\op{Energy}(g; G) := \sum_{\{x,y\} \in \mcl E(G)} (g(x) - g(y))^2 ,
\eqen
with $n$-tuple edges counted $n$ times.  
\end{defn}

\begin{defn}  \label{def-continuum-dirichlet}
For a domain $D\subset \BB C$ and a function $f : D\rta \BB R$ whose gradient $\nabla f$ exists in the distributional sense, we define its \emph{Dirichlet energy} 
\eqbn
\op{Energy}(f ; D) := \int_D |\nabla f(z)|^2 \,dz .
\eqen 
\end{defn}
 
\begin{thm}[Dirichlet energy bound] \label{thm-dirichlet-harmonic0}
Suppose $f$ is continuously differentiable, the gradient $\nabla f$ is Lipschitz continuous, and $D$ has \emph{bounded convexity} in the sense that there exists $C = C(D) > 0$ such that any two points $z,w\in D$ can be joined by a path in $D$ of Euclidean length at most $C|z-w|$.
There are constants $\alpha    >0$ (depending only on $\gamma$) and $A   > 0$ (depending only on $D$, $\gamma$, and the Lipschitz constants for $f$ and $\nabla f$) such that with probability at least $1 - O_\ep(\ep^\alpha)$, the discrete and continuum Dirichlet energies of $\frk f^\ep$ and $f$ are related by
\eqb \label{eqn-dirichlet-harmonic0}
\op{Energy}\left( \frk f^\ep ; \mcl G^\ep(D) \right) \leq A \op{Energy}(f ; D)  .
\eqe 
\end{thm}

We will actually prove a more quantitative version of Theorem~\ref{thm-dirichlet-harmonic0} below (see Theorem~\ref{thm-dirichlet-harmonic}), which makes the dependence of $A$ more explicit.

One reason why bounds for Dirichlet energy are important is that one can express many quantities related to random walk on the graph --- such as the Green's function, effective resistances, and return probabilities --- in terms of the discrete Dirichlet energy of certain functions (see, e.g.,~\cite[Section 2]{lyons-peres}). 
These relationships together with Theorem~\ref{thm-dirichlet-harmonic0} lead to a lower bound for the Green's function of random walk on $\mcl G^\ep$ on the diagonal, or equivalently for the effective resistance to the boundary of a Euclidean ball (Theorem~\ref{thm-green-lower} just below). Further applications of our Dirichlet energy estimates will be explored in~\cite{gm-spec-dim}. 

For $n\in\BB N$ and $\ep > 0$, let $\op{Gr}_n^\ep(\cdot,\cdot)$ be the Green's function of $\mcl G^\ep$ at time $n$, i.e., $\op{Gr}_n^\ep(x,y)$ for vertices $x,y \in \mcl V\mcl G^\ep$  gives the (conditional given $\mcl G^\ep$) expected number of times that simple random walk on $\mcl G^\ep$ started from $x$ hits $y$ before time $n$.  
 
\begin{thm}[Green's function on the diagonal] \label{thm-green-lower} 
Fix $\rho \in (0,1)$ and let $\tau^\ep$ for $\ep > 0$ be the exit time of simple random walk on $\mcl G^\ep$ from $\mcl V\mcl G^\ep(B_{\rho}(0))$. There exists $\alpha >0$ (depending only on $\gamma$) and $A   > 0$ (depending only on $\rho$ and $\gamma$) such that 
\eqb \label{eqn-green-lower}
\BB P\left[ \frac{ \op{Gr}_{\tau^\ep}^\ep(0,0) }{ \op{deg}\left(0 ; \mcl G^\ep \right)}  \geq \frac{1}{A} \log \ep^{-1} \right] \geq 1 - O_\ep(\ep^\alpha) \quad \text{as $\ep\rta 0$}.
\eqe 
Furthermore, the simple random walk on $\mcl G^\ep$ is a.s.\ recurrent. 
\end{thm}
 
The recurrence of random walk on $\mcl G^\ep$ can also be deduced from the general recurrence criterion for random planar maps due to Gurel-Gurevich and Nachmias~\cite{gn-recurrence} (see Section~\ref{sec-mated-crt-map}), but our results give an independent proof. 
Note, however, that our results do \emph{not} give an independent proof of the recurrence of random walk on other planar maps, such as the UIPT. 
 
Our next main result gives a H\"older continuity bound for the functions $\frk f^\ep$ for $\ep  >0$ in terms of the Euclidean metric. 
 
\begin{thm}[H\"older continuity] \label{thm-cont0}
Suppose $D$ is simply connected and $f|_{\bdy D}$ is $\chi$-H\"older continuous (with respect to the ambient Euclidean metric) for some exponent $\chi  >0$. 
There are constants $\alpha  = \alpha(\gamma )  >0$, $\xi = \xi(\rho,\chi,\gamma)> 0$, and $A' = A'(f,D,\gamma) > 0$ such that with probability at least $1-O_\ep(\ep^\alpha)$, the discrete harmonic function $\frk f^\ep$ defined just below~\eqref{eqn-bdy-data} satisfies 
\allb \label{eqn-log-cont0} 
|\frk f^\ep(x) - \frk f^\ep(y)| 
 \leq A' (\ep \vee |\eta(x) - \eta(y)|)^\xi  ,\quad \forall x , y \in \mcl V\mcl G^\ep(D ) ,
\alle 
where $\eta$ is the space-filling SLE$_{\kappa'}$ as in Section~\ref{sec-standard-setup}. 
\end{thm}
 
As in the case of Theorem~\ref{thm-dirichlet-harmonic0}, we will prove a more quantitative version of Theorem~\ref{thm-cont0}; see Theorem~\ref{thm-cont}.  
The proof of this theorem proceeds by way of a ``uniform ellipticity" type estimate for simple random walk on $\mcl G^\ep$, which says that the walk has uniformly positive probability to stay close to a fixed path in $\BB C$ (Proposition~\ref{prop-pos-prob}). 
 
Theorem~\ref{thm-cont0} gives a polynomial bound for the rate at which the maximal length of an edge of the graph $\mcl G^\ep(D)$ under the so-called Tutte embedding with identity boundary data converges to 0 as $\ep\rta 0$ (since $\mcl G^\ep$ is a triangulation, this is equivalent to the analogous statement with faces in place of edges). Note that~\cite{gms-tutte} shows that the Tutte embedding with identity boundary data converges to the identity, but gives no quantitative bound on the maximal length of the embedded edges.   

To state this more precisely, let $\Phi_1^\ep$ be the function $\frk f^\ep$ from above with $f(z) = \re z$ and let $\Phi_2^\ep$ be defined analogously with $f(z) = \im z$. Then $\Phi^\ep := (\Phi_1^\ep , \Phi_2^\ep) : \mcl V\mcl G^\ep(D) \rta \BB R^2$ is discrete harmonic on the interior of $\mcl G^\ep(D)$ and approximates the map $x\mapsto \eta(x)$ on $\mcl V\mcl G^\ep(\bdy D)$. The function $\Phi^\ep$ is called the \emph{Tutte embedding of $\mcl G^\ep(D)$ with identity boundary data}. 

It is easy to see that the maximal size of the cells $H_x^\ep$ for $x\in \mcl V\mcl G^\ep(D)$ is at most some positive power of $\ep$ with probability tending to 1 as $\ep\rta 0$ (Lemma~\ref{lem-max-cell-diam}). Applying Theorem~\ref{thm-cont0} to each coordinate of $\Phi^\ep$ and considering vertices $x$ and $y$ which are connected by an edge in $\mcl G^\ep$ yields the following.

\begin{cor}[Maximal length of embedded edges] \label{cor-max-edge}
Define the Tutte embedding $\Phi^\ep$ with identity boundary data as above. If $D\subset B_{\rho}(0)$ is simply connected, then there exists $\xi' = \xi'(\rho,\gamma) > 0$ such that with probability tending to 1 as $\ep\rta 0$, 
\eqb \label{eqn-max-edge}
\max_{\{x,y\} \in \mcl E\mcl G^\ep(D)} | \Phi^\ep(x) - \Phi^\ep(y)| \leq O_\ep\left( \ep^{\xi'} \right)  
\eqe
\end{cor}

It is a major open problem to prove that the maximal length of the embedded edges of other types of random planar maps---e.g., uniform random planar maps or planar maps sampled with probability proportional to the number of spanning trees---under the Tutte embedding (or under other embeddings, like the circle packing~\cite{stephenson-circle-packing}) tends to 0 as the total number of vertices tends to 0. Indeed, this is believed to be a key obstacle to proving that such embedded maps converge to $\gamma$-LQG in various senses, as conjectured, e.g., in~\cite{shef-kpz,shef-zipper,dkrv-lqg-sphere,curien-glimpse}. 

Corollary~\ref{cor-max-edge} suggests a possible approach to proving that the maximal edge length for various additional types of embedded random planar maps, besides just the mated-CRT map, also tends to zero. The reason for this is that in many cases it is possible to transfer estimates from the mated-CRT map to estimates for other random planar maps modulo polylogarithmic multiplicative errors. So far, this has been done for graph distances~\cite{ghs-map-dist}, random walk speed~\cite{gm-spec-dim,gh-displacement}, and random walk return probabilities~\cite{gm-spec-dim}. However, we have not yet found a way to transfer modulus of continuity bounds for harmonic functions, which is what is needed to deduce an analog of Corollary~\ref{cor-max-edge} for other planar map models. 
 

\subsection{Outline}
\label{sec-outline}

Figure~\ref{fig-outline} shows a diagram of the logical connections between the main results related to this paper.
In Section~\ref{sec-prelim}, we will review some facts from the theory of SLE and LQG, prove that the law of the degree of a vertex of the mated-CRT map has an exponential tail (Lemma~\ref{lem-exp-tail}), and prove that the maximum diameter of the cells of $\mcl G^\ep$ which intersect a fixed Euclidean ball decays polynomially in $\ep$ (Lemma~\ref{lem-max-cell-diam}). 
We then state an estimate (Proposition~\ref{prop-area-diam-deg-gamma}) which says that if $D\subset \BB C$ and $f :D\rta \BB R$ is a sufficiently regular function, then except on an event of probability decaying polynomially in $\ep$, 
\eqb \label{eqn-area-diam-deg-int0}
\int_D f(z) \frac{\op{diam}(H_{x_z^\ep}^\ep)^2}{\op{area}(H_{x_z^\ep}^\ep)} \op{deg}(H_{x_z^\ep}^\ep) \,dz = O_\ep(1)
\eqe 
where here we recall that $H_{x_z^\ep}^\ep$ is the cell of $\mcl G^\ep$ containing $z$. The proof of this estimate is deferred to Section~\ref{sec-area-diam-deg}. Intuitively,~\eqref{eqn-area-diam-deg-int0} says that the measure which assigns mass $\frac{\op{diam}(H_{x_z^\ep}^\ep)^2}{\op{area}(H_{x_z^\ep}^\ep)} \op{deg}(H_{x_z^\ep}^\ep)$ to each $z\in\BB C$ is not too much different from Lebesgue measure, which in turn is a consequence of the fact that $\frac{\op{diam}(H_{x_z^\ep}^\ep)^2}{\op{area}(H_{x_z^\ep}^\ep)} \op{deg}(H_{x_z^\ep}^\ep)$ is of constant order for most $z\in D$.

In Section~\ref{sec-energy-estimate}, we assume the aforementioned estimate~\eqref{eqn-area-diam-deg-int0} and deduce our main results. We first prove in Section~\ref{sec-sum-to-int} an estimate to the effect that if $f : D\rta\BB R$ is as in~\eqref{eqn-area-diam-deg-int0}, then with high probability
\eqb  \label{eqn-diam-deg-sum0}
\sum_{x\in  \mcl V\mcl G^\ep(D)} f(\eta(x)) \op{diam}(H_x^\ep)^2 \op{deg}(H_x^\ep)  = O_\ep(1) ,
\eqe 
which follows from~\eqref{eqn-area-diam-deg-int0} by breaking up the integral in~\eqref{eqn-area-diam-deg-int0} into integrals over individual cells. The bound~\eqref{eqn-diam-deg-sum0} is used in Section~\ref{sec-dirichlet-mean} to prove an upper bound for the discrete Dirichlet energy of $x\mapsto f  (\eta(x))$, which in turn implies (a more precise version of) Theorem~\ref{thm-dirichlet-harmonic0} since discrete harmonic functions minimize Dirichlet energy. In Section~\ref{sec-recurrence}, we deduce Theorem~\ref{thm-green-lower} from this more general bound.
In Section~\ref{sec-pos-prob}, we use our Dirichlet energy bound to show that the simple random walk on $\mcl G^\ep$ has uniformly positive probability to stay close to a fixed Euclidean path, even if we condition on $\mcl G^\ep$. The basic idea is to first prove a lower bound for the probability of hitting the inner boundary of an annulus before the outer boundary (using Dirichlet energy estimates and the Cauchy-Schwarz inequality) then cover a path by such annuli. 
 In Section~\ref{sec-cont}, we use the result of Section~\ref{sec-pos-prob} to prove a H\"older continuity estimate for harmonic functions on $\mcl G^\ep(D)$ which includes Theorem~\ref{thm-cont0} as a special case. 
 
In Section~\ref{sec-area-diam-deg}, we prove~\eqref{eqn-area-diam-deg-int0}, taking the moment bounds for the squared diameter over area and degree of the cells of $\mcl G^\ep$ from~\cite[Theorem 4.1]{gms-tutte} as a starting point. Heuristically, these moment bounds say that cells are not too likely to be ``long and skinny" and are not too likely to have large degree. The proof is outlined in Section~\ref{sec-lln-outline}, and is based on using long-range independence properties for the GFF to bound the variance of the integral appearing in~\eqref{eqn-area-diam-deg-int0}. 

Appendix~\ref{sec-gff-abs-cont} contains some basic estimates for the GFF which are needed in our proofs. Appendix~\ref{sec-index} contains an index of notation.

\begin{figure}[ht!]
 \begin{center}
\includegraphics[scale=.85]{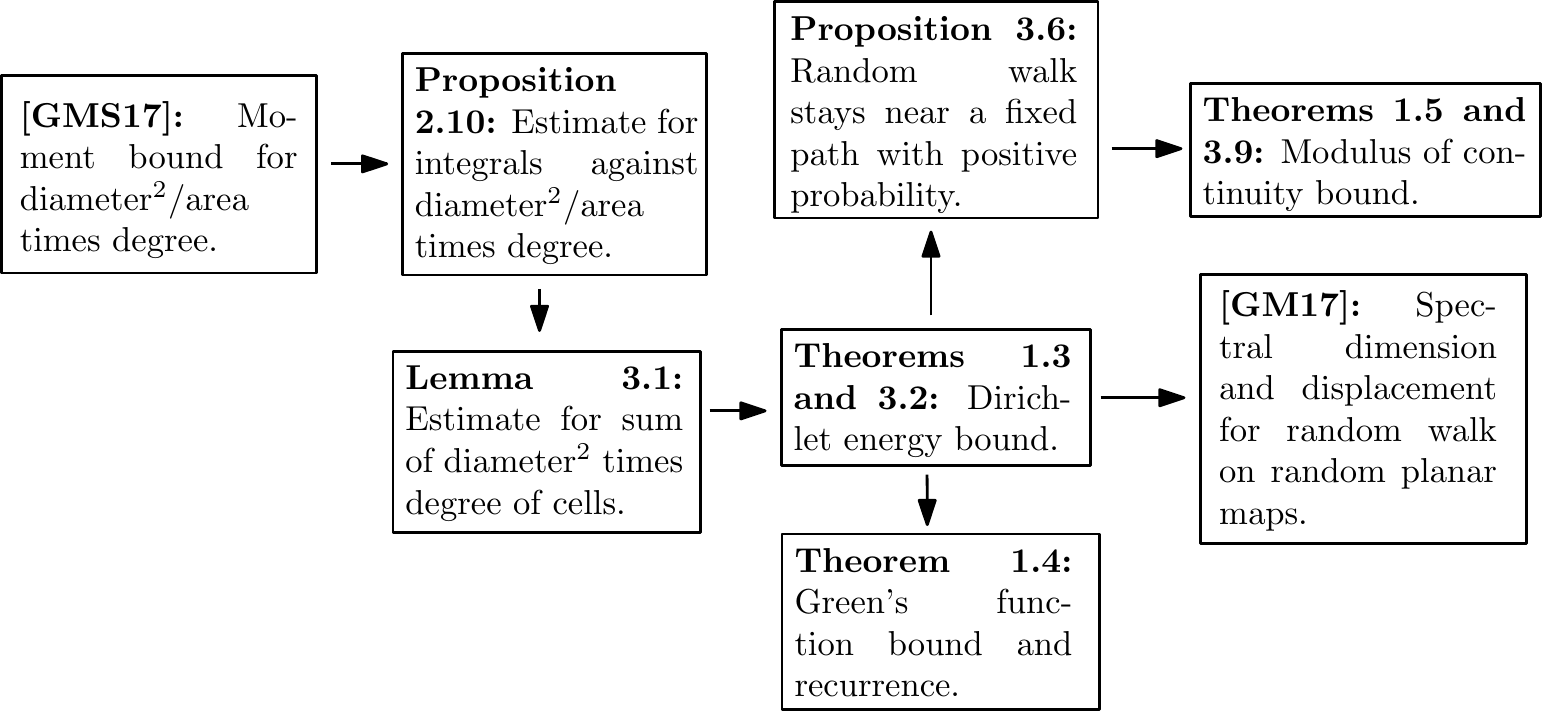} 
\caption{Schematic illustration of how various results related to this paper fit together.
}\label{fig-outline}
\end{center}
\end{figure}

\medskip
\noindent\textbf{Acknowledgements.}
We thank two anonymous referees for helpful comments on an earlier version of this article.
We thank the Mathematical Research Institute of Oberwolfach for its hospitality during a workshop where part of this work was completed.  E.G.\ was partially funded by NSF grant DMS-1209044.  S.S.\ was partially supported by NSF grants DMS-1712862 and DMS-1209044 and a Simons Fellowship with award number 306120.

\section{Preliminaries}
\label{sec-prelim}

\subsection{Background on GFF, LQG, and SLE}
\label{sec-sle-lqg-prelim}

Throughout this paper, we always fix an LQG parameter $\gamma \in (0,2)$ and a corresponding SLE parameter $\kappa'  = 16/\gamma^2  > 4$. 
Here we provide some background on the main continuum objects involved in this paper, namely the Gaussian free field, 
Liouville quantum gravity, and space-filling SLE$_{\kappa'}$. A reader who is already familiar with these objects can safely skip this subsection.

\subsubsection{The Gaussian free field}
\label{sec-gff-prelim}
 
Here we give a brief review of the definition of the zero-boundary and whole-plane Gaussian free fields. We refer the reader to~\cite{shef-gff} and the introductory sections of~\cite{ss-contour,ig1,ig4} for more detailed expositions. 

For an open domain $D\subset \BB C$ with harmonically non-trivial boundary (i.e., Brownian motion started from a point in $D$ a.s.\ hits $\bdy D$), we define $\mcl H(D)$ be the Hilbert space completion of the set of smooth, compactly supported functions on $D$ with respect to the \emph{Dirichlet inner product},
\eqb \label{eqn-dirichlet}
(\phi,\psi)_\nabla = \frac{1}{2\pi} \int_D \nabla \phi(z) \cdot \nabla \psi(z) \,dz .
\eqe
In the case when $D= \BB C$, constant functions $c$ satisfy $(c,c)_\nabla = 0$, so to get a positive definite norm in this case we instead take $\mcl H(\BB C)$ to be the Hilbert space completion of the set of smooth, compactly supported functions $\phi$ on $\BB C$ with $\int_{\BB C} \phi(z) \,dz = 0$, with respect to the same inner product~\eqref{eqn-dirichlet}.  
 
The \emph{(zero-boundary) Gaussian free field} on $D$ is defined by the formal sum
\eqb \label{eqn-gff-sum}
h = \sum_{j=1}^\infty X_j \phi_j 
\eqe
where the $X_j$'s are i.i.d.\ standard Gaussian random variables and the $\phi_j$'s are an orthonormal basis for $\mcl H(D)$. The sum~\eqref{eqn-gff-sum} does not converge pointwise, but for each fixed $\phi \in \mcl H(D)$, the formal inner product $(h ,\phi)_\nabla := \sum_{j=1}^\infty X_j (\phi_j , \phi)_\nabla $ is a centered Gaussian random variable and these random variables have covariances $\BB E [(h,\phi)_\nabla (h,\psi)_\nabla] = (\phi,\psi)_\nabla$. In the case when $D \not=\BB C$ and $D$ has harmonically non-trivial boundary, one can use integration by parts to define the ordinary $L^2$ inner products $(h,\phi) := -2\pi (h,\Delta^{-1}\phi)_\nabla$, where $\Delta^{-1}$ is the inverse Laplacian with zero boundary conditions, whenever $\Delta^{-1} \phi \in \mcl H(D)$. This allows one to define the GFF as a distribution (generalized function). See~\cite[Section 2]{shef-gff} for some discussion about precisely which spaces of distributions the GFF takes values in.

For $z\in D$ and $r > 0$ such that $B_r(z) \subset D$, we write $h_r(z) $ for the circle average of $h$ over $\partial B_r(z)$, as in~\cite[Section 3.1]{shef-kpz}. 
Following~\cite[Section 3.1]{shef-kpz}, to define this circle average precisely, one can let $\xi_r^z(w) := -\log\max\{r , |w-z|\}$, so that $-\Delta\xi_r^z$ (defined in the distributional sense) is $2\pi$ times the uniform measure on $\bdy B_r(z)$. One then defines $h_r(z)$ to be the Dirichlet inner product $ (h , \xi_r^z)_\nabla$.

In the case when $D=\BB C$, one can similarly define $(h ,\phi) := -2\pi (h ,\Delta^{-1}\phi)_\nabla$ where $\Delta^{-1}$ is the inverse Laplacian normalized so that $\int_{\BB C} \Delta^{-1} \phi(z) \, dz = 0$. With this definition, one has $(h+c , \phi) = (h ,\phi) + (c,\phi) = (h,\phi)$ for each $\phi \in \mcl H(\BB C)$, so the whole-plane GFF is only defined as a distribution modulo a global additive constant; that is, $h$ can be viewed as an equivalence class of distributions under the equivalence relation whereby two distributions are equivalent if their difference is a constant.  
 We will typically fix the additive constant for the GFF (i.e., choose a particular equivalence class representative) by requiring that the circle average $h_1(0)$ over $\bdy\BB D$ is zero. That is, we consider the field $h - h_1(0)$, which is well-defined not just modulo additive constant.  
The law of the whole-plane GFF is scale and translation invariant modulo additive constant, which means that for $z\in\BB C$ and $r>0$ one has $h(r\cdot +z) - h_r(z) \eqD h - h_1(0)$. 

If $h$ is a GFF on $D$, we can define the restriction of $h$ to an open set $U\subset D$ as the restriction of the distributional pairing $\phi \mapsto (h,\phi)$ to test functions $\phi$ which are supported on $V$.
It does not make literal sense to restrict the GFF to a closed set $K\subset D$, but the $\sigma$-algebra generated by $h|_K$ can be defined as $\bigcap_{\ep > 0} \sigma(h|_{B_\ep(K)})$, where $B_\ep(K)$ is the Euclidean $\ep$-neighborhood of $K$. 
Hence it makes sense to speak of, e.g., ``conditioning on $h|_K$".
 
The zero-boundary GFF on $D$ possesses the following Markov property (see, e.g.,~\cite[Section 2.6]{shef-gff}). Let $U \subset D$ be a sub-domain with harmonically non-trivial boundary.
Then we can write $h = \frk h + \rng h$, where $\frk h$ is a random distribution on $D$ which is harmonic on $U$ and is determined by $h|_{D\setminus U}$; and $\rng h$ is a zero-boundary GFF on $U$ which is independent from $h|_{D\setminus U}$. 
The restrictions of these distributions to $U$ are called the \emph{harmonic part} and \emph{zero-boundary part} of $h|_U$, respectively.

In the whole-plane case, one has a slightly more complicated Markov property due to the need to fix the additive constant.  
We state two versions of this Markov property, one with the field viewed modulo additive constant and one with the additive constant fixed.
The first version is a re-statement of~\cite[Proposition 2.8]{ig4}.

\begin{lem} \label{lem-whole-plane-markov'}
Let $h'$ be a whole-plane GFF viewed modulo additive constant. 
For each open set $U\subset\BB C$ with harmonically non-trivial boundary, we have the decomposition 
\eqb
h'  =   \frk h' + \rng h' ,
\eqe
where $\frk h'$ is a random distribution viewed modulo additive constant which is harmonic on $U$ and is determined by $h'|_{\BB C\setminus U}$, viewed modulo additive constant;
 and $\rng h'$ is a zero-boundary GFF on $\BB C\setminus U$ which is determined by the equivalence class of $h'|_{\BB C\setminus U}$ modulo additive constant.
\end{lem} 

We refer to the distributions $\frk h'|_U$ and $\rng h'|_U$ as the \emph{harmonic part} and \emph{zero-boundary part} of $h'|_U$, respectively.

Now suppose we want to fix the additive constant for the field so that $h_1(0) = 0$, i.e., we want to consider $h' - h'_1(0)$. 
In the setting of Lemma~\ref{lem-whole-plane-markov'}, the distributions $\frk h' - h_1'(0)$ and $\rng h'$ are \emph{not} independent if $\bdy\BB D\cap U\not=\emptyset$ since $h_1'(0)$ depends on $\rng h'$. 
Nevertheless, it turns out that a slight modification of these distributions are independent.

\begin{lem} \label{lem-whole-plane-markov}
Let $h$ be a whole-plane GFF with the additive constant chosen so that $h_1(0) = 0$. 
For each open set $U\subset\BB C$ with harmonically non-trivial boundary, we have the decomposition
\eqb
h  =   \frk h + \rng h
\eqe
where $\frk h$ is a random distribution which is harmonic on $U$ and is determined by $h|_{\BB C\setminus U}$ and $\rng h$ is independent from $\frk h$ and has the law of a zero-boundary GFF on $U$ minus its average over $\bdy \BB D \cap U$. If $U$ is disjoint from $\bdy \BB D$, then $\rng h$ is a zero-boundary GFF and is independent from $h|_{\BB C\setminus U}$. 
\end{lem} 
\begin{proof}
Let $h'$ be a whole-plane GFF viewed modulo additive constant, so that $h = h' - h_1'(0)$. 
Write $h' = \frk h'  + \rng h'$ as in Lemma~\ref{lem-whole-plane-markov'}.  
Let $\rng h_1'(0)$ be the average of $\rng h'$ over $\bdy \BB D$ (equivalently, over $\bdy\BB D\cap U$).
Also let $\frk h'_1(0) = h_1'(0) - \rng h_1'(0)$ be the average of $\frk h'$ over $\bdy\BB D$. 
We define
\eqb
\frk h := \frk h' - \frk h_1'(0) \quad \text{and} \quad \rng h := \rng h' - \rng h_1'(0) .
\eqe
Then $\frk h$ is a harmonic function in $U$ and is well-defined (not just modulo additive constant) and $\rng h$ is a zero-boundary GFF in $U$ minus its average over $\bdy \BB D \cap U$.
By definition, we have $\frk h + \rng h = h' - h_1'(0) = h$.  Furthermore, $\frk h$ (resp.\ $\rng h$) is determined by $\frk h'$ (resp.\ $\rng h'$), so $\frk h$ and $\rng h$ are independent. Since $\frk h'$ is determined by $h'|_{\BB C\setminus U}$, viewed modulo additive constant, it follows that $\frk h'$ is determined by $h|_{\BB C\setminus U}$. Hence $\frk h$ is determined by $h|_{\BB C\setminus U}$.  

If $\bdy\BB D$ is disjoint from $U$, then $\rng h'_1(0) = 0$ so $\frk h|_{\BB C\setminus \ol U} = h|_{\BB C\setminus \ol U}$. 
This implies that $\rng h$ is a zero-boundary GFF and $\rng h$ is independent from $ h|_{\BB C\setminus \ol U}$. 
\end{proof}

\subsubsection{Liouville quantum gravity}
\label{sec-lqg-prelim}

Fix $\gamma \in (0,2)$. Following~\cite{shef-kpz,shef-zipper,wedges}, we define a \emph{$\gamma$-Liouville quantum gravity (LQG) surface} to be an equivalence class of pairs $(D,h)$, where $D\subset \BB C$ is an open set and $h$ is a distribution on $D$ (which will always be taken to be a realization of a random distribution which locally looks like the Gaussian free field), with two such pairs $(D,h)$ and $(\wt D , \wt h)$ declared to be equivalent if there is a conformal map $f : \wt D \rta D$ such that
\eqb \label{eqn-lqg-coord}
\wt h = h\circ f + Q\log |f'| \quad \text{for} \quad Q = \frac{2}{\gamma}  + \frac{\gamma}{2} .
\eqe
One can similarly define a $\gamma$-LQG surface with $k\in\BB N$ marked points. This is an equivalence class of $k+2$-tuples $(D,h,x_1,\dots,x_k)$ with the equivalence relation defined as in~\eqref{eqn-lqg-coord} except that the map $f$ is required to map the marked points of one surface to the corresponding marked points of the other.
We call different choices of the distribution $h$ corresponding to the same LQG surface different \emph{embeddings} of the surface. 

If $h$ is a random distribution on $D$ which can be coupled with a GFF on $D$ in such a way that their difference is a.s.\ a continuous function, then one can define the \emph{$\gamma$-LQG area measure} $\mu_h$ on $D$, which is defined to the a.s.\ limit
\eqbn
\mu_h = \lim_{\ep \rta 0} \ep^{\gamma^2/2} e^{\gamma h_\ep(z)} \, dz
\eqen
in the Prokhorov distance (or local Prokhorov distance, if $D$ is unbounded) as $\ep\rta 0$ along powers of 2~\cite{shef-kpz}. Here $h_\ep(z)$ is the circle-average of $h$ over $\bdy B_\ep(z)$, as defined in~\cite[Section~3.1]{shef-kpz} and discussed in Section~\ref{sec-prelim}. One can similarly define a boundary length measure $\nu_h$ on certain curves in $D$, including $\bdy D$~\cite{shef-kpz} and SLE$_\kappa$ type curves for $\kappa =\gamma^2$ which are independent from $h$~\cite{shef-zipper}. If $h$ and $\wt h$ are related by a conformal map as in~\eqref{eqn-lqg-coord}, then $f_* \mu_{\wt h} = \mu_h$ and $f_* \nu_{\wt h} = \nu_h$. 
Hence $\mu_h$ and $\nu_h$ can be viewed as measures on the LQG surface $(D,h)$. We note that there is a more general theory of regularized measures of this type, called \emph{Gaussian multiplicative chaos} which originates in work of Kahane~\cite{kahane}. See~\cite{rhodes-vargas-review,berestycki-gmt-elementary} for surveys of this theory. 
   
In this paper, we will be interested in two different types of $\gamma$-LQG surface. The first and most basic type of LQG surface we consider is the one where $h$ is a whole-plane GFF, as in Section~\ref{sec-gff-prelim}. We will typically fix the additive constant for the whole-plane GFF by requiring that the circle average over $\bdy\BB D$ is 0. 

The other type of $\gamma$-LQG surface with the topology of the plane which we will be interested in is the \emph{$\alpha$-quantum cone} for $\alpha \in (-\infty,Q)$, which is a doubly marked LQG surface $(\BB C ,h , 0, \infty)$ introduced in~\cite[Definition~4.10]{wedges}.  
Roughly speaking, the $\alpha$-quantum cone is obtained by starting with a whole-plane GFF plus $\alpha\log (1/|\cdot|)$ then ``zooming in" near the origin and re-scaling~\cite[Proposition~4.13(ii) and Lemma~A.10]{wedges}. 

We will not need the precise definition of the $\alpha$-quantum cone in this paper, but we recall it here for completeness.
Recall the Hilbert space $\mcl H(\BB C)$ used in the definition of the whole-plane GFF. 
Let $\mcl H^0(\BB C)$ (resp.\ $\mcl H^\dagger(\BB C)$) be the subspace of $\mcl H(\BB C)$ consisting of functions which are constant (resp.\ have mean zero) on each circle $\bdy B_r(0)$ for $r > 0$. By~\cite[Lemma 4.9]{wedges}, $\mcl H(\BB C)$ is the orthogonal direct sum of $\mcl H^0(\BB C)$ and $\mcl H^\dagger(\BB C)$. 

\begin{defn}[Quantum cone] \label{def-quantum-cone}
For $\alpha < Q$, the $\alpha$-quantum cone is the LQG surface $(\BB C , h , 0 , \infty)$ with the distribution $h$ defined as follows.
Let $B$ be a standard linear Brownian motion and let $\wh B$ be a standard linear Brownian motion conditioned so that $\wh B_t + (Q-\alpha) t > 0$ for all $t >0$. 
Let $A_t = B_t - \alpha t$ for $t\geq 0$ and let $A_t = \wh B_{-t} + \alpha t$ for $t < 0$. 
Then the projection of $h$ onto $\mcl H^0(\BB C)$ takes the constant value $A_t$ on each circle $B_{e^{-t}}(0)$. 
The projection of $h$ onto $\mcl H^\dagger(\BB C)$ is independent from the projection onto $\mcl H^0(\BB C)$ and agrees in law with the corresponding projection of a whole-plane GFF.
\end{defn}

We will typically be interested in quantum cones with $\alpha = 0$ or $\alpha = \gamma$. The case $\alpha =\gamma$ is special since a $\gamma$-LQG surface has a $\gamma$-log singularity at a typical point sampled from its $\gamma$-LQG measure (see, e.g.,~\cite[Section 3.3]{shef-kpz}), so the $\gamma$-quantum cone describes the local behavior of such a surface near a quantum typical point. The $\gamma$-quantum cone is also the type of LQG surface appearing in the embedding of the mated-CRT map. Similarly, the 0-quantum cone describes the behavior of a $\gamma$-LQG surface near a Lebesgue typical point. 
 
By the definition of an LQG surface, one can get another distribution describing the $\gamma$-quantum cone by replacing $h$ by $h(r\cdot) + Q\log r$ for some $r > 0$.
But, we will almost always consider the particular choice of $h$ appearing in Definition~\ref{def-quantum-cone}, which satisfies $\sup\{r > 0 : h_r(0) +Q \log r = 0\} = 1$. 
This choice of $h$ is called the \emph{circle average embedding}.  
A useful property of the circle average embedding (which is essentially immediate from~\cite[Definition 4.10]{wedges}) is that $h|_{\BB D}$ agrees in law with the corresponding restriction of a whole-plane GFF plus $-\alpha \log |\cdot|$, normalized so that its circle average over $\bdy \BB D$ is 0.

\subsubsection{Space-filling SLE$_{\kappa'}$}
\label{sec-wpsf-prelim}

The Schramm-Loewner evolution (SLE$_\kappa$) for $\kappa > 0$ is a one-parameter family of random fractal curves originally defined by Schramm in~\cite{schramm0}. 
SLE$_\kappa$ curves are simple for $\kappa \in (0,4]$, self-touching, but not space-filling or self-crossing, for $\kappa \in (4,8)$, and space-filling (but still not self-crossing) for $\kappa \geq 8$~\cite{schramm-sle}. 
One can consider SLE$_\kappa$  curves between two marked boundary points of a simply connected domain (chordal), from a boundary point to an interior point (radial), or between two points in $\BB C\cup\{\infty\}$ (whole-plane). We refer to~\cite{lawler-book} or~\cite{werner-notes} for an introduction to SLE. 
We will occasionally make reference to whole-plane SLE$_\kappa(\rho)$, a variant of whole-plane SLE$_\kappa$ where one keeps track of an extra marked ``force point" which is defined in~\cite[Section 2.1]{ig4}. However, we will not need many of its properties so we will not provide a detailed definition here.

Space-filling SLE$_{\kappa'}$ is a variant of SLE$_{\kappa'}$ for $\kappa ' > 4$ which was originally defined in~\cite[Section~1.2.3]{ig4} (see also~\cite[Section~1.4.1]{wedges} for the whole-plane case). Here we will review the construction of whole-plane space-filling SLE$_{\kappa'}$ from~$\infty$ to~$\infty$, which is the only version we will use in this paper. 

The basic idea of the construction is that, by SLE duality~\cite{zhan-duality1,zhan-duality2,dubedat-duality,ig1,ig4}, the outer boundary of an ordinary SLE$_{\kappa'}$ curve stopped at any given time is a union of SLE$_\kappa$-type curves for $\kappa = 16/\kappa' \in (0,4)$. It is therefore natural to try to construct a space-filling SLE$_{\kappa'}$-type curve by specifying its outer boundary at each fixed time. 
To construct the needed boundary curves, we will use the theory of imaginary geometry, which allows us to couple many different SLE$_\kappa$ curves with a common GFF. 

Let $\chi^{\op{IG}} := 2/\sqrt\kappa -\sqrt\kappa/2$. 
Following~\cite[Section 2.2]{ig4}, we define a \emph{whole-plane GFF viewed modulo a global additive multiple of $2\pi \chi^{\op{IG}}$} to be a random equivalence class of distributions obtained as follows. First, sample $h^{\op{IG}}$ from the law of the whole-plane GFF with the additive constant chosen so that $h^{\op{IG}}_1(0) = 0$.
Then, consider the equivalence class of $h^{\op{IG}}$ w.r.t.\ the equivalence relation whereby $h_1\sim h_2$ if and only if $h_1-h_2$ is a constant in $2\pi\chi^{\op{IG}}\BB Z$. 
Here, IG stands for ``Imaginary Geometry" and is used to distinguish the field $h^{\op{IG}}$ from the field $h$ corresponding to an LQG surface).  

Let $h^{\op{IG}}$ be a whole-plane GFF viewed modulo a global additive multiple of $2\pi\chi^{\op{IG}}$.
By~\cite[Theorem~1.1]{ig4}, for each fixed $z\in \BB C$ and $\theta \in (0,2\pi)$, one can define the \emph{flow line} of $h^{\op{IG}}$ started from $z$ with angle $\theta$, which is a whole-plane SLE$_\kappa(2-\kappa)$ curve from $z$ to $\infty$ coupled with $h^{\op{IG}}$, where $\kappa = 16/\kappa' \in (0,4)$ is the dual SLE parameter. Whole-plane SLE$_\kappa(2-\kappa)$ is a variant of SLE$_\kappa$ which is defined rigorously in~\cite[Section~2.1]{ig4}. For our purposes we will only need the flow lines started from points $z\in\BB Q^2$ with angles $\pi/2$ and $-\pi/2$, which we denote by $\eta_z^L$ and $\eta_z^R$, respectively (the $L$ and $R$ stand for ``left" and ``right", for reasons which will become apparent momentarily). 

For distinct $z,w \in \BB Q^2$, the flow lines $\eta_z^L$ and $\eta_w^L$ a.s.\ merge upon intersecting, and similarly with $R$ in place of $L$. The two flow lines $\eta_z^L$ and $\eta_z^R$ started at the same point a.s.\ do not cross, but these flow lines bounce off each other without crossing if and only if $\kappa' \in (4,8)$, equivalently $\kappa  \in (2,4)$~\cite[Theorem~1.7]{ig4}. 

We define a total order on $\BB Q^2$ by declaring that $z$ comes before $w$ if and only if $w$ lies in a connected component of $\BB C\setminus (\eta_z^L\cup \eta_z^R)$ which lies to the right of $\eta_z^L$ (equivalently, to the left of $\eta_z^R$).  The whole-plane analog of~\cite[Theorem~4.12]{ig4} (which can be deduced from the chordal case; see~\cite[Footnote 4]{wedges}) shows that there is a well-defined continuous curve $\eta : \BB R\rta \BB C$ which traces the points of $\BB Q^2$ in the above order, is such that $\eta^{-1}(\BB Q^2)$ is a dense set of times, and is continuous when parameterized by Lebesgue measure, i.e., in such a way that $\op{area}(\eta([a,b])) =b-a$ whenever $a < b$. The curve $\eta$ is defined to be the \emph{whole-plane space-filling SLE$_{\kappa'}$ from $\infty$ to $\infty$} associated with $h^{\op{IG}}$. 
 
The definition of $\eta$ implies that for each $z\in\BB C$, it is a.s.\ the case that the left and right boundaries of $\eta$ stopped when it first hits $z$ are equal to the flow lines $\eta_z^L$ and $\eta_z^R$ (which can be defined for a.e.\ $z\in\BB C$ simultaneously as the limits of the curves $\eta_w^L$ and $\eta_w^R$ as $\BB Q^2 \ni w \rta z$ w.r.t., e.g., the local Hausdorff distance). 
See Figure~\ref{fig-peanosphere-bm-def}.
The topology of $\eta$ is rather simple when $\kappa' \geq 8$. In this case, the left/right boundary curves $\eta_z^L$ and $\eta_z^R$ do not bounce off each other, so for $a < b$ the set $\eta([a,b])$ has the topology of a disk. In the case when $\kappa' \in (4,8)$, the curves $\eta_z^L$ and $\eta_z^R$ intersect in an uncountable fractal set and for $a<b$ the interior of the set $\eta([a,b])$ a.s.\ has countably many connected components, each of which has the topology of a disk. 

It is shown in~\cite[Theorem~1.16]{ig4} that for chordal space-filling SLE, the curve~$\eta$ is a.s.\ determined by~$h^{\op{IG}}$. The analogous statement in the whole-plane case can be proven using the same argument or deduced from the chordal case and~\cite[Footnote 4]{wedges}. We will need the following refined version of this statement. 

\begin{lem} \label{lem-wpsf-determined}
Let $D\subset \BB C$ and let $h^{\op{IG}}$ and $\eta$ be as above. Assume that $\eta$ is parameterized by Lebesgue measure. 
Let $U \subset D$ be an open set and for $z\in U\cap \BB Q^2$, let $T_z$ (resp.\ $S_z$) be the last time $\eta$ enters $U$ before hitting $z$ (resp.\ the first time $z$ exits $U$ after hitting $z$). Then for each $\delta > 0$, $h^{\op{IG}}|_{B_\delta(U)}$ a.s.\ determines the collection of curve segments
\eqb \label{eqn-wpsf-segments}
\left( \eta(\cdot + T_z)|_{[0 , S_z - T_z]} \right)_{z\in U\cap\BB Q^2} .
\eqe 
Here $B_\delta(U)$ is the Euclidean $\delta$-neighborhood of $U$, as in Section~\ref{sec-basic-notation}.
\end{lem}
 

\begin{proof}
It is shown in~\cite[Theorem~1.2]{ig4} that the left/right boundary curves $\eta_z^L$ and $\eta_z^R$ for $z\in\BB Q^2$ are a.s.\ determined by $h^{\op{IG}}$. 
For $z\in U\cap\BB Q^2$, let $\tau_z^L$ (resp.\ $\tau_z^R$) be the exit time of $\eta_z^L$ (resp.\ $\eta_z^R$) from $U$. 
By~\cite[Theorem~1.1]{ig4}, each of the sets $\eta_z^L([0,\tau_z^L])$ and $\eta_z^R([0,\tau_z^R])$ is a local set for $h^{\op{IG}}$ in the sense of~\cite[Lemma~3.9]{ss-contour}, so is a.s.\ conditionally independent from $h^{\op{IG}}|_{\BB C\setminus B_\delta(U)}$ given $h^{\op{IG}}|_{B_\delta(U)}$. Each of these sets is a.s.\ determined by $h^{\op{IG}}$, so is a.s.\ determined by $h^{\op{IG}}|_{B_\delta(U)}$. It is clear from the definition of space-filling SLE$_{\kappa'}$ given above that the curve segments~\eqref{eqn-wpsf-segments} are a.s.\ determined by $\eta_z^L|_{[0,\tau_z^L]}$ and $\eta_z^R|_{[0,\tau_z^R]}$ for $z\in U\cap\BB Q^2$. 
\end{proof}

\subsection{The degree of the root vertex has an exponential tail}
\label{sec-mated-crt-map}

Most of the results in this paper make use of the embedding of the mated-CRT maps which comes from SLE-decorated LQG (see Section~\ref{sec-peanosphere}).  However, the following result is proved directly from the ``Brownian motion" definition of the mated-CRT maps in~\eqref{eqn-inf-adjacency}, and does not rely on this embedding.

\begin{lem} \label{lem-exp-tail}
Let $\gamma \in (0,2)$ and let $\mcl G^\ep$ for $\ep > 0$ be a mated-CRT map. There are constants $c_0,c_1 > 0$, depending only on $\gamma$, such that for $n\in\BB N$, $\ep > 0$, and $x\in\ep\BB Z  $, 
\eqbn
\BB P\left[ \op{deg}\left( x ; \mcl G^\ep \right)  > n  \right] \leq c_0 e^{-c_1 n} .
\eqen
\end{lem}
\begin{proof}
By Brownian scaling and translation invariance, the law of the pointed graph $(\mcl G^\ep , x)$ does not depend on $\ep$ or $x$, so we can assume without loss of generality that $\ep  =1$ and $x = 0$.
By~\eqref{eqn-inf-adjacency}, the time reversal symmetry of $(L,R)$, and the fact that $(L,R) \eqD (R,L)$, it suffices to show that there exists constants $c_0,c_1 > 0$ as in the statement of the lemma such that with
\eqbn
N :=  \#  \left\{ y\in \BB N   :  \left( \inf_{t\in [-1,0]} L_t \right) \vee \left( \inf_{t\in [y-1,y]} L_t \right) < \inf_{t \in [0,y-1]} L_t \right\} ,
\eqen
we have $\BB P[N > n] \leq c_0 e^{-c_1 n}$. This follows from a straightforward Brownian motion argument based on the fact that for each stopping time $\tau$ for $L$ with $L_\tau \leq L_0 = 0$, it holds with positive conditional probability given $L|_{(-\infty,\tau]}$ that $\inf_{t\in [\tau , \tau + 1]} L_t \leq \inf_{t\in [-1,0]} L_t$; along with the Gaussian tail bound for $\inf_{t\in [-1,0]} L_t$.
\end{proof} 

From Lemma~\ref{lem-exp-tail} and a union bound, we get the following upper bound for the maximal degree of the cells of $\mcl G^\ep$ which intersect a specified Euclidean ball. 
  
\begin{lem} \label{lem-max-deg-nonlocal}
Suppose we are in the setting of Section~\ref{sec-standard-setup} with $h$ equal to the circle-average embedding of a $\gamma$-quantum cone into $(\BB C , 0, \infty)$. 
If we define the vertex $x_z^\ep \in \ep \BB Z  $ as in~\eqref{eqn-pt-vertex}, then for each $\zeta > 0$,
\eqb \label{eqn-max-deg-nonlocal}
\BB P\left[ \max_{z\in\BB D} \op{deg}\left( x_z^\ep ; \mcl G^\ep \right)  \leq (\log \ep^{-1})^{1+\zeta} \right] \geq  1 -  o_\ep^\infty(\ep)  .
\eqe
\end{lem}
\begin{proof}
By standard SLE/LQG estimates (see, e.g.,~\cite[Proposition 6.2]{hs-euclidean}), for $M>0$ we have $\BB D \subset \eta([-\ep^{-M} , \ep^{-M}])$ except on an event of probability decaying like some positive power of $\ep^M$ (the power depends only on $\gamma$). By Lemma~\ref{lem-exp-tail} and a union bound, it holds with probability $1-o_\ep^\infty(\ep)$ that $\op{deg}(x ; \mcl G^\ep) \leq (\log \ep^{-1})^{1+\zeta}$ for each $x\in [-\ep^{-M} , \ep^{-M}]_{\ep\BB Z }$. Combining the above estimates and sending $M\rta\infty$ shows that~\eqref{eqn-max-deg-nonlocal} holds.
\end{proof}

In light of Lemma~\ref{lem-exp-tail}, we can deduce the recurrence of the simple random walk on $\mcl G^\ep$ from the results of~\cite{gn-recurrence} (we will give an independent proof in Section~\ref{sec-recurrence}).
Indeed, by~\cite[Theorem~1.1]{gn-recurrence}, the random walk on an infinite rooted random planar map $(G,\BB v)$ is recurrent provided the law of the degree of the root vertex $\BB v$ has an exponential tail and $(G,\BB v)$ is the distributional limit of finite rooted planar maps in the local (Benjamini-Schramm) topology~\cite{benjamini-schramm-topology}. The first condition for $(\mcl G^\ep , 0)$ follows from Lemma~\ref{lem-exp-tail}. To obtain the second condition, we observe that the law of $\mcl G^\ep$ is invariant under the operation of translating its vertex set by $\ep$, and consequently
$\mcl G^\ep$ is the distributional local limit as $n\rta\infty$ of the planar map whose vertex set is $[-n \ep ,n\ep]_{\ep\BB Z  }$, with two vertices connected by an edge if and only if they are connected by an edge in $\mcl G^\ep$, each rooted at a uniformly random vertex in $[-n\ep,n\ep]_{\BB Z  }$. 

We note that it is also known that the simple random walk on the adjacency graph of cells associated with a space-filling SLE$_{\kappa'}$ on an independent 0-quantum cone is recurrent: indeed, this follows from~\cite[Theorem 1.16]{gms-random-walk} and~\cite[Proposition 3.1]{gms-tutte}.

\subsection{Maximal cell diameter}
\label{sec-max-cell-diam}

In this brief subsection we establish a polynomial upper bound for the maximum size of the cells of $\mcl G^\ep$ which intersect a fixed Euclidean ball. In other words, we prove an analog of Corollary~\ref{cor-max-edge} with the SLE/LQG embedding in place of the Tutte embedding, which we will eventually use to prove Corollary~\ref{cor-max-edge}. 

\begin{lem} \label{lem-max-cell-diam}
Suppose we are in the setting of Section~\ref{sec-standard-setup}, with $h$ either a whole-plane GFF normalized so that its circle average over $\bdy\BB D$ is zero or the circle-average embedding of a 0-quantum cone or a $\gamma$-quantum cone. 
For each $q \in \left( 0 , \tfrac{2}{(2+\gamma)^2} \right)$, each $\rho \in (0,1)$, and each $\ep \in (0,1)$, 
\eqb \label{eqn-max-cell-diam}
\BB P\left[ \op{diam}\left(H_x^\ep \right) \leq \ep^q ,\: \forall x\in \mcl V\mcl G^\ep(B_{\rho}(0))  \right] \geq 1 -  \ep^{\alpha(q,\gamma) + o_\ep(1)}  ,
\eqe  
where the rate of the $o_\ep(1)$ depends only on $q$, $\rho$, and $\gamma$ and
\eqbn 
\alpha(q,\gamma) := \frac{q}{2\gamma^2  } \left( \frac{1}{q} - 2 - \frac{\gamma^2}{2} \right)^2 - 2 q  .
\eqen
\end{lem}

Lemma~\ref{lem-max-cell-diam} is an easy consequence of the following basic estimate for the $\gamma$-LQG measure.

\begin{lem} \label{lem-min-ball}
Let $h$ be as in Lemma~\ref{lem-max-cell-diam}. For $\delta \in (0,1)$, $p >2\gamma$, and $\rho\in (0,1)$, 
\eqb \label{eqn-min-ball} 
\BB P\left[ \inf_{z\in B_{\rho}(0) } \mu_{h}(B_\delta(z)) \geq \delta^{2+\tfrac{\gamma^2}{2} +p}  \right] \geq 1 - \delta^{\tfrac{p^2}{2\gamma^2} -2 + o_\delta(1)}  ,
\eqe 
with the rate of the $o_\delta(1)$ depending on $p$, $\rho$, and $\gamma$. 
\end{lem}
\begin{proof}
If $h$ is a circle-average embedding of an $\alpha$-quantum cone, then $h|_{\BB D}$ agrees in law with a whole-plane GFF normalized so that its circle average over $\bdy\BB D$ is 0 plus $-\alpha\log|\cdot|$. If $h$ is a whole-plane GFF, then $\mu_{h-\alpha\log|\cdot|}(A) \geq \mu_h(A)$ for each Borel set $A\subset \BB D$. So, we can restrict attention to the case when $h$ is a whole-plane GFF normalized so that its circle average over $\bdy\BB D$ is 0.

Let $\{h_r\}_{r\geq 0}$ be the circle average process of $h$ and fix $\zeta \in (0,1)$. 
By standard estimates for the $\gamma$-LQG measure (see, e.g.,~\cite[Lemma~3.12]{ghm-kpz}), for $z\in \BB C$, 
\eqbn
\BB P\left[ \mu_{h}(B_\delta(z)) < \delta^{2+\tfrac{\gamma^2}{2} + \zeta }  e^{\gamma h_\delta(z)} \right] =o_\delta^\infty(\delta) .
\eqen
For $z\in B_{ \rho}(0)$, the random variable $h_\delta(z)$ is centered Gaussian with variance $\log \delta^{-1}  + O_\delta(1)$. Therefore,
\eqbn
\BB P\left[ e^{\gamma h_\delta(z)} < \delta^{p-\zeta} \right] \leq \delta^{\tfrac{(p-\zeta)^2}{2\gamma^2}} .
\eqen
Combining these estimates and sending $\zeta\rta 0$ shows that 
\eqb \label{eqn-min-ball-1pt}
\BB P\left[ \mu_{h}(B_\delta(z)) < \delta^{2+ \tfrac{\gamma^2}{2}  +p } \right] \leq \delta^{\tfrac{p^2}{2\gamma^2} + o_\delta(1)} .
\eqe 
We obtain~\eqref{eqn-min-ball} by applying~\eqref{eqn-min-ball-1pt} with $\delta/2$ in place of $\delta$ then taking a union bound over all $z\in (\frac{\delta}{4}\BB Z^2 ) \cap B_{ \rho}(0)$. 
\end{proof}

\begin{proof}[Proof of Lemma~\ref{lem-max-cell-diam}]
Fix $\wt q \in \left( q , \tfrac{2}{(2+\gamma)^2} \right)$. 
By Lemma~\ref{lem-min-ball} applied with $1/\wt q - 2-\gamma^2/2  >2\gamma$ in place of $p$ and $\ep^{\wt q}$ in place of $\delta$, 
it holds with probability at least $1-\ep^{\alpha(\wt q,\gamma) +o_\ep(1)}$ that each Euclidean ball contained in $B_{(1+\rho)/2}(0)$ with radius at least $ \ep^{ \wt q }  $ has $\mu_h$-mass at least $\ep$.  
By~\cite[Proposition~3.4 and Remark~3.9]{ghm-kpz}, it holds except on an event of probability $o_\ep^\infty(\ep)$ that each segment of $\eta $ contained in $\BB D$ with diameter at least $ \ep^{ q}  $ contains a Euclidean ball of radius at least $  \ep^{\wt q }  $. 
Hence with probability $1-\ep^{\alpha(\wt q,\gamma) +o_\ep(1)}$, each segment of $\eta$ which intersects $B_\rho(0)$ and has Euclidean diameter at lest $\ep^q$ has $\mu_h$-mass at least $\ep$. 
Each cell $H_x^\ep = \eta([x-\ep,x])$ is a segment of $\eta$ with $\mu_h$-mass $\ep$, so with probability at least $1-\ep^{\alpha(\wt q,\gamma) +o_\ep(1)}$ each such cell which intersects $B_{ \rho}(0)$ has diameter at most $\ep^{  q}$. 
Sending $\wt q \rta q$ concludes the proof.  
\end{proof}
 
\subsection{Estimates for integrals against structure graph cells}
\label{sec-cell-int}

A key input in our estimates for harmonic functions on $\mcl G^\ep$ (i.e., Theorem~\ref{thm-dirichlet-harmonic0} and Theorem~\ref{thm-cont0}) is a bound for the integrals of Euclidean functions against quantities associated with the cells of $\mcl G^\ep$. 
We state this bound in this subsection, and postpone its proof until Section~\ref{sec-area-diam-deg}. 
 
Suppose we are in the setting of Section~\ref{sec-standard-setup} and that $h$ is the circle-average embedding of  a $\gamma$-quantum cone.
For $z\in \BB C$, let 
\eqb \label{eqn-area-diam-deg-def}
u^\ep(z) :=  \frac{ \op{diam}(H_{x_z^\ep}^\ep )^2}{\op{area}(H_{x_z^\ep}^\ep )}  \op{deg}(x_z^\ep ; \mcl G^\ep) .
\eqe
Our main estimate for $u^\ep(z)$ is a one-sided ``law of large numbers" type estimate for integrals against $u^\ep(z)$. 
The following is a simplified (but perhaps more intuitive) version of our result, which states in quantitative way that the mean value of $u^\ep(z)$ tends to be smaller than a fixed constant.

\begin{prop} \label{prop-area-diam-deg-lln0} 
For each $\rho \in (0,1)$, there are constants $\alpha = \alpha(\gamma) > 0$ and $A =A( \rho, \gamma) > 0$ such that the following is true. Suppose $C>1$ and $\ep \in (0,1)$ and $D\subset B_{ \rho}(0)$ is a domain with $\op{area}(B_r(\bdy D)) \leq C r  $ for each $r \in (0,1)$ and $\op{area}(D) \geq \ep^\alpha$. Then
\eqb \label{eqn-area-diam-deg-lln0}
\BB P\left[ \int_{D}  u^\ep(z) \, dz  \leq  A \op{area}(D)  \right] \geq 1 - O_\ep(\ep^\alpha)
\eqe 
with the rate of the $O_\ep(\ep^\alpha)$ depending only on $C$, $\rho$, and $\gamma$. 
\end{prop}

 
Proposition~\ref{prop-area-diam-deg-lln0} is an immediate consequence of Proposition~\ref{prop-area-diam-deg-gamma} below; it can be derived from Proposition~\ref{prop-area-diam-deg-gamma} by setting $f = 1_D$.  
  
\begin{prop} \label{prop-area-diam-deg-gamma}
For $\ep \in (0,1)$ and $z\in \BB C$, define $u^\ep(z)$ as in~\eqref{eqn-area-diam-deg-def}. 
There exists $\alpha = \alpha(\gamma ) >0$ and $\beta = \beta(\gamma )>0$, and $A = A(  \rho , \gamma) > 0$ such that the following is true. Let $C > 1$ and let $D\subset B_{  \rho}(0)$ be a domain such that $\op{area}(B_r(\bdy D)) \leq C r  $ for each $r \in (0,1)$. 
Also let $f : D\rta [0,\infty)$ be a non-negative function which is $C\ep^{-\beta}$-Lipschitz continuous and which satisfies $\|f\|_\infty\leq C\ep^{-\beta}$. 
Then
\eqb \label{eqn-area-diam-deg-gamma}
\BB P\left[ \int_{D} f(z) u^\ep(z) \, dz  \leq A \int_{D} f(z) \,dz + \ep^\alpha   \right] \geq 1 - O_\ep(\ep^\alpha)
\eqe 
with the rate of the $O_\ep(\ep^\alpha)$ depending only on $C$, $\rho$, and $\gamma$. 
\end{prop}

\section{Estimates for harmonic functions on $\mcl G^\ep$}
\label{sec-energy-estimate}

Suppose we are in the setting of Section~\ref{sec-standard-setup} with $h$ equal to the circle-average embedding of a $\gamma$-quantum cone. 
Recall that $\eta$ is a whole-plane space-filling SLE$_{\kappa'}$ parameterized by $\gamma$-quantum mass with respect to $h$, and $\mcl G^\ep$ for $\ep > 0$ is the associated mated-CRT map. Recall also that for $z\in\BB C$, $x_z^\ep$ is the smallest (and a.s.\ only) element of $\ep\BB Z = \mcl V\mcl G^\ep$ for which $z$ is contained in the cell $H_{x_z^\ep}^\ep = \eta([x_z^\ep -\ep , x_z^\ep])$.

In this section, we assume Proposition~\ref{prop-area-diam-deg-gamma} and use it to deduce various bounds for harmonic functions on the sub-graph $\mcl G^\ep(D)$ defined as in~\eqref{eqn-structure-graph-domain}, which will eventually lead to Theorems~\ref{thm-dirichlet-harmonic0} and~\ref{thm-cont0}.  

Many of the estimates in this subsection will include constants $\alpha,\beta$, and $A$ which are required to be independent of $\ep$, but are allowed to be different in each lemma/proposition/theorem.

Throughout this section, we fix $\rho \in (0,1)$ and work on the ball $B_{\rho}(0)$.  

\subsection{Comparing sums over cells and Lebesgue integrals}
\label{sec-sum-to-int}

In this subsection we establish a variant of Proposition~\ref{prop-area-diam-deg-gamma} which allows us to compare the weighted sum of the values of a function on $\BB C$ over all cells in the restricted structure graph $\mcl V\mcl G^\ep(D)$ to its integral over $D$. 

 \begin{lem} \label{lem-sum-to-int}
There exists $\alpha = \alpha(\gamma ) >0$, $\beta = \beta(\gamma )>0$, and $A = A(   \rho , \gamma) > 0$ such that the following is true. Let $C\geq 1$ and let $D\subset B_{\rho}(0)$ be a domain such that $\op{area}(B_r(\bdy D)) \leq C r  $ for each $r \in (0,1)$. 
Let $f : D\rta [0,\infty)$ be a non-negative function which is $C\ep^{-\beta}$-Lipschitz continuous and which satisfies $\|f\|_\infty\leq C\ep^{-\beta}$ and define $f^\ep : \mcl V\mcl G^\ep(D) \rta \BB R$ by
\eqb \label{eqn-sum-to-int-function}
f^\ep (x) := 
\begin{cases}
f(\eta(x)) ,\quad &x\in \mcl V\mcl G^\ep(D)\setminus \mcl V\mcl G^\ep(\bdy D)  \\
\sup_{z\in H_x^\ep \cap \bdy D} f(z) ,\quad &x\in \mcl V\mcl G^\ep(\bdy D)  .
\end{cases}
\eqe 
Then
\eqb \label{eqn-sum-to-int}
\BB P\left[  \sum_{x\in \mcl V\mcl G^\ep(D)} f^\ep(x) \op{diam}(H_x^\ep)^2 \op{deg}(x;\mcl G^\ep)  \leq A \int_D f(z) \,dz  + \ep^\alpha \right] \geq 1-O_\ep( \ep^\alpha )
\eqe 
at a rate depending only on $C$, $\rho$, and $\gamma$. 
\end{lem}

We note that the choice of boundary data for $f^\ep$ in~\eqref{eqn-sum-to-int-function} (which will also show up in other places) is somewhat arbitrary---we just need boundary data which is close in some sense to the boundary data for $f$ when $\ep$ is small.

\begin{proof}[Proof of Lemma~\ref{lem-sum-to-int}]
Let $q  \in \left(0, \tfrac{2}{(2+\gamma)^2}\right)$, chosen later in a manner depending only on $\gamma$, and let $\beta > 0$ be smaller than the minimum of $q$ and the parameter $\beta$ of Proposition~\ref{prop-area-diam-deg-gamma}. We define the event
\eqb \label{eqn-sum-to-int-event}
\wh E_0^\ep :=   \left\{ \op{diam}\left( H_{x_z^\ep}^\ep \right) \leq \ep^{q} \, \op{and} \,  \op{deg}\left( x_z^\ep ; \mcl G^\ep \right) \leq (\log \ep^{-1})^2, \: \forall z \in B_{ \rho}(0) \right\} ,
\eqe
where here we recall that $H_{x_z^\ep}^\ep$ is the cell of $\mcl G^\ep$ containing $z$. 
By Lemmas~\ref{lem-max-deg-nonlocal} and~\ref{lem-max-cell-diam}, $\BB P[(\wh E_0^\ep)^c]$ decays faster than some positive power of $\ep$. 

Fix $\beta > 0$ to be chosen later in a manner depending only on $\gamma$.
We will bound the sum over $\mcl V\mcl G^\ep(\bdy D)$ and $\mcl V\mcl G^\ep( D) \setminus \mcl V\mcl G^\ep(\bdy D)$ separately. We start with the boundary vertices. If $x\in \mcl V\mcl G^\ep(\bdy D)$ then by the definition~\eqref{eqn-sum-to-int-event} of $\wh E_0^\ep$, we have $\op{deg}(x ;\mcl G^\ep)   \leq (\log \ep^{-1})^2$ and $H_x^\ep \subset B_{\ep^q}(\bdy D)$ on this event. By our hypotheses that $\op{area}(B_r(\bdy D)) \leq C r  $ for each $r \in (0,1)$ and $\|f\|_\infty \leq C \ep^{-\beta}$,  
\allb \label{eqn-sum-to-int-bdy}
&\BB 1_{\wh E_0^\ep} \sum_{x \in  \mcl V\mcl G^\ep(\bdy D) }  f^\ep(x) \op{diam}(H_x^\ep )^2 \op{deg}(x ;\mcl G^\ep)  
 \leq  \BB 1_{\wh E_0^\ep} \| f\|_\infty (\log \ep^{-1})^2 \sum_{x \in \mcl V\mcl G^\ep(\bdy D) } \op{diam}(H_x^\ep )^2   \notag \\
&\qquad \qquad \qquad \qquad \leq \|  f\|_\infty  (\log \ep^{-1})^2 \op{area}\left(    B_{\ep^q}(\bdy D)  \right)   
 \leq C^2 (\log\ep^{-1})^2  \ep^{-\beta     + q}   . 
\alle
Since $\beta < q   $ this last quantity is bounded above by $O_\ep(\ep^{\alpha_0})$ for $\alpha_0  \in (0,  q   - \beta  ) $.  

Now we turn our attention to the interior vertices.
Recall the definition~\eqref{eqn-area-diam-deg-def} of $u^\ep(z)$. For $x\in \mcl V\mcl G^\ep( D) \setminus \mcl V\mcl G^\ep(\bdy D)$,  
\allb \label{eqn-vertex-to-avg}
f^\ep(x) \op{diam}(H_x^\ep)^2 \op{deg}(x ;\mcl G^\ep) 
&= \int_{H_x^\ep} f^\ep(x) u^\ep(z) \,dz  \notag \\
&\leq \int_{H_x^\ep} f(z)  u^\ep(z)  \,dz    
  + C \ep^{- \beta} \op{diam}(H_x^\ep)^3 \op{deg}(x ;\mcl G^\ep)   ,
\alle  
where in the last inequality we use the $C\ep^{-\beta}$-Lipschitz continuity of $f$ to get that for $z\in H_x^\ep$, $|f^\ep(x) - f(z)| \leq C \ep^{-\beta} \op{diam}(H_x^\ep)  $. 
By~\eqref{eqn-vertex-to-avg}, 
\allb \label{eqn-sum-to-int-main}
&\sum_{x \in  \mcl V\mcl G^\ep( D)\setminus  \mcl V\mcl G^\ep(\bdy D) }   f^\ep(x) \op{diam}(H_x^\ep)^2 \op{deg}(x ;\mcl G^\ep)  \notag\\ 
&\qquad \leq  \int_D f(z) u^\ep(z)  \,dz   +   C \ep^{- \beta} \sum_{x\in   \mcl V\mcl G^\ep( D)\setminus  \mcl V\mcl G^\ep(\bdy D) }   \op{diam}(H_x^\ep)^3 \op{deg}(x ;\mcl G^\ep)  .
\alle
 
On $\wh E_0^\ep$, the sum on the right in~\eqref{eqn-sum-to-int-main} satisfies
\allb \label{eqn-sum-to-int-error}
\ep^{- \beta} \sum_{x\in \mcl V\mcl G^\ep( D)\setminus  \mcl V\mcl G^\ep(\bdy D) }   \op{diam}(H_x^\ep)^3  \op{deg}(x; \mcl G^\ep)  
&= \ep^{- \beta} \sum_{x\in \mcl V\mcl G^\ep( D)\setminus  \mcl V\mcl G^\ep(\bdy D) }   \int_{H_x^\ep}  \frac{ \op{diam}(H_x^\ep)^3}{\op{area}(H_x^\ep)}  \op{deg}(x; \mcl G^\ep)  \, dz\notag \\
&\leq   \ep^{- \beta} \int_{D}   \op{diam}(H_{x_z^\ep}^\ep ) u^\ep(z)   \, dz    
\leq \ep^{q- \beta} \int_{D}    u^\ep(z)   \, dz .
\alle  
Proposition~\ref{prop-area-diam-deg-gamma} (applied to $f$ and with 1 in place of $f$) shows that there are constants $\alpha_1 = \alpha_1(\gamma)   \in (0,\alpha_0]$ and $A = A(\rho,\gamma) > 0$ such that with probability at least $1-O_\ep(\ep^{\alpha_1})$, 
\eqb \label{eqn-sum-to-int-convert}
  \int_D f(z) u^\ep(z)  \,dz  \leq A \int_D f(z) \,dz + \ep^{\alpha_1}  \quad \op{and} \quad  \int_{D}    u^\ep(z)   \, dz \leq A \op{area}(D) + \ep^{\alpha_1} .
\eqe
Plugging~\eqref{eqn-sum-to-int-convert} and~\eqref{eqn-sum-to-int-error} into~\eqref{eqn-sum-to-int-main} and then adding the resulting estimate to~\eqref{eqn-sum-to-int-bdy} and possibly shrinking $\alpha_1$ shows that on $\wh E_0^\ep$, 
\eqb \label{eqn-sum-to-int0}
     \sum_{x \in \mcl V\mcl G^\ep( D) } f^\ep(x)   \op{diam}(H_x^\ep)^2 \op{deg}(x;\mcl G^\ep)    
     \leq A  \int_D f(z)  \,dz   + O_\ep(\ep^{\alpha_1}) .
\eqe 
Since $\BB P[(\wh E_0^\ep)^c]$ decays like a positive power of $\ep$, we obtain~\eqref{eqn-sum-to-int} with an appropriate choice of $\alpha \in (0,\alpha_1)$.   
\end{proof}

\subsection{Dirichlet energy bounds}
\label{sec-dirichlet-mean}

In this subsection, we will use Proposition~\ref{prop-area-diam-deg-gamma} to prove bounds for the Dirichlet energy of discrete harmonic functions on subgraphs of $\mcl G^\ep$ in terms of the Dirichlet energy of functions on subsets of $\BB C$. 
We will consider the following setup. Recall that we have fixed $\rho \in (0,1)$. 
For $C \geq 1$ and $\delta \in (0,1)$, let $\mcl C_C(\delta ) = \mcl C_C (\delta , \rho)$ be the set of pairs $(D,f)$ where $D$ is an open subset of $B_{\rho}(0)$ and $f : \ol D\rta\BB R$ is a differentiable function such that the following is true. 
\begin{enumerate}
\item $\op{area}(B_r(\bdy D)) \leq C r $ for each $r \in (0,1)$. 
\item $D$ has \emph{$C$-bounded convexity}, i.e., for each $z,w\in D$ there is a path from $z$ to $w$ contained in $D$ which has length at most $C |z-w|$. 
\item $\nabla f$ is $ \delta^{-1}$-Lipschitz continuous and both $\|f\|_\infty$ and $\|\nabla f\|_\infty$ are at most $ \delta^{-1}$.  
\end{enumerate}  
The main result of this subsection is the following more quantitative version of Theorem~\ref{thm-dirichlet-harmonic0}.

\begin{thm} \label{thm-dirichlet-harmonic} 
There are constants $\alpha = \alpha(\gamma ) > 0$ and $\beta =\beta(\gamma ) > 0$ such that for each $C\geq 1$ there exists $A = A( C ,   \rho , \gamma) >0$ such that the following hold for each $\ep\in (0,1)$, each $C\geq 1$, and each $(D,f) \in \mcl C_C (\ep^\beta  )$. 
Let $\frk f^\ep : \mcl V\mcl G^\ep( D)  \rta \BB R$ be the function such that 
\eqbn
\frk f^\ep(x) = \sup_{z\in H_x^\ep \cap \bdy D} f(z) ,\quad \forall x\in \mcl V\mcl G^\ep(\bdy D) 
\eqen
and $\frk f^\ep$ is discrete harmonic on $\mcl V\mcl G^\ep( D) \setminus \mcl V\mcl G^\ep( \bdy D) $. Then (recalling Definitions~\ref{def-discrete-dirichlet} and~\ref{def-continuum-dirichlet}),
\eqb \label{eqn-dirichlet-harmonic}
\BB P\left[ \op{Energy}\left(\frk f^\ep ; \mcl G^\ep(D) \right)  \leq A  \op{Energy}\left(f ; D \right)  + \ep^\alpha  \right] \geq 1 - O_\ep(   \ep^\alpha )
\eqe 
at a rate depending only on $C ,  \rho$, and $\gamma$. 
\end{thm}

Theorem~\ref{thm-dirichlet-harmonic} will be an immediate consequence of the following estimate, which in turn is deduced from Lemma~\ref{lem-sum-to-int}. 

\begin{lem} \label{lem-dirichlet-mean}   
There are constants $\alpha = \alpha(\gamma ) > 0$,  $\beta =\beta( \gamma  ) > 0$, and $A = A( C ,  \rho , \gamma) >0$ such that the following is true for each $\ep\in (0,1)$ and each $(D,f) \in \mcl C_C(\ep^\beta  )$. 
As in Lemma~\ref{lem-sum-to-int}, define $f^\ep : \mcl V\mcl G^\ep(D) \rta \BB R$ by
\eqb \label{eqn-function-approx}
f^\ep (x) := 
\begin{cases}
f(\eta(x)) ,\quad &x\in \mcl V\mcl G^\ep(D)\setminus \mcl V\mcl G^\ep(\bdy D)  \\
\sup_{z\in H_x^\ep \cap \bdy D} f(z) ,\quad &x\in \mcl V\mcl G^\ep(\bdy D)  .
\end{cases}
\eqe 
Then
\eqb \label{eqn-dirichlet-mean}
 \BB P\left[ \op{Energy}\left( f^\ep  ; \mcl G^\ep(D) \right)  \leq A  \op{Energy}\left(f ; D \right) + \ep^\alpha   \right] \geq 1 - O_\ep( \ep^\alpha   )
\eqe  
at a rate depending only on $C , \rho$, and $\gamma$.
\end{lem} 
\begin{proof}
Fix $q \in \left(0, \tfrac{2}{(2+\gamma)^2}\right)$, chosen in a manner depending only on $\gamma$, and let  
\eqb \label{eqn-max-cell-event}
E_0^\ep = E_0^\ep(q,\rho ) := \left\{ \op{diam}\left( H_{x_z^\ep}^\ep \right) \leq \ep^{q}  ,\, \forall z\in B_{\rho}(0) \right\} .
\eqe 

Also fix $\beta > 0$ to be chosen later in a manner depending only on $\gamma$ and suppose $(D,f) \in \mcl C_C(\ep^\beta )$. 
By analogy with~\eqref{eqn-function-approx}, define
\eqbn
F^\ep (x) := 
\begin{cases}
|\nabla f(\eta(x))|^2 ,\quad &x\in \mcl V\mcl G^\ep(D)\setminus \mcl V\mcl G^\ep(\bdy D)  \\
\sup_{z\in H_x^\ep \cap \bdy D} |\nabla f(z)|^2 ,\quad &x\in \mcl V\mcl G^\ep(\bdy D)  .
\end{cases}
\eqen
Now consider an edge $\{x,y\} \in \mcl E\mcl G^\ep(D) $. 
Then $H_x^\ep\cap H_y^\ep\not=\emptyset$, so $\op{diam}(H_x^\ep\cup H_y^\ep) \leq \op{diam}(H_x^\ep) + \op{diam}(H_y^\ep)$.
By the $C$-convexity of $D$, for any $z\in H_x^\ep \cap D$ and any $w\in H_y^\ep \cap D$, there is a path $P_{z,w}$ from $z$ to $w$ in $D$ of Euclidean length at most $C (\op{diam}(H_x^\ep) + \op{diam}(H_y^\ep))$. By the $ \ep^{-\beta}$-Lipschitz continuity of $\nabla f$, for each $u\in P_{z,w}$ we have 
\eqbn
|\nabla f(u)| 
\leq |\nabla f(z)| + C \ep^{-\beta} \left(  \op{diam}(H_x^\ep  ) + \op{diam}(H_y^\ep  )     \right)
\leq \sqrt{F^\ep(x)} + 2 C \ep^{-\beta} \left(  \op{diam}(H_x^\ep  ) + \op{diam}(H_y^\ep  )     \right) ,
\eqen
and similarly with $y$ in place of $x$. 
Therefore,
\alb
|f^\ep(x) - f^\ep(y)| 
&\leq C \left(  \op{diam}(H_x^\ep  ) + \op{diam}(H_y^\ep  )     \right)  \sup_{z \in H_x^\ep \cap D ,\, w\in H_y^\ep \cap D} \sup_{u\in P_{z,w}} |\nabla f(u)|  \\
&\leq  C \sqrt{F^\ep(x)} \op{diam}(H_x^\ep ) +  C \sqrt{F^\ep(y) } \op{diam}(H_y^\ep)  
 +  4 C \ep^{-\beta} \left( \op{diam}(H_x^\ep )^2  + \op{diam}(H_y^\ep)^2  \right) . 
\ale   
Using the above estimate and the inequality $(a+b)^2\leq 2(a^2+b^2)$ and breaking up the sum over edges based on those edges which have a given vertex $x$ as an endpoint, we obtain that on $E_0^\ep$, 
\allb \label{eqn-deriv-sum-decomp}
&\op{Energy}\left( f^\ep  ; \mcl G^\ep(D) \right) 
 \preceq \sum_{x \in \mcl V\mcl G^\ep(D)  }   F^\ep(x) \op{diam}(H_x^\ep )^2 \op{deg}(x ;\mcl G^\ep)  
 + \ep^{-2\beta} \sum_{x\in \mcl V\mcl G^\ep(D)  }   \op{diam}(H_x^\ep )^4  \op{deg}(x; \mcl G^\ep)  \notag\\ 
&\qquad \qquad \preceq \sum_{x \in \mcl V\mcl G^\ep(D)  }   F^\ep(x) \op{diam}(H_x^\ep )^2 \op{deg}(x ;\mcl G^\ep)  
 + \ep^{2q-2\beta} \sum_{x\in \mcl V\mcl G^\ep(D)  } \op{diam}(H_x^\ep )^2    \op{deg}(x; \mcl G^\ep) 
\alle 
with implicit constant depending only on $C $, where here we use that $\op{diam}(H_x^\ep ) \leq \ep^q$ on $E_0^\ep$. Since $(D,f) \in \mcl C_C (\ep^\beta)$, the function $|\nabla f|^2$ is $2  \ep^{-2\beta}$-Lipschitz and $\| |\nabla f|^2\|_\infty \leq   \ep^{-2\beta}$. We can therefore apply Lemma~\ref{lem-sum-to-int} (with each of $|\nabla f|^2$ and $1$ in place of $f$) to see that if $\beta$ is smaller than the minimum of $q$ and $1/2$ times the parameter $\beta$ from Lemma~\ref{lem-sum-to-int}, then the following is true.
For appropriate constants $A,\alpha > 0$ as in the statement of the lemma, it holds except on an event of probability decaying faster than some positive power of $\ep$ that the right side of~\eqref{eqn-deriv-sum-decomp} is bounded above by $A  \op{Energy}\left(f ; D \right) + \ep^\alpha $. Since $\BB P[(E_0^\ep)^c]$ decays like a positive power of $\ep$, this concludes the proof. 
\end{proof}

\begin{proof}[Proof of Theorems~\ref{thm-dirichlet-harmonic0} and~\ref{thm-dirichlet-harmonic}]
Since discrete harmonic functions minimize Dirichlet energy subject to specified boundary data, Theorem~\ref{thm-dirichlet-harmonic} is an immediate consequence of Lemma~\ref{lem-dirichlet-mean}. Theorem~\ref{thm-dirichlet-harmonic0}, in turn, follows from Theorem~\ref{thm-dirichlet-harmonic}.  
\end{proof}

\subsection{Green's function and recurrence}
\label{sec-recurrence}

We will now explain why Theorem~\ref{thm-dirichlet-harmonic} implies Theorem~\ref{thm-green-lower}. The main step is the following upper bound for the Dirichlet energy of certain discrete harmonic functions on $\mcl G^\ep$.

\begin{lem} \label{lem-annulus-energy}
There exists $\alpha = \alpha(\gamma) > 0$ such that for each $\rho \in (0,1)$, there exists $A = A( \rho,\gamma) > 0$ such that for each $s\in [0,\rho/2]$ and each $\ep\in (0,1)$, it holds with probability at least $1-O_\ep(\ep^\alpha)$ (at a rate depending only on $\rho$ and $\gamma$) that the following is true. 
Let $\frk f_s^\ep : \mcl V\mcl G^\ep(B_{ \rho}(0) \setminus B_s(0) ) \to [0,1]$ be the function which is equal to $0$ on $\mcl V\mcl G^\ep(\bdy B_{\rho}(0) )$, $1$ on $\mcl V\mcl G^\ep(\bdy B_s(0)  )$, and is discrete harmonic on the rest of $\mcl V\mcl G^\ep(B_{\rho}(0) \setminus B_s(0) )$. 
Then (in the notation of Definition~\ref{def-discrete-dirichlet})
\eqb \label{eqn-annulus-energy}
\op{Energy}\left( \frk f_s^\ep ; \mcl G^\ep(B_{\rho}(0) ) \right) \leq \frac{A}{\log(\ep^{-1} \vee s^{-1})}  .
\eqe 
\end{lem}

We note that by Lemma~\ref{lem-max-cell-diam}, we have $\mcl V\mcl G^\ep(\bdy B_{\rho}(0)) \cap \mcl V\mcl G^\ep( B_{\rho/2}(0) ) = \emptyset$ --- which implies that $\frk f_s^\ep$ is well-defined for each $s \in [0,\rho/2]$ --- except on an event of probability decaying faster than some positive power of $\ep$. 

\begin{proof}[Proof of Lemma~\ref{lem-annulus-energy}]
To lighten notation, define the open annulus $\BB A_s := B_{ \rho}(0) \setminus \ol{B_s(0)}$. 
We will apply Theorem~\ref{thm-dirichlet-harmonic} to the function $\frk g_s  : \ol{\BB A_{s, }} \to [0,1]$ which is equal to $0$ on $\bdy B_{ \rho}(0)$, $1$ on $\bdy B_s(0)$, and is harmonic on the interior of $\BB A_{s  }$. That is, $\frk g_s(z) = \log(\rho/|z|) / \log(\rho/s)$. 
A direct calculation shows that the Euclidean Dirichlet energy of $\frk g_s$ on $\BB A_{s  }$ is $2\pi/\log(\rho/s)$. Furthermore, $\frk g_s$ and each of its first and second order partial derivatives are bounded above by a universal constant times a universal negative power of $s$ on $\ol{\BB A_{s }}$. 
Consequently, Theorem~\ref{thm-dirichlet-harmonic} implies that there exists $\beta=\beta(\gamma)> 0$ and an appropriate choice of $\alpha$ and $A$ as in the statement of the lemma such that the statement of the lemma is true if we impose the additional requirement that $s\geq \ep^\beta$. 

To remove the restriction that $s \geq \ep^\beta$, we extend $\frk f_s^\ep$ to all of $\mcl V\mcl G^\ep(B_{\rho}(0) )$ by requiring it to be identically equal to $1$ on $\mcl V\mcl G^\ep(B_s(0))$. Then the total Dirichlet energy of $\frk f_s^\ep$ is unchanged and if $s'  >s$, then $\frk f_s^\ep$ and $\frk f_{s'}^\ep$ agree on $ \mcl V\mcl G^\ep(\BB A_{s })$. Since $\frk f_{s }^\ep$ has the minimal Dirichlet energy among all functions on $ \mcl V\mcl G^\ep(\BB A_{s })$ with the same boundary data, we infer that $s\mapsto \op{Energy}\left( \frk f_s^\ep ; \mcl G^\ep(\BB A_{s } ) \right)$ is non-decreasing. Therefore,~\eqref{eqn-annulus-energy} for $s = \ep^\beta$ implies~\eqref{eqn-annulus-energy} for $s \in [0,\ep^\beta]$ with $A/\beta$ in place of $A$.  
\end{proof}   

\begin{proof}[Proof of Theorem~\ref{thm-green-lower}]
By Dirichlet's principle (see, e.g.,~\cite[Exercise 2.13]{lyons-peres}), if $\frk f_0^\ep : \mcl V\mcl G^\ep(B_\rho(0)) \to [0,1]$ is the function which vanishes on $\mcl V\mcl G^\ep(\bdy B_\rho(0))$, is equal to $1$ at $0$, and is otherwise discrete harmonic then 
\eqbn
\frac{\op{Gr}_{\tau^\ep}^\ep(0,0)}{\op{deg}\left(0 ; \mcl G^\ep \right)} =   \op{Energy}\left(\frk f_0^\ep ; \mcl G^\ep(B_{\rho}(0)  ) \right)^{-1}  
\eqen
Hence the Green's function bound~\eqref{eqn-green-lower} follows from Lemma~\ref{lem-annulus-energy}.
 
To deduce the recurrence of simple random walk on $\mcl G^\ep$ from this bound, it suffices to consider the case when $\ep = 1$ since the law of $\mcl G^\ep$ (as a graph) does not depend on $\ep$. We will use the scaling property of the $\gamma$-quantum cone (described just below) to produce an increasing sequence of sub-graphs of $\mcl G^1$ (each corresponding to an open ball of random radius), whose union is all of $\mcl G^1$, with the property that the Green's function of the walk stopped upon exiting these subgraphs a.s.\ tends to $\infty$, which implies recurrence by a well-known criterion~\cite[Theorem 2.3]{lyons-peres}.  
 For this purpose, for $b >0$ let $R_b := \sup\{r > 0 : h_r(0) + Q\log r = \frac{1}{\gamma}\log b\}$, where $h_r(0)$ denotes the circle average.
 Note that $R_0 = 1$ since $h$ is assumed to have the circle average embedding. By~\cite[Proposition 4.13(i)]{wedges}, for $b>0$ we have $h \eqD h^b$ for $h^b := h(R_b\cdot) + Q\log R_b - \frac{1}{\gamma} \log b $, where $Q = 2/\gamma+\gamma /2$ is as in~\eqref{eqn-lqg-coord}. It is easily seen from the definition of $h$ (Definition~\ref{def-quantum-cone}) that a.s.\ $R_b\rta\infty$ as $b\rta\infty$. Since $\eta$ is sampled independently from $h$ and then parameterized by $\gamma$-LQG mass with respect to $h$, it follows that $(h^b,\eta^b) \eqD (h,\eta)$ for $\eta^b := R_b^{-1} \eta(b\cdot)$. In particular, $\mcl G^1(B_{\rho R_b}(0)) \eqD \mcl G^{1/b}(B_{ \rho}(0)) $. Applying this with $b = 2^{ k}$ for $k\in\BB N$, using~\eqref{eqn-green-lower} with $\ep = 2^{-k}$, and applying the Borel-Cantelli lemma now gives the desired recurrence.
\end{proof}

\subsection{Random walk on $\mcl G^\ep$ stays close to a curve with positive probability}
\label{sec-pos-prob}

In this subsection, we will prove Proposition~\ref{prop-pos-prob}, which says that, roughly speaking, the simple random walk on $\mcl G^\ep$ has positive probability to stay close to a fixed Euclidean curve for a long time, even if we condition on $\mcl G^\ep$. This estimate is the key input in the proof of our modulus of continuity bound in Section~\ref{sec-cont} below, but we expect it to also have other applications. 

\begin{defn} \label{def-walk-pt}
For $x\in\mcl V\mcl G^\ep$, we write $\ol{\BB P}_x^\ep$ for the conditional law given $(h,\eta )$ (which determines $\mcl G^\ep$) of the simple random walk $X^\ep$ on $\mcl G^\ep$ started from $x$. 
\end{defn}

\begin{prop} \label{prop-pos-prob}
For each $\rho \in (0,1)$, there exists $s_0 = s_0(\rho,\gamma) > 0$, $\alpha=\alpha(\gamma) > 0$, and $\beta = \beta(\gamma) > 0$ such that the following is true. 
Let $P\subset B_{\rho}(0)$ be a compact connected set, let $\ep \in (0,1)$, and let $\wt r , r > 0$ with $\ep^\beta \leq \wt r < r \leq \op{dist}(P, \bdy B_{\rho}(0))$. 
Also set
\eqb \label{eqn-pos-prob-N}
N  = N(P , \wt r , r) := \min\left\{ n \in \BB N : \text{$B_{\wt r}(P)$ can be covered by $n$ Euclidean balls of radius $s_0 (r-\wt r) $} \right\} .
\eqe 
Then with probability at least $1- N O_\ep(   \ep^\alpha)$ (at a rate depending only on $\rho$ and $\gamma$), 
\eqb \label{eqn-pos-prob}
\min_{x\in\mcl V\mcl G^\ep(B_{\wt r}(P))} \min_{z\in B_{\wt r}(P) }
\ol{\BB P}_x^\ep\left[ \text{$X^\ep$ enters $\mcl V\mcl G^\ep(B_r(z))$ before leaving $\mcl V\mcl G^\ep(B_r(P))$} \right] \geq 2^{-N} .
\eqe 
\end{prop}

We note that Proposition~\ref{prop-pos-prob} is not implied by the quenched convergence of $X^\ep$ to Brownian motion modulo time parameterization (proven in~\cite[Theorem 3.4]{gms-tutte}) since the latter convergence does not give a quantitative bound for the annealed probability that~\eqref{eqn-pos-prob} holds. 

To prove Proposition~\ref{prop-pos-prob}, we will show that a simple random walk on $\mcl G^\ep$ started close to the inner boundary of a Euclidean annulus is likely to hit the inner boundary before the outer boundary (Lemma~\ref{lem-annulus-cross}). This leads to Proposition~\ref{prop-pos-prob} by considering $N$ such annuli with the property that the union of their inner boundaries contains a path from $\eta(x)$ to $z$. See Figure~\ref{fig-pos-prob} for an illustration of the proof.

\begin{figure}[ht!]
\begin{center}
\includegraphics[scale=.7]{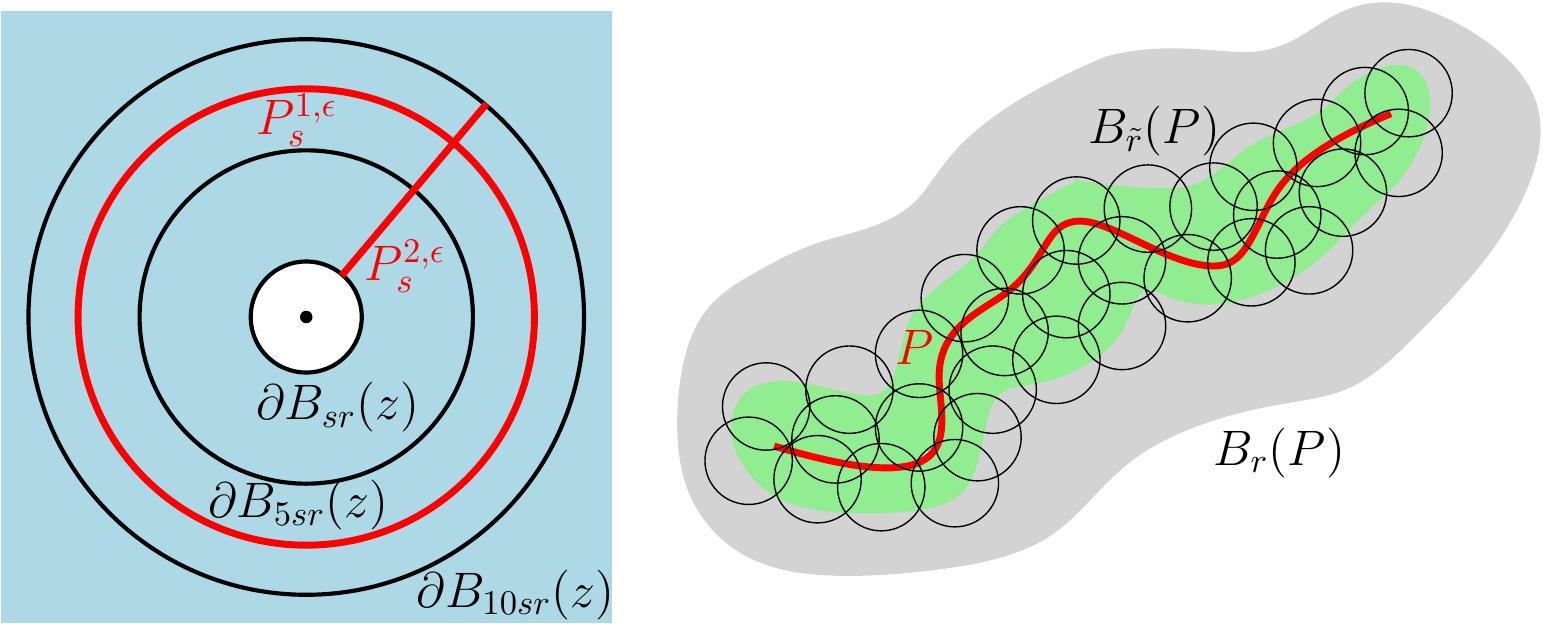} 
\caption{\label{fig-pos-prob} \textbf{Left:} Illustration of the proof of Lemma~\ref{lem-annulus-cross}. We use Lemma~\ref{lem-path-sum} and Theorem~\ref{thm-dirichlet-harmonic} to find a circle $P_{s }^{1,\ep}$ separating $\bdy B_{5 sr}(z)$ and $\bdy B_{10 sr}(z)$ and a radial line segment $P_{s }^{2,\ep}$ from $\bdy B_{sr}(z)$ to $\bdy B_{10 sr}(z)$ over which the total variation of $\frk f_s^\ep$ is at most a constant times $1/\sqrt{\log s^{-1}}$. By the maximum principle, this gives us a bound for the maximum value of $\frk f_s^\ep$ on $B_{5sr}(z)\setminus B_{sr}(z)$. \textbf{Right:} Illustration of the proof of Proposition~\ref{prop-pos-prob}. We cover $B_{\wt r}(P)$ by $N$ balls of radius $s_0 (r-\wt r)$, then use Lemma~\ref{lem-annulus-cross} to force a random walk to follow a ``string" of such balls from $\eta(x)$ to $z$. 
}
\end{center}
\end{figure}

Our desired bound for the probability of exiting an annulus at a point of its inner boundary can be re-phrased as a pointwise bound for a certain discrete harmonic function on $\mcl G^\ep$. 
The following technical lemma (which is a variant of~\cite[Lemma 2.16]{gms-random-walk}) enables us to transfer from the Dirichlet energy bounds of Section~\ref{sec-dirichlet-mean} to the needed pointwise bounds.
The idea of the statement and proof of the lemma is to consider a collection of paths $\{P_t\}_{t\in [a,b]}$, indexed by some finite interval, with the property that the Euclidean distance between $P_s$ and $P_t$ is bounded below by $|s-t|$. 
Due to the Cauchy-Schwarz inequality, the $t$-average of the total variation of a function on $\mcl V\mcl G^\ep$ over the paths $P_t$ can be bounded above in terms of the Dirichlet energy of $f^\ep$. 
There must be one path $P_t$ over which the total variation of $f^\ep$ is smaller than average, which (together with the maximum principle) will allow us to prove pointwise bounds for harmonic functions on $\mcl G^\ep$ in Lemma~\ref{lem-annulus-cross} below.

\begin{lem} \label{lem-path-sum}
There exists $\alpha = \alpha(  \gamma) >0$ such that for each $C\geq 1$ and each $\rho\in (0,1)$, we can find $A = A( C,  \rho,\gamma) > 0$ such that the following is true. Let $D\subset  B_{\rho}(0)$ be a domain such that $\op{area}(B_r(\bdy D)) \leq C r $ for each $r \in (0,1)$. 
For $\ep\in (0,1)$, it holds with probability at least $1-O_\ep(\ep^\alpha)$, at a rate depending only on $  C,  \rho$, and $\gamma$, that the following holds.
Let $\{P_t\}_{t\in [a,b]}$ be a collection of compact subsets of $\ol D$ such that $\op{dist}(P_s,P_t)  \geq C^{-1} |s-t|$ for each $s,t\in [a,b]$. 
Then for each function $f^\ep : \mcl V\mcl G^\ep(D) \rta \BB R$, 
\eqb \label{eqn-path-sum} 
  \int_a^b \sum_{\{x,y\} \in \mcl E \mcl G^\ep(P_t) } |f^\ep(x) - f^\ep(y)| \, dt \leq   
  A (\op{area}(D) + \ep^\alpha)^{1/2}  \op{Energy}\left(  f^\ep  ; \mcl G^\ep(D) \right)^{ 1/2}  .
\eqe   
\end{lem}
\begin{proof}
By Lemma~\ref{lem-sum-to-int} (applied with $f = \BB 1_D$) we can find $\alpha  = \alpha (  \gamma)>0$ and $A_0 = A_0(  C,    \rho,\gamma) >0$ such that with probability $1 -  O_\ep(\ep^{\alpha })$, 
\eqb \label{eqn-path-sum-area}
 \sum_{x \in \mcl V \mcl G^\ep( D )} \op{diam}(H_x^\ep)^2 \op{deg}(x ; \mcl G^\ep)  \leq A_0 \left( \op{area}(D) + \ep^{\alpha } \right) . 
\eqe 
It therefore suffices to show that if~\eqref{eqn-path-sum-area} holds, then for an appropriate choice of $A$ as in the statement of the lemma, the estimate~\eqref{eqn-path-sum} holds for every possible choice of $\{P_t\}_{t\in [a,b]}$ and $f^\ep$. 

For such a collection of paths $\{P_t\}_{t\in[a,b]}$ and an edge $\{x,y\} \in  \mcl E\mcl G^\ep(D)$, let $M^\ep(x,y)$ be the Lebesgue measure of the set of $t\in [a,b]$ for which $\{x,y\} \in \mcl E\mcl G^\ep(P_t )$. By interchanging the order of integration and summation, for any $f^\ep : \mcl V\mcl G^\ep(D)\rta \BB R$, 
\eqb \label{eqn-path-sum-count}
 \int_a^b \sum_{\{x,y\} \in \mcl E \mcl G^\ep(P_t) } |f^\ep(x) - f^\ep(y)| \, dt
 \leq   \sum_{\{x,y\} \in \mcl E \mcl G^\ep(D) } |f^\ep(x) - f^\ep(y)| M^\ep(x,y) .
\eqe 
Since the cells $H_x^\ep$ and $H_y^\ep$ intersect whenever $\{x,y\} \in \mcl E\mcl G^\ep$, our hypothesis on the paths $P_t$ implies that if $\{x,y\} \in \mcl E\mcl G^\ep(P_t )$, then $\{x,y\} \notin \mcl E\mcl G^\ep(P_s)$ whenever $|s-t| \geq C \op{diam}(H_x^\ep \cup H_y^\ep)$. Therefore,
\eqb \label{eqn-path-count-diam}
M^\ep(x,y) \leq C (\op{diam}(H_x^\ep) + \op{diam}(H_y^\ep) )  .
\eqe 
By~\eqref{eqn-path-sum-count},~\eqref{eqn-path-count-diam}, and the Cauchy-Schwarz inequality, we see that if~\eqref{eqn-path-sum-area} holds, then 
\allb \label{eqn-path-sum-holder}
 \int_a^b \sum_{\{x,y\} \in \mcl E \mcl G^\ep(P_t) } |f^\ep(x) - f^\ep(y)| \, dt
&\preceq  C \left( \sum_{\{x,y\} \in \mcl E \mcl G^\ep(D )}  (\op{diam}(H_x^\ep)^2 + \op{diam}(H_y^\ep)^2 ) \right)^{1/2}  \op{Energy}\left(  f^\ep  ; \mcl G^\ep(D) \right)^{1/2}      \notag \\
&\preceq  C  \left( \sum_{x \in \mcl V \mcl G^\ep(D )}     \op{diam}(H_x^\ep)^2 \op{deg}(x ; \mcl G^\ep)         \right)^{1/2}     \op{Energy}\left(  f^\ep  ; \mcl G^\ep(D) \right)^{1/2} \notag \\
&\leq     C A_0^{1/2} \left( \op{area}(D) + \ep^{\alpha_0}  \right)^{1/2}    \op{Energy}\left(  f^\ep  ; \mcl G^\ep(D) \right)^{1/2} ,
\alle
with universal implicit constants. Thus~\eqref{eqn-path-sum} holds for an appropriate choice of $A$. 
\end{proof}

The following lemma says that the random walk on $\mcl G^\ep$ started close to the inner boundary of a Euclidean annulus is likely to hit the inner boundary before the outer boundary. 
The lemma will be a consequence of Lemma~\ref{lem-path-sum} applied to the discrete harmonic function which equals 0 on the inner boundary and 1 on the outer boundary.

\begin{lem} \label{lem-annulus-cross}
For each $\rho\in (0,1)$, there exists $\alpha = \alpha(\gamma) > 0$, $\beta = \beta(\gamma) > 0$, and $A = A( \rho,\gamma) > 0$ 
such that for each $r\in [\ep^\beta ,\rho]$, each $\ep\in (0,1)$, and each $z\in B_{\rho}(0)$ such that $B_r(z) \subset B_{\rho}(0)$, it holds with probability at least $1-O_\ep(\ep^\alpha)$ (at a rate depending only on $\rho$ and $\gamma$) that for each $s\in [\ep^\beta ,1/10]$, 
\eqb \label{eqn-annulus-cross}
\min_{x\in \mcl V\mcl G^\ep( B_{4 s r}(z)  )} \ol{\BB P}^\ep_x \left[ \text{$X^\ep$ hits $\mcl V\mcl G^\ep(\bdy B_{s r}(z) )$ before $\mcl V\mcl G^\ep(\bdy B_r(z) ) $} \right] \geq 1 - \frac{A}{\sqrt{\log s^{-1} } } .
\eqe 
\end{lem}
\begin{proof}
See Figure~\ref{fig-pos-prob}, left panel, for an illustration of the proof.
Fix $z\in B_{\rho}(0)$ and $r\in (0,\rho]$ with $B_r(z) \subset B_{\rho}(0)$. To lighten notation, define the open annulus
\eqbn
\BB A_{a,b}(z) := B_b(z) \setminus \ol{B_a(z)} ,\quad \forall b > a > 0. 
\eqen
For $s \in (0,1/2]$ and $\ep\in (0,1)$, let $\frk f_s^\ep : \mcl V\mcl G^\ep( \BB A_{s r , r}(z) ) \rta [0,1]$ be the function which equals 0 on $\mcl V\mcl G^\ep(\bdy B_{s r}(z) )$, 1 on $\mcl V\mcl G^\ep(\bdy B_r(z) )$, and is discrete harmonic on the rest of $\mcl V\mcl G^\ep( \BB A_{sr , r}(z) )$. We need an upper bound for the values of $\frk f_s^\ep$ on $\mcl V\mcl G^\ep(\BB A_{sr , 4sr}(z))$. 
\medskip

\noindent\textit{Step 1: Dirichlet energy bound.} We first bound the Dirichlet energy of $\frk f_{s }^\ep$ . By Theorem~\ref{thm-dirichlet-harmonic}, applied to the Euclidean harmonic function  on $\BB A_{sr , r}(z)$ which equals 0 on $\bdy B_{s r}(z)$ and 1 on $\bdy B_r(z)$, and the same argument used in the proof of Lemma~\ref{lem-annulus-energy}, we find that there exists $\alpha_0 = \alpha_0(\gamma) > 0$ and $A_0 = A_0(\rho,\gamma) > 0$ such that for each $\ep \in (0,1)$ and each fixed $s\in [0 , 1/2]$ it holds with probability at least $1-O_\ep(\ep^{\alpha_0})$ (at a rate depending only on $\rho$ and $\gamma$) that 
\eqb \label{eqn-annulus-dirichlet}
\op{Energy}\left( \frk f_s^\ep ;  \mcl V\mcl G^\ep( \BB A_{s r , r}(z) )  \right) \leq   \frac{  A_0}{  \log s^{-1}   }  .
\eqe
\medskip

\noindent\textit{Step 2: averaging over segments and circles.} 
We will now apply Lemma~\ref{lem-path-sum} to two different collections of paths to deduce~\eqref{eqn-annulus-cross} from~\eqref{eqn-annulus-dirichlet}. 
For $s\in  [\ep^{\beta_0} ,1/10]$ define the concentric circles
\eqbn
P_{s,t}^1 :=  \bdy B_{5 (s r + t) }(z) ,\quad \forall t \in \left[ 0 ,   s r \right] 
\eqen
and the radial line segments across $\BB A_{sr, 10 sr}(z)$ 
\eqbn
P_{s,t}^2 := \left[ s r \exp\left( 2\pi \BB i ( s r)^{-1} t \right) + z , 10 s  r \exp\left( 2\pi \BB i (s r)^{-1}  t \right) + z \right]
\quad \forall t \in \left[0 , s r \right] .
\eqen 
We observe that for each $t \in [0, s r ]$, each of $P_{s,t}^1$ and $P_{s,t}^2$ is contained in $\ol{\BB A_{  s r,  10 s r}(z)}$.
Furthermore, there is a universal constant $C>0$ such that for $t,t' \in [0,  sr]$ and $i\in \{1,2\}$ we have $\op{dist}(P_{s,t}^i,P_{s,t'}^i) \geq C^{-1} |t-t'|$.

By Lemma~\ref{lem-path-sum} (applied with $D =  \BB A_{  s r, 10 s r}(z)$ and each of the collections of paths $\{P_{s,t}^i\}_{t\in [0,  s r]}$ for $i\in \{1,2\}$) and~\eqref{eqn-annulus-dirichlet}, we can find constants $\alpha_1 = \alpha_1(\gamma) \in (0,\alpha_0 ]$ and $A_1 = A_1( \rho,\gamma) > 0$ such that for each $\ep\in (0,1)$ and each $s \in [\ep^{\alpha_1/2} , 1 /10]$ it holds with probability $1 - O_\ep(\ep^{\alpha_1})$ that for each $i\in \{1,2\}$, 
\allb \label{eqn-annulus-cross-mean}
\frac{1}{ s r} \int_0^{s r} \sum_{\{x,y\} \in \mcl E \mcl G^\ep(P_t ) } | \frk f_s^\ep (x) -  \frk f_s^\ep (y)|   \, dt
&\leq \frac{A_1  ( ( s r)^2  + \ep^{\alpha_1}  )^{1/2} }{ s r}   \op{Energy}\left(   \frk f_s^\ep  ; \mcl G^\ep( \BB A_{  s r , 10 s r}(z) ) \right)^{ 1/2}  \notag\\
&\leq  \frac{A_2}{ \sqrt{ \log s^{-1}  }} 
\alle
for some constant $A_2 > 0$ depending only on $\rho$ and $\gamma$. 
\medskip

\noindent\textit{Step 3: existence of good paths.} 
If~\eqref{eqn-annulus-cross-mean} holds for each $i\in \{1,2\}$, then since the average of any function is bounded below by its minimum value, there must exist $t_i^\ep \in [0, s r]$ such that with $P_s^{i,\ep} := P_{s,t_i^\ep}^\ep$, 
\eqb \label{eqn-annulus-cross-sum}
\sum_{\{x,y\} \in \mcl E \mcl G^\ep(P_s^{i,\ep}  ) } | \frk f_s^\ep (x) -  \frk f_s^\ep (y)|  \leq \frac{A_2}{ \sqrt{ \log s^{-1} } } .
\eqe 
The union $P_{s}^{1,\ep} \cup P_{s}^{2,\ep}$ is connected, intersects $\bdy B_{ s r}(z)$, and disconnects $\BB A_{s r , 5 s r }(z)$ from $\infty$. Since $\frk f_s^\ep$ vanishes on $\mcl V\mcl G^\ep(\bdy B_{s r}(z) )$, we infer from~\eqref{eqn-annulus-cross-sum} and the maximum principle for the discrete harmonic function $\frk f_s^\ep$ that with probability at least $1-O_\ep(\ep^{\alpha_1})$, 
\eqbn
\max_{x\in \mcl V\mcl G^\ep(\BB A_{ s r  , 5 s r}(z) )} \frk f_s^\ep(x) \leq \frac{2A_2}{ \sqrt{ \log s^{-1} } } , 
\eqen 
which by the definition of $\frk f_s^\ep$ implies 
\eqb \label{eqn-annulus-cross0}
\min_{x\in \mcl V\mcl G^\ep(\BB A_{ 5 s r}(z) )} \ol{\BB P}^\ep_x \left[ \text{$X^\ep$ hits $\mcl V\mcl G^\ep(\bdy B_{s r}(z) )$ before $\mcl V\mcl G^\ep(\bdy B_r(z ) ) $} \right] \geq 1 - \frac{2A_2}{\sqrt{\log s^{-1} } } .
\eqe 
The statement of the lemma with $\alpha$ slightly smaller than $\alpha_1$, $\beta  =  \alpha_1/2   $, and $A = 4A_2$ follows by applying~\eqref{eqn-annulus-cross0} for a collection of $O_\ep(\log\ep^{-1})$ different values of $s\in [\ep^\beta , 1/10] $, chosen so that each interval $[s , 4s]$ for $s\in [\ep^\beta , 1/4]$ is contained in $[s ' , 5 s']$ for some $s'$ in this collection, then taking a union bound (note that this last step is why we used $5s$ instead of $4s$ above). 
\end{proof}

\begin{proof}[Proof of Proposition~\ref{prop-pos-prob}]
We will iteratively apply Lemma~\ref{lem-annulus-cross} and the Markov property of the random walk to a collection of balls which cover $B_{\wt r}(P)$. 
Let $\alpha = \alpha(\gamma)>0$, $\beta=\beta(\gamma) > 0$, and $A=A(\rho,\gamma) > 0$ be as in Lemma~\ref{lem-annulus-cross} and let $s_0 \in (0,1/10]$ be chosen so that such that $A / \sqrt{\log s^{-1}} \leq 1/2$. To lighten notation, set $r' := s_0(r-\wt r)$. 

By~\eqref{eqn-pos-prob-N}, for each $P, \ep , \wt r , r$ as in the statement of the lemma we can find a finite deterministic set $W \subset B_{\wt r + r'}(P)$ such that
\eqb \label{eqn-finite-ball-cover}
\# W \leq N \quad \op{and} \quad  B_{\wt r}(P) \subset \bigcup_{w\in W} B_{r' }(w) .
\eqe  
By Lemma~\ref{lem-annulus-cross} (applied with $r-\wt r$ in place of $r$) and a union bound over all $w\in W$, it holds with probability at least $1-N O_\ep(\ep^\alpha)$ that
\eqb \label{eqn-all-ball-pos}
\min_{w\in W} \min_{x\in \mcl V\mcl G^\ep( B_{4r'}(w) )} \ol{\BB P}^\ep_x \left[ \text{$X^\ep$ hits $\mcl V\mcl G^\ep(\bdy B_{r'}(w) )$ before $\mcl V\mcl G^\ep(\bdy B_{r-\wt r}(w) ) $} \right] \geq \frac12 .
\eqe

Suppose now that~\eqref{eqn-all-ball-pos} holds. By~\eqref{eqn-finite-ball-cover}, for each $x\in \mcl V\mcl G^\ep(B_{\wt r}(P))$ and each $z \in B_{\wt r}(P)$ we can find distinct points $w_0,\dots,w_m \in W$ such that $H_x^\ep\cap  B_{r'}(w_0) \not=\emptyset$, $z \in  B_{r'}(w_m)$, and $B_{r'}(w_{k-1})\cap B_{r'}(w_k)\not=\emptyset$ for each $k\in [1,m]_{\BB Z}$. 
Since $s_0 \leq 1/10$ and each $w \in W$ lies in $B_{\wt r + r'}(P)$, each ball $B_{4r'}(w_k)$ for $k\in [1,m]_{\BB Z}$ is contained in $B_r(P)$.
Moreover, each $B_{r'}(w_{k-1})$ is contained in $B_{4r'}(w_k)$. 
By $m$ applications of~\eqref{eqn-all-ball-pos} and the Markov property of $X^\ep$, it holds with $\ol{\BB P}_x^\ep$-probability at least $2^{-m} \geq 2^{-N}$ that $X^\ep$ enters each $\mcl V\mcl G^\ep( B_{r'}(w_k) )$ before leaving $\mcl V\mcl G^\ep(B_r(P))$, in which case $X^\ep$ enters $\mcl V\mcl G^\ep(B_r(z)) \supset \mcl V\mcl G^\ep(B_{r'}(w_m))$ before leaving $\mcl V\mcl G^\ep(B_r(P))$. Thus~\eqref{eqn-pos-prob} holds. 
\end{proof}

\subsection{H\"older continuity for harmonic functions on $\mcl G^\ep$}
\label{sec-cont}

In this subsection, we use Proposition~\ref{prop-pos-prob} to deduce a uniform ($\ep$-independent) H\"older continuity estimate for harmonic functions on $\mcl G^\ep$, which is a more quantitative version of Theorem~\ref{thm-cont0}
 
\begin{thm} \label{thm-cont}
For each $\rho \in (0,1)$, there exists $\alpha  = \alpha(\gamma) > 0$, $\xi=\xi(\rho, \gamma) > 0$, and $A = A(\rho,\gamma)> 0$ such that for each $\ep\in (0,1)$, the following holds with probability at least $1-O_\ep(\ep^\alpha)$. Let $D\subset B_{\rho}(0)$ be a connected domain and let $\frk f^\ep : \mcl V\mcl G^\ep(D) \rta \BB R$ be discrete harmonic on $\mcl V\mcl G^\ep(D)\setminus \mcl V\mcl G^\ep(\bdy D)$. Then we have the interior continuity estimate
\eqb \label{eqn-interior-cont}
|\frk f^\ep(x) - \frk f^\ep(y)| \leq A \|\frk f^\ep\|_\infty  \left(    \frac{\ep\vee |\eta(x) - \eta(y)|}{\op{dist}(\eta(x), \bdy D)} \right)^\xi ,\quad \forall x,y\in \mcl V\mcl G^\ep(D) ,
\eqe 
where $\|\frk f^\ep\|_\infty$ denotes the $L^\infty$ norm.  

If $D$ is simply connected, then we also have the boundary continuity estimate
\allb \label{eqn-bdy-cont}
& \min_{y \in   \mcl V\mcl G^\ep(\bdy D\cap B_{t}(\eta(x)) )} \frk f^\ep(y) - A  \| \frk f^\ep \|_\infty \left( \frac{t}{\op{dist}(\eta(x) ,\bdy D)} \right)^{-\xi} \leq \frk f^\ep(x)  \notag \\
& \qquad 
  \leq \max_{y \in   \mcl V\mcl G^\ep(\bdy D\cap B_{t}(\eta(x)) )} \frk f^\ep(y) + A \| \frk f^\ep \|_\infty \left( \frac{t}{\op{dist}(\eta(x) ,\bdy D)} \right)^{-\xi} ,  \quad \forall x \in \mcl V\mcl G^\ep(D), \quad \forall t > 0  .
\alle
In particular, if $\frk f^\ep$ has H\"older continuous boundary data, in the sense that there exists $\chi \in (0,1]$ and $C >0$ such that
\eqb \label{eqn-holder-cont-hyp}
|\frk f^\ep(x) - \frk f^\ep(y)| \leq C \left(\ep \vee |\eta(x) - \eta(y)| \right)^\chi ,\quad \forall  x,y\in \mcl V\mcl G^\ep(\bdy D)
\eqe 
then 
\eqb \label{eqn-holder-cont}
|\frk f^\ep(x) - \frk f^\ep(y)| \leq \max\left\{ C , A\|\frk f^\ep\|_\infty \right\} \left(\ep \vee |\eta(x) - \eta(y)| \right)^{(\xi \wedge\chi)/4} ,\quad \forall  x,y\in \mcl V\mcl G^\ep( D).
\eqe
\end{thm}

We note that~\eqref{eqn-holder-cont} immediately implies Theorem~\ref{thm-cont0}. 
The basic idea of the proof of Theorem~\ref{thm-cont} is to bound the total variation distance between the conditional laws given $(h,\eta )$ of the positions where the simple random walks on $\mcl G^\ep$ started at two nearby vertices of $\mcl V\mcl G^\ep(D)$ first hit $\mcl V\mcl G^\ep(\bdy D)$. This is sufficient to establish a modulus of continuity bound since we are working with discrete harmonic functions. Our bound for total variation distance will be established using the following lemma, which is an easy consequence of Wilson's algorithm (see~\cite[Lemma 3.12]{gms-random-walk} for a proof).

\begin{lem}[\!\!\cite{gms-random-walk}] \label{lem-disconnect-coupling}
Let $G$ be a connected graph and let $A\subset \mcl V(G)$ be a set such that the simple random walk started from any vertex of $G$ a.s.\ hits $A$ in finite time.
For $x\in \mcl V(G)$, let $X^x$ be the simple random walk started from $x$ and let $\tau^x$ be the first time $X^x$ hits $A$. 
For $x,y\in \mcl V(G) \setminus A$,
\eqb \label{eqn-disconnect-coupling}
\BB d_{\op{TV}}\left(X^x_{\tau^x} , X^y_{\tau^y} \right) \leq  1 - \BB P\left[ \text{$X^x$ disconnects $y$ from $A$ before time $\tau^x$} \right] ,
\eqe 
where $\BB d_{\op{TV}}$ denotes the total variation distance.
\end{lem} 

To apply Lemma~\ref{lem-disconnect-coupling} in our setting, we need a lower bound for the probability that simple random walk on $\mcl G^\ep$ surrounds a nearby point before traveling a long distance. Proposition~\ref{prop-pos-prob} tells us that random walk on $\mcl G^\ep$ has uniformly positive probability to surround the inner boundary of an annulus of fixed aspect ratio before hitting the outer boundary (see Lemma~\ref{lem-pos-disconnect}). Iterating this over dyadic annuli and applying Lemma~\ref{lem-disconnect-coupling} will then give us a polynomial upper bound on the total variation distance between the hitting distributions for random walk started from two nearby vertices of $\mcl G^\ep$. 
For the statement of the next lemma, we recall the notation $\ol{\BB P}_x^\ep$ from Definition~\ref{def-walk-pt}.

\begin{lem} \label{lem-pos-disconnect}
For each $\rho \in (0,1)$, there exists $\alpha=\alpha(\gamma) > 0$, $\beta = \beta(\gamma) > 0$, and $p = p(\rho,\gamma)  > 0$ such that for each $\ep \in (0,1)$, the following holds with probability $1-O_\ep(\ep^\alpha)$. For each $z\in B_{\rho}(0)$ and each $r \in [\ep^\beta ,\rho]$ with $B_{2r}(z) \subset B_{\rho}(0)$,
\allb \label{eqn-pos-disconnect} 
&\min_{x\in \mcl V\mcl G^\ep(B_r(z))} 
\ol{\BB P}_x^\ep\big[ \text{$X^\ep$ disconnects $\mcl V\mcl G^\ep(\bdy B_{r/2}(z))$ from $\mcl V\mcl G^\ep(\bdy B_{2r}(z))$} \notag \\
& \qquad\qquad\qquad\qquad\qquad\qquad \text{before hitting $\mcl V\mcl G^\ep(\bdy B_{2r}(z))$} \big] 
\geq p .
\alle
\end{lem}
\begin{proof}
If $X^\ep$ starts inside $\mcl V\mcl G^\ep(B_r(z))$, then $X^\ep$ must hit $\mcl V\mcl G^\ep(\bdy B_r(z))$ before hitting $\mcl V\mcl G^\ep\bdy B_{2r}(z))$.
By the strong Markov property of $X^\ep$, it suffices to prove~\eqref{eqn-pos-disconnect} with the minimum taken over $\mcl V\mcl G^\ep(\bdy B_r(z))$ instead of $\mcl V\mcl G^\ep(B_r(z))$.
Let $\alpha_0 =\alpha_0(\gamma) > 0$ and $\beta_0 = \beta_0(\gamma) > 0$ be chosen so that the conclusion of Proposition~\ref{prop-pos-prob} is satisfied. 
It is easily seen from Proposition~\ref{prop-pos-prob}, applied with $r/4$ in place of $r$, $\wt r = r/8$, say, and $P$ equal to each of three appropriately chosen arcs of $\bdy B_r(z)$ that there is a $p = p(\rho,\gamma) > 0$ such that for each fixed $z\in B_{\rho}(0)$ and each fixed $r \in [ \ep^{\beta_0} ,\rho]$ with $B_{5r/4}(z) \subset B_{\rho}(0)$, it holds with probability at least $1-O_\ep(\ep^{\alpha_0})$ that  
\allb 
 \min_{x\in\mcl V\mcl G^\ep( \bdy B_r(z)   )}  
&\ol{\BB P}_x^\ep\big[ \text{$X^\ep$ disconnects $\mcl V\mcl G^\ep(\bdy B_{3r/4}(z))$ from $\mcl V\mcl G^\ep(\bdy B_{5r/4}(z))$}\notag \\ 
&\qquad \qquad \qquad \qquad \qquad  \text{before hitting $\mcl V\mcl G^\ep(\bdy B_{5r/4}(z))$}  \big]  \geq p. 
\alle
The statement of the lemma for a small enough choice of $\alpha \in (0,\alpha_0)$ and $\beta \in (0,\beta_0)$ follows from this last estimate by taking a union bound over all $z\in (\ep^{\alpha_0/100} \BB Z^2 )\cap B_{\rho}(0)$ and all $r\in  \ep^{\alpha_0/100} \BB Z $ with $r \geq \ep^{\beta_0}$ and $B_{2r}(z) \subset B_{\rho}(0)$. 
\end{proof}

\begin{proof}[Proof of Theorem~\ref{thm-cont}]
Let $\alpha  , \beta ,$ and $p$ be chosen so that the conclusion of Lemma~\ref{lem-pos-disconnect} is satisfied for $\ep > 0$. We can assume without loss of generality that $\beta < \frac{1}{50 (2+\gamma)^2}$, so that Lemma~\ref{lem-max-cell-diam} applies with $q = 100\beta$. Let $E^\ep$ be the event that~\eqref{eqn-pos-disconnect} holds for each $z\in B_{\rho}(0)$ and each $r \in [\ep^\beta ,\rho]$ with $B_{2r}(z) \subset B_{\rho}(0)$ and that $\op{diam}(H_x^\ep) \leq \ep^{100\beta}$ for each $x\in\mcl V\mcl G^\ep(B_{\rho}(0))$, so that by Lemma~\ref{lem-pos-disconnect} and Lemma~\ref{lem-max-cell-diam}, after possibly shrinking $\alpha$ we can arrange that $\BB P[E^\ep] = 1-O_\ep(\ep^\alpha)$. 
Throughout the proof we assume that $E^\ep$ occurs and we let $\frk f^\ep : \mcl V\mcl G^\ep(D) \rta \BB R$ be a discrete harmonic function as in the theorem statement. 

Applying the Markov property of the walk $X^\ep$ and the estimate~\eqref{eqn-pos-disconnect} with $r  =e^{-k}$ for each $k\in [\log \lfloor \op{dist}(\eta(x) , \bdy D)^{-1} \rfloor , \log \lceil (2 (\ep^\beta \vee |\eta(x) - \eta(y)| ) )^{-1} \rceil ]_{\BB Z}$, then multiplying over all such $k$, shows that for each $x,y \in \mcl V\mcl G^\ep(D)$, 
\allb \label{eqn-pos-iterate}
&\ol{\BB P}_x^\ep\left[ \text{$X^\ep$ disconnects $y$ from $\mcl V\mcl G^\ep(\bdy D)$ before hitting $\mcl V\mcl G^\ep(\bdy D)$} \right] \notag \\
&\qquad \qquad \geq 1 - p^{\log \tfrac{ \op{dist}(\eta(x), \bdy D) }{ \ep^\beta \vee |\eta(x) - \eta(y)|} - 2}  \geq 1 - A \left(    \frac{\ep\vee |\eta(x) - \eta(y)|}{\op{dist}(\eta(x), \bdy D)} \right)^\xi
\alle
for constants $A > 0$ and $\xi > 0$ depending only on $p$ and $\beta$ (and hence only on $\rho$ and $\gamma$). Note that we have absorbed $\beta$ into $\xi$. 
By~\eqref{eqn-pos-iterate} and Lemma~\ref{lem-disconnect-coupling}, we find that the total variation distance between the $\ol{\BB P}_x^\ep$-law of the first place where $X^\ep$ hits $\mcl V\mcl G^\ep(\bdy D)$ and the $\ol{\BB P}_y^\ep$-law of the first place where $X^\ep$ hits $\mcl V\mcl G^\ep(\bdy D)$ is at most the right side of~\eqref{eqn-pos-iterate}. Since $\frk f^\ep$ is discrete harmonic, this implies~\eqref{eqn-interior-cont}. 

Now assume that $D$ is simply connected.
To prove the boundary estimate~\eqref{eqn-bdy-cont}, we observe that if $x\in \mcl V\mcl G^\ep(D)$ and $r \geq 2 \op{dist}(\eta(x) ,\bdy D)$, then since $\BB C\setminus D$ is connected, in order for a random walk started from $x$ to disconnect $\mcl V\mcl G^\ep(B_{r/2}(0))$ from $\infty$, it must first hit $\mcl V\mcl G^\ep(\bdy D)$. By applying the Markov property of $X^\ep$ and~\eqref{eqn-pos-disconnect} with $r = e^{-k}$ for $k\in \left[\log \lceil t^{-1} \rceil  ,  \log \lfloor 2  \op{dist}(\eta(x) ,\bdy D)^{-1} \rfloor  \right]$, we get that for $t >2 \op{dist}(\eta(x) , \bdy D)$, 
\allb  
&\ol{\BB P}_x^\ep\left[ \text{$X^\ep$ first hits $\mcl V\mcl G^\ep(\bdy D)$ at a point in $\mcl V\mcl G^\ep(\bdy D\cap B_{t}(\eta(x) ) )$} \right] \notag \\
&\qquad \qquad \geq 1 - p^{\tfrac{t}{\op{dist}(\eta(x) ,\bdy D)}  - 2}  \geq 1 - A \left( \frac{t}{\op{dist}(\eta(x) ,\bdy D)} \right)^{-\xi}  
\alle
for a possibly larger choice of $A$ and smaller choice of $\xi$. Again using that $\frk f^\ep$ is discrete harmonic, we obtain~\eqref{eqn-bdy-cont} (the when case $t< 2 \op{dist}(\eta(x) , \bdy D)$ can be dealt with by increasing $A$). 
  
Assume now that the boundary H\"older continuity condition~\eqref{eqn-holder-cont-hyp} holds. We will deduce~\eqref{eqn-holder-cont} from~\eqref{eqn-interior-cont} and~\eqref{eqn-bdy-cont}. If $x,y\in \mcl V\mcl G^\ep(D)$ with $  \op{dist}(\eta(x), \bdy D) \leq (\ep \vee |\eta(x) - \eta(y)| )^{1/2} $, then by~\eqref{eqn-interior-cont} we get $|\frk f^\ep(x) - \frk f^\ep(y)| \leq A \|\frk f^\ep\|_\infty (\ep\vee |\eta(x) - \eta(y)| )^{\xi/2}$. On the other hand, if $ \op{dist}(\eta(x), \bdy D) \leq (\ep \vee |\eta(x) - \eta(y)|)^{1/2}$ then also $ \op{dist}(\eta(y), \bdy D) \leq 2 (\ep \vee |\eta(x) - \eta(y)|)^{1/2}$. By applying~\eqref{eqn-bdy-cont} to each of $x$ and $y$ and using~\eqref{eqn-holder-cont-hyp} to estimate the boundary terms, we get that for $t>0$, 
\eqb \label{eqn-use-bdy-holder}
|\frk f^\ep(x) - \frk f^\ep(y)| \leq  C t^\chi  + A  \| \frk f^\ep \|_\infty \left( \frac{t}{(\ep \vee |\eta(x) - \eta(y)|)^{1/2} } \right)^{-\xi} . 
\eqe 
Choosing $t = (\ep \vee |\eta(x) - \eta(y)|)^{1/4}$ and combining with our earlier estimate for the case when $  \op{dist}(\eta(x), \bdy D) \leq (\ep \vee |\eta(x) - \eta(y)| )^{1/2} $ gives~\eqref{eqn-holder-cont}. 
\end{proof}

\section{Estimates for the area, diameter, and degree of a cell}
\label{sec-area-diam-deg}

The goal of this section is to prove Proposition~\ref{prop-area-diam-deg-gamma}.
Throughout most of this section, we restrict attention to the case when $h$ is either the circle-average embedding of a 0-quantum cone or a whole-plane GFF normalized so that $h_1(0) = 0$. Note that the structure graph corresponding to such a choice of $h$ is \emph{not} the same as the mated-CRT map. We will transfer to the $\gamma$-quantum cone case, which corresponds to the mated-CRT map, in Section~\ref{sec-gamma-cone-proof}.
Throughout, we define the cell $H_{x_z^\ep}^\ep$ containing a fixed point $z\in\BB C$ as in Section~\ref{sec-standard-setup} and we define $u^\ep(z)$ as in~\eqref{eqn-area-diam-deg-def}.

Most of this section is devoted to the proof of the following variant of Proposition~\ref{prop-area-diam-deg-gamma} for a 0-quantum cone or whole-plane GFF which does not assume any continuity conditions for $f$ or $D$. Proposition~\ref{prop-area-diam-deg-gamma} (which we recall is a statement about the $\gamma$-quantum cone) will be deduced from this proposition and an absolute continuity argument in Section~\ref{sec-gamma-cone-proof}. 

\begin{prop} \label{prop-area-diam-deg-lln}
Suppose we are in the setting of Section~\ref{sec-standard-setup} with $h$ equal to either the circle-average embedding of a 0-quantum cone or a whole-plane GFF normalized so that $h_1(0) = 0$.
For each $\rho \in (0,1)$, there are constants $A =A(\rho, \gamma) > 0$ and $\alpha = \alpha(\gamma)  >0$ such that for each $\ep \in (0,1)$, each bounded measurable function $f :  B_{\rho}(0) \rta [0,\infty)$, and each Borel measurable set $D\subset B_{\rho}(0) $, 
\eqb \label{eqn-area-diam-deg-lln}
\BB P\left[ \int_{D} f(z) u^\ep(z) \, dz  \leq A \int_{D} f(z) \,dz + \ep^\alpha \|f\|_\infty \right] \geq 1 - O_\ep(\ep^\alpha) 
\eqe 
at a rate depending only on $\rho$ and $\gamma$. 
\end{prop}

The reason why we first prove the statement for the 0-quantum cone given in Proposition~\ref{prop-area-diam-deg-lln} is as follows.
For a 0-quantum cone, the origin is a ``Lebesgue typical point"; in particular, there is no log singularity. 
Hence, estimates for $u^\ep(0)$ can be transferred to estimates for $u^\ep(z)$ for a deterministic point $z\in B_\rho(0)$ (or a point sampled uniformly from Lebesgue measure on $B_\rho(0)$); see Lemma~\ref{lem-moment-given-circle-avg}. 
This will allow us to apply the bounds for $\op{diam}^2(H_0^\ep) / \op{area}(H_0^\ep)$ and $\op{deg}(H_0^\ep)$ in the case of the 0-quantum cone from~\cite[Section 4]{gms-tutte} to estimate the integral appearing in~\eqref{eqn-area-diam-deg-lln}. 
At one point in the proof of Proposition~\ref{prop-area-diam-deg-lln} (in particular, Lemma~\ref{lem-max-deg-in}), we will need to prove an estimate for the $\gamma$-quantum cone, then transfer to the 0-quantum cone. The reason for this is that we want to use the degree bound of Lemma~\ref{lem-exp-tail}, which is proven using the Brownian motion definition of the mated-CRT map.


Throughout this section, for $z\in\BB C$ and $r\geq 0$, we let 
\allb \label{eqn-sle-hit-time}
 \tau_z := \inf\left\{ t \in \BB R : \eta(t) = z \right\} 
\alle 
and note that $\tau_z \in (x_z^\ep -\ep , x_z^\ep]$.

\subsection{Outline of the proof}
\label{sec-lln-outline}

Here we give an outline of the content of the rest of this section.

Throughout this section, we will work with ``localized" versions of $  \op{diam}(H_{x_z^\ep}^\ep )^2 / \op{area}(H_{x_z^\ep}^\ep )$ and $ \op{deg}(x_z^\ep ; \mcl G^\ep)$ for $z\in \BB C$ and $\ep > 0$ which we call $\op{DA}^\ep(z)$, $\op{deg}_{\op{in}}^\ep(z)$, and $\op{deg}_{\op{out}}^\ep(z)$, such that
\eqb \label{eqn-localized-compare0}
\frac{ \op{diam}(H_{x_z^\ep}^\ep )^2}{\op{area}(H_{x_z^\ep}^\ep )} \leq \op{DA}^\ep(z) ,\quad
 \op{deg}(x_z^\ep ; \mcl G^\ep) \leq 2 \left( \op{deg}_{\op{in}}^\ep(z) + \op{deg}_{\op{out}}^\ep(z) \right) ,
\eqe 
and, crucially, all three random variables depend locally on $h$ and $\eta$ (see Lemma~\ref{lem-u_*-msrble}). 
 
\begin{remark} \label{remark-cell-nonlocal}
Structure graph cells are not locally determined by $h$ and $\eta$. Indeed, if $z\in \BB C$ then the cell $H_{x_z^\ep}^\ep$ of $\mcl G^\ep$ containing $z$ is only determined locally modulo an index shift: if we only see the behavior of $ h$ and $\eta$ in an open set $U$ containing $z$, then a priori $H_{x_z^\ep}^\ep$ could be any segment of the form $\eta([t-\ep ,t])$ which contains $z$. There is an $\ep$-length interval of times $t$ for which this is the case. 
\end{remark}
 
The reason why we need things to depend locally on $h$ and $\eta$ is that in Section~\ref{sec-area-diam-deg-lln}, we will use long-range independence estimates for SLE and the GFF to get a second moment bound for the integral appearing in Proposition~\ref{prop-area-diam-deg-lln}. 

The localized quantities appearing in~\eqref{eqn-localized-compare0} are defined in Section~\ref{sec-area-diam-deg-localize} and illustrated in Figure~\ref{fig-deg-localize}.  
The quantity $\op{DA}^\ep(z)$ is the maximum of the ratio of square diameter to area over all $\ep$-length segments of $\eta$ contained in $\eta([\tau_z-\ep ,\tau_z+\ep])$ (one of which is equal to $H_{x_z^\ep}^\ep$). The localized degree is split into two parts: the ``inner degree" $\op{deg}_{\op{in}}^\ep(z)$, which counts the number of $\ep$-length segments of $\eta$ contained in the ball centered at $z$ with radius $4\op{diam}( \eta([\tau_z-\ep , \tau_z+\ep]) )$ (here we use the notation~\eqref{eqn-sle-hit-time}); and the ``outer degree" $\op{deg}_{\op{out}}^\ep(z)$, which counts the number of segments of $\eta$ which intersect both $  \eta([\tau_z-\ep , \tau_z+\ep]) $ and the boundary of this ball. 

In Section~\ref{sec-area-diam-deg-moment}, we state $\ep$-independent moment bounds for the above three quantities which were proven in~\cite{gms-tutte}. 
 
In Section~\ref{sec-diam-reg}, we prove a global regularity estimate (Proposition~\ref{prop-diam-reg}) which bounds the maximum over all $z\in B_{\rho}(0)$ of the localized versions of $  \op{diam}(H_{x_z^\ep}^\ep )$ and $ \op{deg}(x_z^\ep ; \mcl G^\ep)$.  
The bound for $  \op{diam}(H_{x_z^\ep}^\ep )$ follows from Lemma~\ref{lem-max-cell-diam}, but the bound for the localized degree will take a bit more work since Lemma~\ref{lem-max-deg-nonlocal} only provides a bound for the non-localized degree and only applies in the $\gamma$-quantum cone case.

In Section~\ref{sec-area-diam-deg-lln}, we prove Proposition~\ref{prop-area-diam-deg-lln} by, roughly speaking, using the moment bounds of Section~\ref{sec-area-diam-deg-moment} to bound the expectation of $\int_D f(z) u^\ep(z) \, dz$ and using long-range independence results for the Gaussian free field from Appendix~\ref{sec-long-range-ind} to show that the variance of $\int_D f(z) u^\ep(z) \, dz$ decays like a positive power of $\ep$. The fact that we use long-range independence for the GFF is the reason why we need to replace $u^\ep(z)$ by a localized version.

In Section~\ref{sec-gamma-cone-proof}, we deduce Proposition~\ref{prop-area-diam-deg-gamma} from Proposition~\ref{prop-area-diam-deg-lln} using the fact that sampling a point $\BB z$ uniformly from the $\gamma$-LQG measure on some domain and re-centering so that $\BB z$ is mapped to 0 produces a field with a $\gamma$-log singularity at 0, which locally looks like a $\gamma$-quantum cone.

\subsection{Localized versions of area, diameter, and degree}
\label{sec-area-diam-deg-localize}

Suppose we are in the setting  of Section~\ref{sec-standard-setup} with $h$ equal to the circle-average embedding of a 0-quantum cone or a $\gamma$-quantum cone, or a whole-plane GFF normalized so that $h_1(0) = 0$.  

In this subsection we will define modified versions of the quantities $  \op{diam}(H_{x_z^\ep}^\ep )^2 / \op{area}(H_{x_z^\ep}^\ep )$ and $  \op{deg}( x_z^\ep ; \mcl G^\ep)$ appearing in the definition~\eqref{eqn-area-diam-deg-def} of $u^\ep(z)$ which are locally determined by $h$ and $\eta$---in the sense of Lemma~\ref{lem-u_*-msrble} just below---and which we will work with throughout most of this subsection (recall from Remark~\ref{remark-cell-nonlocal} that the original quantities are not locally determined by $h$ and $\eta$). The definitions are illustrated in Figure~\ref{fig-deg-localize}.

\begin{figure}[ht!]
\begin{center}
\includegraphics[scale=.8]{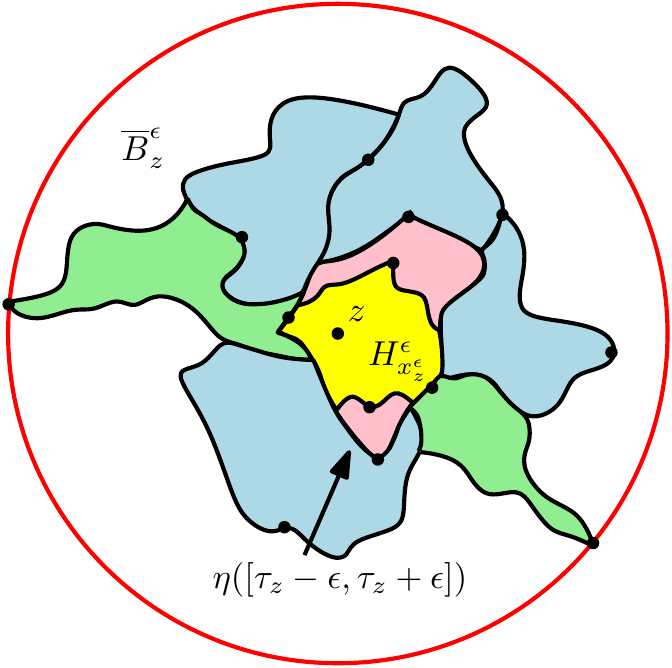} 
\caption{\label{fig-deg-localize} Illustration of the definitions of the localized quantities $\op{DA}^\ep(z)$, $\op{deg}_{\op{in}}^\ep(z)$, and $\op{deg}_{\op{out}}^\ep(z)$. The cell $H_{x_z^\ep}^\ep$ (yellow) is an $\ep$-length segment of $\eta$ contained in $\eta([\tau_z-\ep , \tau_z+\ep])$ (pink). The quantity $\op{DA}^\ep(z)$ is the maximum ratio of squared diameter to area over all of these $\ep$-length segments. The quantity $\op{deg}_{\op{in}}^\ep(z)$ counts the maximum possible number of disjoint $\ep$-length segments of $\eta$ which intersect $\eta([\tau_z-\ep , \tau_z+\ep])$ and are contained in $\ol B_z^\ep$ (disk with red boundary). Several such segments are shown in blue. The quantity $\op{deg}_{\op{out}}^\ep(z)$ counts the maximum possible number of disjoint segments of $\eta$ contained in $\ol B_z^\ep$ with one endpoint in $\eta([\tau_z-\ep ,\tau_z+\ep])$ and the other endpoint in $\bdy \ol B_z^\ep$. Two such segments are shown in green. 
}
\end{center}
\end{figure}

We start with the ratio of squared diameter to area. 
For $z\in \BB C$, let
\eqb \label{eqn-area-diam-localize}
\op{DA}^\ep(z) := \sup\left\{ \frac{\op{diam}\left(\eta([s-\ep , s   ]) \right)^2}{\op{area}\left(\eta([s-\ep , s   ]) \right) } \,:\, s\in\BB R ,\,  \tau_z \in [s-\ep ,s] \right\} ,
\eqe
with $\tau_z$ as in~\eqref{eqn-sle-hit-time}.  
In words, $\op{DA}^\ep(z)$ is a.s.\ equal to the maximum ratio of the squared diameter to the area over all of the segments of $\eta$ with quantum mass $\ep$ which contain $z$, so using $\op{DA}^\ep(z)$ instead of $ \op{diam}(H_{x_z^\ep}^\ep)^2 / \op{area}(H_{x_z^\ep}^\ep)$ removes the arbitrariness coming from the choice to define $\mcl G^\ep$ using elements of $\ep\BB Z  $ rather than $\ep\BB Z+t$ for some $t\in (0,\ep)$. 
By definition, the cell $H_{x_z^\ep}^\ep$ is one of these $\ep$-length segments of $\eta$, so 
\eqb \label{eqn-area-diam-localize-bound}
 \frac{\op{diam}(H_{x_z^\ep}^\ep)^2}{\op{area}(H_{x_z^\ep}^\ep)} \leq   \op{DA}^\ep(z) .
\eqe 

Since the degree depends on more than just the cell itself (unlike the area and the diameter), we need a slightly more complicated definition than~\eqref{eqn-area-diam-localize} to ``localize" the degree. 
Define the closed ball
\eqb \label{eqn-localize-set}
\ol B_z^\ep := \left\{ w \in \BB C : |z-w| \leq 4 \op{diam}\left( \eta([\tau_z-\ep , \tau_z+\ep]) \right) \right\} .
\eqe  
We will define two quantities whose sum provides an upper bound for $\op{deg}(x_z^\ep ; \mcl G^\ep)$. 
\begin{itemize}
\item Let $\op{deg}_{\op{in}}^\ep(z)$ be the largest number $N\in\BB N$ with the following property: there is a collection of $N$ intervals $\{[a_j , b_j]\}_{j\in [1,N]_{\BB Z}}$ which may intersect only at their endpoints, each of which has length $b_j  -a_j = \ep$, satisfies $\eta([a_j ,b_j]) \subset \ol B_z^\ep$, and is such that $\eta([a_j,b_j])\cap \eta([\tau_z-\ep ,\tau_z+\ep]) \not=\emptyset$. 
\item Let $\op{deg}_{\op{out}}^\ep(z)$ be the largest number $N' \in\BB N$ with the following property: there is a collection of $N'$ intervals $\{[a_j , b_j]\}_{j\in [1,N']_{\BB Z}}$ which intersect only at their endpoints such that for each $j\in [1,N']_{\BB Z}$, $\eta((a_j,b_j)) $ is contained in the interior of $\ol B_z^\ep$, one of the endpoints $\eta(a_j)$ or $\eta(b_j)$ is contained in $\bdy \ol B_z^\ep$, and the other endpoint is contained in $\eta([\tau_z-\ep , \tau_z+\ep])$.  
\end{itemize}

Since $H_{x_z^\ep}^\ep \subset \eta([\tau_z-\ep , \tau_z+\ep])$, the set of intervals $[x-\ep , x]$ for $x\in\ep\BB Z  $ such that $x\sim x_z^\ep$ in $\mcl G^\ep$ and $H_x^\ep \subset \ol B_z^\ep$ is a collection as in the definition of $\op{deg}_{\op{in}}^\ep(z)$, so the number of such $x$ is at most $\op{deg}_{\op{in}}^\ep(z)$. Similarly, the number of $x\in\ep\BB Z  $ such that $x$ is joined to $ x_z^\ep$ by an edge in $\mcl G^\ep$ and $H_x^\ep\not\subset \ol B_z^\ep$ is at most $\op{deg}_{\op{out}}^\ep(z)$, since for any such $x$ the cell $H_x^\ep$ contains a different interval $[a_j,b_j]$ as in the definition of $\op{deg}_{\op{out}}^\ep(z)$. Since any two vertices of $\mcl G^\ep$ are connected by at most 2 edges,
\eqb \label{eqn-deg-localize-bound}
\op{deg}\left(x_z^\ep ;\mcl G^\ep \right) \leq 2 \left( \op{deg}_{\op{in}}^\ep(z) +   \op{deg}_{\op{out}}^\ep(z) \right) .
\eqe  

The following lemma is our main reason for introducing the quantities $\op{DA}^\ep(z) $, $\op{deg}_{\op{in}}^\ep(z) $, and $  \op{deg}_{\op{out}}^\ep(z)$. 

\begin{lem} \label{lem-u_*-msrble}
Let $h^{\op{IG}}$ be the whole-plane GFF viewed modulo $2\pi \chi^{\op{IG}}$ which is used to construct $\eta$ in Section~\ref{sec-wpsf-prelim}.
For each open set $V\subset \BB C$, the random variable $ \op{DA}^\ep(z)  \BB 1_{(\ol B_z^\ep\subset V)}$ is a.s.\ determined by $h|_V$ and $h^{\op{IG}}|_V$, where here $\ol B_z^\ep$ is the ball defined in~\eqref{eqn-localize-set}. The same is true with $\op{DA}^\ep(z)$ replaced by either $\op{deg}_{\op{in}}^\ep(z) $ or $  \op{deg}_{\op{out}}^\ep(z)$
\end{lem}
\begin{proof}
For each open set $V\subset \BB C$, the measure $\mu_h$ is determined by $h|_V$; indeed, this follows from the circle average construction of $\mu_h$~\cite{shef-kpz}. From this, Lemma~\ref{lem-wpsf-determined}, and the definitions of $\op{DA}^\ep(z)$, $\op{deg}_{\op{in}}^\ep(z)$, and $\op{deg}_{\op{out}}^\ep(z)$, we obtain the statement of the lemma. 
\end{proof}

\subsection{Moment bounds for localized area, diameter, and degree}
\label{sec-area-diam-deg-moment}

The starting point of our proof is the following theorem, which follows from results in~\cite[Section 4]{gms-tutte}.
 
\begin{thm} \label{thm-moment}
Suppose $h$ is the circle-average embedding of a 0-quantum cone and define $\op{DA}^\ep(0)$, $\op{deg}_{\op{in}}^\ep(0)$, and $\op{deg}_{\op{out}}^\ep(0)$ for $\ep > 0$ as in Section~\ref{sec-area-diam-deg-localize}. Then for $\ep > 0$, 
\eqb \label{eqn-area-diam}
 \BB E\left[     \op{DA}^\ep(0)^p \right] \preceq 1 , \quad \forall p \geq 1 
\eqe 
\eqb \label{eqn-fixed-pt-deg-in}
\BB E\left[   \op{deg}_{\op{in}}^\ep(0)^p \right] \preceq 1 ,\quad  \forall p \in [1,4/\gamma^2)  , \quad \op{and} \quad
\eqe  
\eqb \label{eqn-fixed-pt-deg-out}
\BB E\left[    \op{deg}_{\op{out}}^\ep(0)^p  \right] \preceq 1 ,\quad   \forall p\geq 1 ,
\eqe
with the implicit constants depending only on $p$. 
\end{thm}
\begin{proof}
The bounds~\eqref{eqn-area-diam},~\eqref{eqn-fixed-pt-deg-in}, and~\eqref{eqn-fixed-pt-deg-out} in the case $\ep = 1$ follow from~\cite[Propositions 4.4 and 4.5]{gms-tutte}. We will now use the scaling property of the 0-quantum cone to argue that the laws of $\op{DA}^\ep(0)$, $\op{deg}_{\op{in}}^\ep(0)$, and $\op{deg}_{\op{out}}^\ep(0)$ do not depend on $\ep$. To this end, let $R^\ep  := \sup\left\{ r > 0 :  h_r(0)   +   Q \log r = \frac{1}{\gamma} \log \ep  \right\} $, where $h_r(0)$ denotes the circle average, and let
\eqbn
 h^\ep  :=     h(R^\ep \cdot  )    + Q\log R^\ep  - \frac{1}{\gamma} \log \ep \quad \op{and} \quad \eta^\ep(t) :=   \eta(\ep t) / R^\ep .
\eqen
By~\cite[Proposition 4.13(i)]{wedges}, $h^\ep \eqD h  $. 
By the $\gamma$-LQG coordinate change formula~\eqref{eqn-lqg-coord}, $\mu_h(X) = \ep \mu_{h^\ep}(  X / R^\ep)$ for each Borel set $X\subset \BB C$, hence $\eta^\ep$ is parameterized by $\gamma$-quantum mass with respect to $h^\ep$. From this and the scale invariance of the law of space-filling SLE$_{\kappa'}$\footnote{This scale invariance is immediate from the construction of whole-plane space-filling SLE$_{\kappa'}$ from~\cite{ig4}, as described in Section~\ref{sec-wpsf-prelim}, and the scale invariance of the law of the whole-plane GFF modulo additive constant.}, we get $(h^\ep , \eta^\ep) \eqD  (h  ,  \eta )$. On the other hand, since the definitions of $\op{DA}^\ep(0)$, $\op{deg}_{\op{in}}^\ep(0)$, and $\op{deg}_{\op{out}}^\ep(0)$ are unaffected by spatial scaling,  
$\op{DA}^\ep(0)$ is determined by $(h^\ep , \eta^\ep)$ in the same manner that $\op{DA}^1(0)$ is determined by $(h,\eta)$, and similarly for $\op{deg}_{\op{in}}^\ep(0)$ and $\op{deg}_{\op{out}}^\ep(0)$. Hence the laws of these three quantities do not depend on $\ep$ and the theorem statement follows.
\end{proof}

\subsection{Global regularity event for area, diameter, and degree}
\label{sec-diam-reg}

Theorem~\ref{thm-moment} and the translation invariance of the law of the whole-plane GFF modulo additive constant allow us to control the quantities $\op{DA}^\ep(z)$, $\op{deg}_{\op{in}}^\ep(z)$, and $\op{deg}_{\op{out}}^\ep(z)$ for one point $z$ at a time (see Lemma~\ref{lem-moment-given-circle-avg}), but to control various error terms in our estimates we will also need global regularity bounds for these quantities which hold for all points $z$ in a fixed Euclidean ball simultaneously. 
One does not expect $\ep$-independent global bounds (since the number of cells in a fixed ball tends to $\infty$ as $\ep\rta 0$) but the following proposition will be sufficient for our purposes.

\begin{prop} \label{prop-diam-reg} 
Suppose we are in the setting of Section~\ref{sec-standard-setup} with $h$ equal to either a whole-plane GFF normalized so that $h_1(0) = 0$ or the circle-average embedding of a 0-quantum cone into $(\BB C , 0 , \infty)$. Also let $\rho \in (0,1)$, $q\in \left( 0 , \tfrac{2}{(2+\gamma)^2} \right)$, and $\zeta \in \left( 0 , q/4 \right)$.  For $\ep \in (0,1)$, let $E^\ep = E^\ep(\rho,q,\zeta )$ be the event that the following is true.
\begin{enumerate} 
\item For each $z\in B_{\rho}(0)$, $\op{diam}\left(\eta([\tau_z-\ep,\tau_z+\ep]) \right) \leq \ep^{q} $ and $\op{deg}_{\op{in}}^\ep(z) + \op{deg}_{\op{out}}^\ep(z) \leq \ep^{-\zeta}$. 
\label{item-diam-reg-single}
\item The circle average process of $h$ satisfies
\eqb \label{eqn-diam-reg-int}
\sup_{z\in B_{\rho}(0)} | h_{\ep^{q- 4\zeta} }(z)| \leq 2q \log \ep^{-1} ,
\eqe 
and the same is true with $ h$ replaced by the whole-plane GFF $h^{\op{IG}}$ used to construct $\eta$ in Section~\ref{sec-wpsf-prelim} (when it is normalized so that $h^{\op{IG}}_1(0) = 0$). 
\label{item-diam-reg-circle}
\end{enumerate} 
There exists $\alpha = \alpha( \gamma,q,\zeta) > 0$  such that for $\ep\in (0,1)$, 
\eqbn
\BB P\left[ E^\ep \right] \geq 1 - O_\ep(\ep^\alpha) 
\eqen
at a rate depending only on $q$, $\zeta$, $\rho$, and $\gamma$.
\end{prop}

The hardest part of the proof of Proposition~\ref{prop-diam-reg} is the upper bound for the localized degree (as defined just after~\eqref{eqn-localize-set}) which we treat in the following two lemmas. We note that the needed bound is not immediate from Lemma~\ref{lem-max-deg-nonlocal} since we are working with a 0-quantum cone rather than a $\gamma$-quantum cone and we need a bound for localized degree, rather than un-localized degree.
We first consider the inner localized degree, in which case we get a polylogarithmic upper bound with extremely high probability thanks to Lemma~\ref{lem-max-deg-nonlocal}. 

\begin{lem} \label{lem-max-deg-in}
Suppose we are in the setting of Proposition~\ref{prop-diam-reg}. For $\zeta \in (0,1)$, 
\eqb \label{eqn-max-deg-in}
\BB P\left[ \max_{z\in B_{\rho}(0)}   \op{deg}_{\op{in}}^\ep(z)   \leq (\log \ep^{-1})^{1+\zeta} \right] \geq 1  - o_\ep^\infty(\ep)  .
\eqe  
\end{lem}
\begin{proof}
Here we want to apply Lemma~\ref{lem-max-deg-nonlocal} (which is proven using the Brownian motion definition of the mated-CRT map), so we first consider the case when $h$ is the $R$-circle-average embedding of a $\gamma$-quantum cone in $(\BB C   , z_0 , \infty)$ for some $R>0$ and $z_0 \in \BB C$. This means that $h$ is normalized so that $R$ is the largest radius $r>0$ for which $h_r(z_0) + Q\log r =  0$ and $h|_{B_R(z_0)}$ agrees in law with the corresponding restriction of a whole-plane GFF plus $-\gamma \log |\cdot-z_0|$, normalized so that its circle average over $\bdy B_R(z_0)$ is 0.  
From here until~\eqref{eqn-in-deg-reg-no-localize'}, we work in this setting and define the objects involved in Proposition~\ref{prop-diam-reg} with this choice of $h$ (note that the estimates of Section~\ref{sec-area-diam-deg-moment} do not apply in this setting due to the log singularity of $h$ at $z_0$, but we will not need these estimates here).  

By Lemma~\ref{lem-max-deg-nonlocal} (applied to the structure graph generated from the field $h(R\cdot + z_0) +  Q\log R$, which is the circle average embedding of a $\gamma$-quantum cone in $(\BB C , 0 ,\infty)$, and the SLE$_{\kappa'}$ curve $R^{-1} \eta' - z_0$, which is parametrized by $\gamma$-LQG mass w.r.t.\ this field), 
\eqb \label{eqn-in-deg-reg-no-localize}
 \BB P\left[ \max_{z\in B_R(z_0) }  \op{deg} \left(x_z^\ep;\mcl G^\ep\right)    \leq (\log \ep^{-1})^{1+\zeta}   \right] \geq 1 - o_\ep^\infty(\ep) .
\eqe 
In fact, the proof of Lemma~\ref{lem-max-deg-nonlocal} shows that also
\eqb \label{eqn-in-deg-reg-no-localize-bdy}
 \BB P\left[ \max_{z\in B_R(z_0) } \max_{j\in \{-1,1\}}  \op{deg} \left(x_z^\ep + j \ep ;\mcl G^\ep\right)    \leq (\log \ep^{-1})^{1+\zeta}   \right] \geq 1 - o_\ep^\infty(\ep) ,
\eqe
which is not quite implied by~\eqref{eqn-in-deg-reg-no-localize} since it could be that $H_{x_z^\ep+j \ep}^\ep$ does not intersect $B_R(z_0)$ if $z$ is close to $\bdy B_R(z_0)$. 
Still assuming that $h$ is the $R$-circle-average embedding $\gamma$-quantum cone, we now transfer from~\eqref{eqn-in-deg-reg-no-localize} and~\eqref{eqn-in-deg-reg-no-localize-bdy} to an estimate for $\op{deg}_{\op{in}}^\ep(z)$. 

By definition, if we set $N := \op{deg}_{\op{in}}^\ep(z)$ then there is a collection of $N$ intervals $\{[a_j , b_j]\}_{j\in [1,N]_{\BB Z}}$ which intersect only at their endpoints, each of which has length $b_j  -a_j = \ep$, satisfies $\eta([a_j ,b_j]) \subset \ol B_z^\ep$, and is such that $\eta([a_j,b_j]) \cap \eta([\tau_z-\ep ,\tau_z+\ep])$ contains a non-trivial connected set.
We have
\eqbn
  \eta([\tau_z-\ep,\tau_z+\ep]) \subset H_{x_z^\ep-\ep}^\ep \cup  H_{x_z^\ep }^\ep \cup H_{x_z^\ep+\ep}^\ep ,
\eqen  
so for each $j\in [1,N]_{\BB Z}$ the segment $\eta([a_j , b_j])$ intersects either $H_{x_z^\ep-\ep}^\ep $, $H_{x_z^\ep }^\ep $, or $H_{x_z^\ep+\ep}^\ep$ along a non-trivial boundary arc or at an interior point. Hence $\eta([a_j,b_j]) $ intersects the interior of $H_y^\ep$ for some $y\in\ep\BB Z$ such that $y$ is either equal to or connected by an edge in $\mcl G^\ep$ to one of $x_z^\ep-\ep$, $x_z^\ep$, or $x_z^\ep+\ep$. Since the intervals $[a_j , b_j]$ intersect only at their endpoints and each has length $\ep$, each such vertex $y$ can correspond to at most 2 of the intervals $[a_j,b_j]$. 
Furthermore, the total number of such vertices $y$ is at most 3 plus the sum of the degrees of $x_z^\ep-\ep$, $x_z^\ep$, and $x_z^\ep+\ep$ in $\mcl G^\ep$. Therefore, 
\eqbn
\op{deg}_{\op{in}}^\ep(z)
\leq 2 \sum_{j=-1}^1 \op{deg}(x_z^\ep + j  \ep ; \mcl G^\ep) + 6  .
\eqen
 By~\eqref{eqn-in-deg-reg-no-localize} and~\eqref{eqn-in-deg-reg-no-localize-bdy}, the maximum of this last quantity over all $z\in B_R(z_0)$ is at most $ 6 (\log \ep^{-1})^{1+\zeta} + 6$ except on an event of probability $o_\ep^\infty(\ep)$. Applying this estimate with a slightly smaller value of $\zeta$ (to get rid of the 6's) shows that in the case when $h$ is the $R$-circle average embedding of a $\gamma$-quantum cone into $(\BB C   , z_0 , \infty)$, 
 \eqb \label{eqn-in-deg-reg-no-localize'}
 \BB P\left[ \max_{z\in B_R(z_0) } \op{deg}_{\op{in}}^\ep(z)   \leq (\log \ep^{-1})^{1+\zeta}   \right] \geq 1 - o_\ep^\infty(\ep) .
\eqe 
 
We now transfer to the case when $h$ is equal to either a whole-plane GFF normalized so that $h_1(0) = 0$ or the circle-average embedding of a 0-quantum cone into $(\BB C , 0 , \infty)$ using local absolute continuity. 
The argument will not depend on which choice of $h$ we are considering since the restrictions of both fields to $\BB D$ agree in law. 
Choose $R > 0$ and $z_0 \in \BB C$ such that $\BB D \subset B_{2R/3}(z_0) \setminus B_{R/3}(z_0)$. Let $h'$ be the $R$-circle-average embedding of a $\gamma$-quantum cone in $(\BB C   , z_0 , \infty)$.
Recall that $h|_{B_R(z_0)}$ agrees in law with the corresponding restriction of a whole-plane GFF normalized so that its circle average over $\bdy B_R(z_0)$ is zero, plus $-\gamma\log|\cdot|$.
The proof of~\cite[Proposition~3.4]{ig1} therefore shows that the Radon-Nikodym derivative of the conditional law of $h|_{B_{\rho}(0) }$ with respect to the conditional law of $h'|_{B_{\rho}(0)}$ is equal to
\eqb \label{eqn-max-deg-in-rn}
M := \BB E\left[ \exp\left( (h' , g)_\nabla  -\frac12 (g,g)_\nabla  \right) \,\big|\, h'|_{B_{\rho}(0)}\right] 
\eqe 
where $(\cdot,\cdot)_\nabla$ is the Dirichlet inner product (as in~\eqref{eqn-dirichlet}) and $g$ is a deterministic function with finite Dirichlet energy supported on $B_{(1+\rho)/2}(0)$ which comes from multiplying the log singularity of $h'$ by a smooth bump function supported on $B_{(1+\rho)/2}(0)$. 
Since $(h',g)_\nabla$ is Gaussian with variance $(g,g)_\nabla$, the Radon-Nikodym derivative $M$ has finite moments of all orders. 

We will now argue that the quantity $\max_{z\in B_{\rho}(0)}   \op{deg}_{\op{in}}^\ep(z)$  appearing in~\eqref{eqn-max-deg-in} depends on $h$ in a sufficiently local manner, so we can apply the above Radon-Nikodym derivative bound to transfer from~\eqref{eqn-in-deg-reg-no-localize'} to~\eqref{eqn-max-deg-in}.
For this purpose let $q\in (0,2/(2+\gamma)^2)$. By Lemma~\ref{lem-max-cell-diam} and the definition~\eqref{eqn-localize-set} of $\ol B_z^\ep$,  with $\alpha(q,\gamma)$ as in Lemma~\ref{lem-max-cell-diam} it holds with probability $1- \ep^{\alpha(q,\gamma) + o_\ep(1)}$ that $\ol B_z^\ep \subset B_{\rho + \ep^q}(0)$ for each $z\in B_\rho(0)$. 
In particular, since $\alpha(q,\gamma) \rta \infty$ as $q\rta 0$, we can send $q\rta 0$ to get that for any fixed $\rho' \in (\rho,1)$ it holds with probability $1-o_\ep^\infty(\ep)$ that $\ol B_z^\ep \subset B_{\rho'}(0)$ for each $z\in B_\rho(0)$.
By Lemma~\ref{lem-u_*-msrble}, the event that $\ol B_z^\ep \subset B_{\rho'}(0)$ for each $z\in B_\rho(0)$ is determined by $(h,h^{\op{IG}})|_{B_{\rho'}(0)}$ and on this event the quantity  $\max_{z\in B_{\rho}(0)}   \op{deg}_{\op{in}}^\ep(z)$ is also determined by $(h,h^{\op{IG}})|_{B_{\rho'}(0)}$. 
Hence, we can apply the Radon-Nikodym derivative estimate of the preceding paragraph (with $\rho'$ in place of $\rho$) along with H\"older's inequality to deduce~\eqref{eqn-max-deg-in} from~\eqref{eqn-in-deg-reg-no-localize'}.
\end{proof}

We next turn our attention to the outer localized degree, which  we recall counts the number of crossings of $\eta$ between $\eta([\tau_z-\ep , \tau_z+\ep])$ and the boundary of the ball $\ol B_z^\ep$. 

\begin{lem} \label{lem-max-deg-out}
Suppose we are in the setting of Proposition~\ref{prop-diam-reg}. There exists $\alpha =\alpha(\gamma ) > 0$ such that for each $\zeta \in (0,1)$, 
\eqb \label{eqn-max-deg-out}
\BB P\left[ \sup_{z\in B_{ \rho}(0) }    \op{deg}_{\op{out}}^\ep(z)   >  \ep^{-\zeta} \right] = O_\ep(\ep^\alpha) .
\eqe 
\end{lem}
\begin{proof}  
By Lemmas~\ref{lem-min-ball} and~\ref{lem-max-cell-diam}, there exists $\alpha = \alpha(\gamma) >0$ and $p_1 > p_2 > 0$ such that with probability at least $1-O_\ep(\ep^\alpha)$, it holds that
\eqb \label{eqn-max-deg-out-reg}
\ep^{p_1} \leq \op{diam}\left(\eta([\tau_z-\ep , \tau_z+\ep]) \right) \leq \ep^{p_2} ,\quad \forall z \in B_{\rho}(0)  .
\eqe 
By~\cite[Proposition~3.4 and Remark~3.9]{ghm-kpz}, it holds except on an event of probability $o_\ep^\infty(\ep)$ that the following is true. For each $\delta \in (0,\ep^{p_2}]$ and each $a,b\in\BB R$ with $a < b$ such that $\eta ([a,b]) \subset \BB D$ and $\op{diam}(\eta([a,b]) \geq \delta$, the set $\eta([a,b])$ contains a Euclidean ball of radius at least $\delta^{1+  \zeta/(2p_1)}$. Let $F^\ep$ be the event that this is the case and~\eqref{eqn-max-deg-out-reg} holds, so that $\BB P[F^\ep ] = 1-O_\ep(\ep^\alpha)$. 

Suppose now that $F^\ep$ occurs. Let $z\in B_{\rho}(0)$ and let $\{[a_j , b_j]\}_{j\in [1,N']_{\BB Z}}$ be a collection of intervals as in the definition of $\op{deg}_{\op{out}}^\ep(z)$. Then each $\eta([a_j,b_j])$ is contained in $  \ol B_z^\ep$ and
\eqbn
\op{diam}(\eta([a_j ,b_j])) \geq \op{diam}\left(\eta([\tau_z-\ep , \tau_z+\ep]) \right) \in \left[ \ep^{p_1} , \ep^{p_2} \right] .
\eqen
Therefore, $\eta([a_j ,b_j])$ contains a Euclidean ball of radius at least
\eqbn
\op{diam}\left(\eta([\tau_z-\ep , \tau_z+\ep])\right)^{1 + \zeta/(2p_1)} \geq \ep^{\zeta/2}  \op{diam}\left(\eta([\tau_z-\ep , \tau_z+\ep])  \right)  
\eqen
which is itself contained in $\ol B_z^\ep$. Since $\op{area}(\ol B_z^\ep) \asymp \op{diam}\left(\eta([\tau_z-\ep , \tau_z+\ep]) \right)^2 $ and different segments of the form $\eta([a,b])$ as above intersect only along their boundaries, by comparing areas we find that $\op{deg}_{\op{out}}^\ep(z) \preceq \ep^{-\zeta} $. The statement of the lemma follows from this and our above estimate for $\BB P[F^\ep]$.
\end{proof}

\begin{proof}[Proof of Proposition~\ref{prop-diam-reg}] 
By Lemmas~\ref{lem-max-cell-diam},~\ref{lem-max-deg-in}, and~\ref{lem-max-deg-out}, condition~\ref{item-diam-reg-single} in the definition of $E^\ep$ holds except on an event of probability decaying faster than some positive power of $\ep$ (when we apply Lemma~\ref{lem-max-cell-diam}, we note that $\eta([\tau_z-\ep,\tau_z+\ep])$ is contained in the union of at most three of the cells $H_x^\ep$ for $x\in\ep\BB Z$). 
It follows from~\cite[Corollary~2.5]{lqg-tbm2} that the probability that condition~\ref{item-diam-reg-circle} in the definition of $E^\ep$ fails to occur decays like a positive power of $\ep$. 
\end{proof}

\subsection{Law of large numbers for integrals over structure graph cells}
\label{sec-area-diam-deg-lln}

In this subsection we will prove Proposition~\ref{prop-area-diam-deg-lln}. 
Let $h^{\op{IG}}$ be the whole-plane GFF viewed used to construct $\eta$ as in Section~\ref{sec-wpsf-prelim} (recall that IG stands for ``imaginary geometry"), and assume that $h^{\op{IG}}$ is normalized so its circle average over $\bdy\BB D$ is 0. 

For $z \in \BB C$ and $\ep\in (0,1)$, define $u^\ep(z)$ as in~\eqref{eqn-area-diam-deg-def} and define its localized analog
\eqb \label{eqn-area-diam-deg-def*} 
u^\ep_*(z) := \op{DA}^\ep(z) \left(\op{deg}_{\op{in}}^\ep(z)  + \op{deg}_{\op{out}}^\ep(z) \right) ,
\eqe
where here $\op{DA}^\ep(z)$, $\op{deg}_{\op{in}}^\ep(z)$, and $\op{deg}_{\op{out}}^\ep$ are defined as in Section~\ref{sec-area-diam-deg-localize}. 
We note that~\eqref{eqn-area-diam-localize-bound} and~\eqref{eqn-deg-localize-bound} together imply that
\eqb \label{eqn-u-u_*-compare}
u^\ep(z) \leq 2 u^\ep_*(z) .
\eqe  

The remainder of this subsection is devoted to the proof of the following proposition, which (by~\eqref{eqn-u-u_*-compare}) immediately implies Proposition~\ref{prop-area-diam-deg-lln}.

\begin{prop} \label{prop-area-diam-deg-lln*}
The statement of Proposition~\ref{prop-area-diam-deg-lln} is true with $u_*^\ep(z)$ in place of $u^\ep(z)$. 
\end{prop}
 
Fix $\rho \in (0,1)$. Also fix 
\eqb \label{eqn-q-zeta-choice}
q \in \left( 0 ,  \min\left\{ \tfrac{2}{(2+\gamma)^2} , \gamma \right\} \right) \quad \op{and} \quad \zeta \in (0,q/100) , 
\eqe 
chosen in a manner depending only on $\gamma$, and define the regularity event $E^\ep = E^\ep(\rho,q,\zeta)$ as in Proposition~\ref{prop-diam-reg} for this choice of $\rho,q$, and $\zeta$. We note that the particular choice of $q$ and $\zeta$ satisfying~\eqref{eqn-q-zeta-choice} does not matter for the proof.

The idea of the proof of Proposition~\ref{prop-area-diam-deg-lln} is to show that (roughly speaking) the variance of $\int_{D} f(z) u_*^\ep(z)\,dz$ on $E^\ep$ decays like a positive power of $\ep$ using the local independence result Lemma~\ref{lem-rn-mean}; and bound the expectation of this integral on $E^\ep$ using Theorem~\ref{thm-moment}.

Since we will be using long-range independence for $(h,h^{\op{IG}})$, we need to consider a localized version of the event $E^\ep$. To this end, 
for $z\in \BB C$ and $\ep\in (0,1)$, let $F^\ep(z)$ be the event that the following is true.
\begin{enumerate}
\item $\op{diam}\left( \eta([\tau_z-\ep,\tau_z+\ep]) \right) \leq   \ep^{q   }$. \label{item-diam-reg-local}
\item $\op{deg}_{\op{in}}^\ep(z)  + \op{deg}_{\op{out}}^\ep(z) \leq \ep^{-\zeta}$. \label{item-deg-reg-local} 
\item $| h_{\ep^{q-4\zeta}}(z)| \vee | h_{\ep^{q-4\zeta}}^{\op{IG}}(z)| \leq 2q \log \ep^{-1}$. \label{item-ca-reg-local} 
\end{enumerate} 
By Lemma~\ref{lem-u_*-msrble},
\eqb \label{eqn-Fu_*-determined}
\BB 1_{F^\ep(z)} u_*^\ep(z) \: \text{is a.s.\ determined by} \: (h , h^{\op{IG}})|_{B_{4\ep^q}(z)} \: \op{and} \: ( h_{\ep^{q-4\zeta}}(z)   ,    h_{\ep^{q-4\zeta}}^{\op{IG}}(z) ) .
\eqe
We also note that by definition, 
\eqb \label{eqn-lln-event-contain}
\bigcap_{z\in B_{\rho}(0)} F^\ep(z) = E^\ep . 
\eqe 

To prove Proposition~\ref{prop-area-diam-deg-lln*}, we will need moment bounds for $\BB 1_{F^\ep(z)} u_*^\ep(z)$ for all $z\in B_{\rho}(0)$ (not just the moment bound when $z=0$ which comes from Theorem~\ref{thm-moment}). 
In fact, since our local independence result Lemma~\ref{lem-rn-mean} involves the conditional law of a random variable $X$ given the circle average of the field, we will need a moment bound for $\BB 1_{F^\ep(z)} u_*^\ep(z)$ when we condition on certain circle averages of $ h$ and $h^{\op{IG}}$.
 
\begin{lem} \label{lem-moment-given-circle-avg}
Suppose $h$ is a whole-plane GFF with $h_1(0) =0$ and define the events $F^\ep(z)$ as above.
There exists $\ep_* = \ep_*(\rho , \gamma ) \in (0,1)$ such that for each $p\geq 1 $,  
\eqb \label{eqn-moment-given-circle-avg-DA} 
\BB E\left[ \BB 1_{F^\ep(z)} \op{DA}^\ep(z)^p \,|\,  h_{\ep^{q- 4\zeta}}(z)  ,\,    h_{\ep^{q- 4\zeta}}^{\op{IG}}(z)  \right]  \preceq  1 ,\quad \forall z\in B_{\rho}(0) ,\quad \forall \ep \in (0,\ep_*] 
\eqe 
and for each $p \in (1,4/\gamma^2)$, 
\eqb \label{eqn-moment-given-circle-avg} 
\BB E\left[ \BB 1_{F^\ep(z)} u_*^\ep(z)^p \,|\,  h_{\ep^{q- 4\zeta}}(z)  ,\,    h_{\ep^{q- 4\zeta}}^{\op{IG}}(z)  \right]  \preceq  1 ,\quad \forall z\in B_{\rho}(0) ,\quad \forall \ep \in (0,\ep_*] 
\eqe 
with deterministic implicit constant depending only on $p$, $\rho$, and $\gamma$. 
\end{lem} 

The larger range of possible values of $p$ in~\eqref{eqn-moment-given-circle-avg-DA} as compared to~\eqref{eqn-moment-given-circle-avg} is due to the difference in the range of possible $p$ values in Theorem~\ref{thm-moment}.

\begin{proof}[Proof of Lemma~\ref{lem-moment-given-circle-avg}]
We will prove~\eqref{eqn-moment-given-circle-avg}; the estimate~\eqref{eqn-moment-given-circle-avg-DA} is proven in an identical manner except that we only need to use~\eqref{eqn-area-diam} instead of all three of the estimates of Theorem~\ref{thm-moment}. 

Fix $\frk a_1,\frk a_2 \in [-2q\log \ep^{-1} , 2q\log \ep^{-1}]$. To prove~\eqref{eqn-moment-given-circle-avg}, we will condition on $\{h_{\ep^{q-4\zeta}}(z) = \frk a_1, \:  h_{\ep^{q-4\zeta}}^{\op{IG}}(z) = \frk a_2\}$, apply an affine transformation sending  $B_{\ep^{q-4\zeta}}(z)$ to $\BB D$ and re-normalize to get a new pair of fields with the same law as $(h,h^{\op{IG}})|_{\BB D}$, then apply Theorem~\ref{thm-moment} to this new field/curve pair with $\ep$ replaced by a larger ($\frk a_1$-dependent) value determined by the $\gamma$-LQG coordinate change formula. Note that we can restrict attention to this range of $\frk a_1,\frk a_2$ due to condition~\ref{item-ca-reg-local} in the definition of $F^\ep(z)$.
 
By Lemma~\ref{lem-big-circle-rn'} (applied to each of the independent fields $h$ and $h^{\op{IG}}$ and with $\delta = \ep^{q-4\zeta}$) there exists $\ep_* \in (0,1)$ as in the statement of the lemma such that for $\ep \in (0,\ep_*]$, the conditional law of $(h , h^{\op{IG}})|_{B_{\ep^{q-4\zeta}}(z)}$ given $\{h_{\ep^{q-4\zeta}}(z) = \frk a_1, \:  h_{\ep^{q-4\zeta}}^{\op{IG}}(z) = \frk a_2\}$ is absolutely continuous with respect to the law of the restriction to $B_{\ep^{q-4\zeta}}(z)$ of a pair $(  h^{\frk a_1} ,  h^{\op{IG}, \frk a_2})$ of independent whole-plane GFFs normalized to have circle averages $\frk a_1$ and $\frk a_2$, respectively, over $\bdy B_{\ep^{q-4\zeta}}(z)$. Furthermore, since $\frk a_1,\frk a_2 \in [-2q\log \ep^{-1} , 2q\log \ep^{-1}]$, for each $p>0$ the $p$th moment of the Radon-Nikodym derivative $M_{\frk a_1,\frk a_2}$ of Lemma~\ref{lem-big-circle-rn'} is bounded above by a constant depending only on $\rho$ and $p$. 

Let $X_{\frk a_1,\frk a_2}$ be the random variable which is determined by $(  h^{\frk a_1} ,  h^{\op{IG}, \frk a_2})|_{B_{\ep^{q-4\zeta}}(z)}$ in the same manner that $\BB 1_{F^\ep(z)} u_*^\ep(z)$ is determined by $( h , h^{\op{IG}})|_{B_{\ep^{q-4\zeta}}(z)}$ (c.f.~\eqref{eqn-Fu_*-determined}). The preceding paragraph together with H\"older's inequality shows that for $\ep \in (0,\ep_*]$ and $p , p'  \in (1,4/\gamma^2)$ with $p < p'$, 
\eqb \label{eqn-moment-give-circle-avg-compare}
\BB E\left[  \BB 1_{F^\ep(z)}  u_*^\ep(z)^p \,|\, h_{\ep^{q-4\zeta}}(z) = \frk a_1, \:  h_{\ep^{q-4\zeta}}^{\op{IG}}(z) = \frk a_2 \right] \preceq \BB E\left[ X_{\frk a_1, \frk a_2}^{p'} \right]^{p/p'}
\eqe 
with the implicit constant depending only on $\rho$, $p$, and $p'$. 

We now estimate the right side of~\eqref{eqn-moment-give-circle-avg-compare} using Theorem~\ref{thm-moment}. 
If we compose $(h^{\frk a_1} , h^{\op{IG},\frk a_2})$ with an affine transformation which takes $\BB D$ to $B_{\ep^{q-4\zeta}}(z)$ and subtract $(\frk a_1, \frk a_2)$, we obtain a new pair of fields $(\wh h  , \wh h^{\op{IG} })$ with the same law as $(h,h^{\op{IG}})|_{\BB D}$. 
Let
\eqbn
\wh \ep := e^{-\gamma \frk a_1} \ep^{1-\gamma Q (q-4\zeta)} ,
\eqen
recall the ball $\ol B_0^\ep$ from~\eqref{eqn-localize-set}, 
and let $\wh X_{\frk a_1,\frk a_2}$ be the random variable which is determined by $ (\wh h  , \wh h^{\op{IG} })$ in the same manner that $ u^{\wh\ep}_*(0) \BB 1_{\ol B_0^{\wh \ep} \subset \BB D}$ is determined by $(h ,h^{\op{IG}})|_{\BB D}$ (c.f.\ Lemma~\ref{lem-u_*-msrble}).

Using the $\gamma$-LQG coordinate change formula~\eqref{eqn-lqg-coord}, we find that $u_*^\ep(z)$ (resp.\ $\ep^{-(q-4\zeta)}(\ol B_z^\ep-z)$) is determined by $(\wh h  , \wh h^{\op{IG} })$ in the same manner that $u_*^{\wh\ep}(0)$ (resp.\ $\ol B_0^{\wh \ep}$) is determined by $(h , h^{\op{IG}})$. 
Since $\ol B_z^\ep \subset B_{\ep^{q-4\zeta}}(z)$ on $F^\ep(z)$, we infer that a.s.\ $X_{\frk a_1, \frk a_2} \leq \wh X_{\frk a_1, \frk a_2}$. 
Since the $h|_{\BB D}$ agrees in law with the corresponding restriction of the circle-average embedding of a 0-quantum cone, we can apply Theorem~\ref{thm-moment} with $\wh\ep$ in place of $\ep$ to get 
\eqbn
\BB E\left[ X_{\frk a_1,\frk a_2}^{p'} \right] \leq \BB E\left[ \wh X_{\frk a_1,\frk a_2}^{p'} \right]  \preceq 1
\eqen
for each $p' \in (1,4/\gamma^2)$. Combining this with~\eqref{eqn-moment-give-circle-avg-compare} concludes the proof.
\end{proof}

\begin{proof}[Proof of Proposition~\ref{prop-area-diam-deg-lln*}]
For most of the proof we consider the case of a whole-plane GFF normalized so that $h_1(0) =0$; we transfer to the case of a 0-quantum cone only at the very end.
\medskip

\noindent\textit{Step 0: setup.}
By~\eqref{eqn-moment-given-circle-avg} of Lemma~\ref{lem-moment-given-circle-avg}, there exists $A = A(\rho, \gamma) > 0$ and $\ep_* = \ep_*(\rho,\gamma ) \in (0,1)$  such that for $z\in B_{\rho}(0)$ and $\ep \in (0,\ep_*]$, a.s.\ 
\eqb \label{eqn-use-cond-moment}
\BB E\left[ \BB 1_{F^\ep(z)} u_*^\ep(z) \,|\,    h_{\ep^{q-4\zeta}}(z) ,\,  h_{\ep^{q- \zeta}}^{\op{IG}}(z)   \right] \leq A .
\eqe 
Let 
\eqb \label{eqn-ol-u-def}
\ol u_*^\ep(z) := \BB 1_{F^\ep(z)} u_*^\ep(z)  - \BB E\left[ \BB 1_{F^\ep(z)} u_*^\ep(z) \,|\,  h_{\ep^{q-4\zeta}}(z) ,\,  h_{\ep^{q-4\zeta}}^{\op{IG}}(z)  \right] ,
\eqe 
so that by~\eqref{eqn-Fu_*-determined}, $\ol u_*^\ep(z)$ is a.s.\ determined by $(  h , h^{\op{IG}} )|_{B_{4\ep^q}(z)}$, $ h_{ \ep^{q-4\zeta}}(z)$, and $h^{\op{IG}}_{ \ep^{q-4\zeta}}(z)$.
By~\eqref{eqn-lln-event-contain}, if $E^\ep$ occurs then $ F^\ep(z) $ occurs for each $z\in B_{\rho}(0)$. Hence~\eqref{eqn-use-cond-moment} implies that on $E^\ep$, a.s.\ 
\eqb \label{eqn-add-int-to-var}
\int_{D} f(z) u_*^\ep(z)\,dz  - A \int_{D} f(z) \,dz 
\leq \int_{D} f(z) \ol u_*^\ep(z) \, dz .
\eqe 
We will now show that the right side of~\eqref{eqn-add-int-to-var} is unlikely to be larger than a positive power of $\ep$ by showing that its second moment is small on $E^\ep$. 
 
We have
\allb \label{eqn-diam-area-deg-var-decomp}
 \BB E\left[ \BB 1_{E^\ep}  \left(  \int_{D} f(z) \ol u_*^\ep(z) \, dz \right)^2 \right] 
&= \int_{D} \int_{D} f(z) f(w) \BB E\left[  \BB 1_{E^\ep}  \ol u_*^\ep(z) \ol u_*^\ep(w) \right] \,dz \, dw \notag\\
&= \sum_{i=1}^2 \iint_{W_i^\ep} f(z) f(w) \BB E\left[  \BB 1_{E^\ep}  \ol u_*^\ep(z) \ol u_*^\ep(w) \right] \,dz \, dw   
\alle
where 
\eqb \label{eqn-W-def}
W_1^\ep := \left\{ (z,w) \in  D \times D : |z-w| \leq 2 \ep^{q- 4\zeta } \right\} \quad\op{and}\quad
W_2^\ep := (D \times D)\setminus W_1^\ep  .
\eqe
We will bound the integrals over $W_1^\ep$ and $W_2^\ep$ separately. 
\medskip

\noindent\textit{Step 1: the integral over $W_1^\ep$.}
Let $z,w\in B_\rho(0)$. Using the definition~\eqref{eqn-ol-u-def} of $\ol U_*^\ep(z)$, we make the following calculation, each line of which we justify just below. 
\allb \label{eqn-u-prod}
&\BB E\left[ \BB 1_{E^\ep} \ol u_*^\ep(z) \ol u_*^\ep(w) \right] \notag\\
&\qquad \leq \BB E\left[ \BB 1_{E^\ep} u_*^\ep(z) u_*^\ep(w) \right] + \BB E\bigg[ \BB E\left[ \BB 1_{F^\ep(z)} u_*^\ep(z) \,|\,  h_{\ep^{q-4\zeta}}(z) ,\,  h_{\ep^{q-4\zeta}}^{\op{IG}}(z)  \right] \BB E\left[ \BB 1_{F^\ep(w)} u_*^\ep(w) \,|\,  h_{\ep^{q-4\zeta}}(w) ,\,  h_{\ep^{q-4\zeta}}^{\op{IG}}(w)  \right] \bigg] \notag\\
&\qquad \leq \BB E\left[ \BB 1_{E^\ep} u_*^\ep(z) u_*^\ep(w) \right] + A^2 \quad \notag \\
&\qquad \leq \ep^{-2\zeta} \BB E\left[ \BB 1_{E^\ep} \op{DA}^\ep(z) \op{DA}^\ep(w)  \right] + A^2   .
\alle
The first inequality in~\eqref{eqn-u-prod} comes from expanding and dropping the negative terms.
The second inequality comes from~\eqref{eqn-use-cond-moment}.
The last inequality comes from the fact that $ \op{deg}_{\op{in}}^\ep(z) + \op{deg}_{\op{out}}^\ep(z)   \leq \ep^{-\zeta}$ for all $z\in B_{\rho}(0)$ on $E^\ep$.

By taking an unconditional expectation in~\eqref{eqn-moment-given-circle-avg-DA} of Lemma~\ref{lem-moment-given-circle-avg} and recalling that $F^\ep(z) \supset E^\ep$, we find that for small enough values of $\ep >0$, for each $p > 0$ the $p$th moment of $\BB 1_{E^\ep}  \op{DA}^\ep(z)$ is bounded above by a constant depending only on $p,\rho$, and $\gamma$ for $z \in B_{\rho}(0)$. 
Using this and the Cauchy-Scwarz inequality to bound the last line of~\eqref{eqn-u-prod}, we get
\eqb \label{eqn-u-prod'}
\BB E\left[ \BB 1_{E^\ep} \ol u_*^\ep(z) \ol u_*^\ep(w) \right] 
\preceq \ep^{-2\zeta}   + A^2 \preceq \ep^{-2\zeta} ,
\eqe 
with the implicit constant depending only on $p,\rho$, and $\gamma$. 

Using~\eqref{eqn-u-prod'} and the definition~\eqref{eqn-W-def} of $W_1^\ep$, we now get
\allb \label{eqn-diam-area-deg-diag}
 \iint_{W_1^\ep}   f(z) f(w)  \BB E\left[ \BB 1_{E^\ep}   \ol u_*^\ep(z) \ol u_*^\ep(w) \right] \,dz \, dw  
&\preceq   \iint_{W_1^\ep}   f(z) f(w) \ep^{-2\zeta}    \, dz\,dw  \notag\\
&\preceq \ep^{-2\zeta} \| f\|_\infty^2 \op{Vol}(W_1^\ep) \notag\\
&\preceq \ep^{ 2q - 10\zeta } \| f\|_\infty^2 \op{area}(D)  .
\alle  
Note that $ 2q- 10\zeta >0$ by our choice of $\zeta$ from~\eqref{eqn-q-zeta-choice}.
\medskip

\noindent\textit{Step 2: the integral over $W_2^\ep$.}
We now consider the integral over the off-diagonal region $W_2^\ep$. Here we need to use local independence. 
Recall that $\ol u_*^\ep(z)$ for $z\in B_{\rho}(0)$ is a.s.\ determined by $(h , h^{\op{IG}})|_{B_{4\ep^q}(z)}$, $ h_{ \ep^{q-4\zeta}}(z)$, and $h^{\op{IG}}_{ \ep^{q-4\zeta}}(z)$. Furthermore, by~\eqref{eqn-ol-u-def} we have $ \BB E\left[ \ol u_*^\ep(z) \,|\,  h_{\ep^{q-4\zeta}}(z) ,\,  h_{\ep^{q-4\zeta}}^{\op{IG}}(z)  \right]   = 0$.
Lemma~\ref{lem-rn-mean} (applied with $\delta=4\ep^{4\zeta}$, $s=1/2$, and $X = \ol u_*^\ep(z)$) together with the invariance of the law of the whole-plane GFF under translation and scaling, modulo additive constant, shows that for $1 < p < p' < 4/\gamma^2$, there exists $b = b(p,p') >0$ such that for small enough $\ep  > 0$ (how small depends only on $p,p'$, and $\zeta$),
\allb \label{eqn-use-rn-mean}
\BB E\left[ \left| \BB E\left[ \ol u_*^\ep(z) \,|\, (h, h^{\op{IG}})|_{\BB C\setminus B_{\ep^{q-4\zeta}}(z)}      \right] \right|^p \right] 
  \preceq  \ep^{2 p \zeta }  \BB E\left[ |\ol u_*^\ep(z)|^p   \right]  
+ e^{-b \ep^{-2 \zeta} } \BB E\left[ |\ol u_*^\ep(z)|^{p'} \right]
\alle
with implicit constant depending only on $\gamma$ provided we choose $p$ and $p'$ in a manner which depends only on $\gamma$.
By taking the unconditional expectation of both sides of~\eqref{eqn-moment-given-circle-avg} from Lemma~\ref{lem-moment-given-circle-avg}, we find that the right side of~\eqref{eqn-use-rn-mean} is bounded above by a constant (depending only on $\rho$ and $\gamma$) times $\ep^{2 p\zeta }$. 

For $(z,w) \in W_2^\ep$, we have $|z-w| > 2\ep^{q- 4\zeta}$. The random variable $\ol u_*^\ep(w)$ is a.s.\ determined by  $(h, h^{\op{IG}})|_{\BB C\setminus B_{\ep^{q-4\zeta}}(z)}  $, so by Lemma~\ref{lem-moment-given-circle-avg}, for each pair $(z,w) \in W_2^\ep$ and each $p\in (1,4/\gamma^2)$,   
\allb
\BB E\left[\BB 1_{E^\ep} \ol u_*^\ep(z) \ol u_*^\ep(w) \right] 
&= \BB E\left[ \BB E\left[  \ol u_*^\ep(z) \,|\,  (h , h^{\op{IG}})|_{\BB C\setminus B_{\ep^{q-4\zeta}}(z)}\right]  \ol u_*^\ep(w) \right] \notag\\
&\preceq \BB E\left[ \left| \BB E\left[  \ol u_*^\ep(z) \,|\,  (h , h^{\op{IG}})|_{\BB C\setminus B_{\ep^{q-4\zeta}}(z)}\right] \right|^p \right]^{1/p} \BB E\left[     u_*^\ep(w)^{\frac{p}{1-p}} \BB 1_{F^\ep(w)}  \right]^{1-1/p} \quad \text{(by H\"older)} \notag \\
&\preceq   \ep^{2\zeta}    \BB E\left[    u_*^\ep(w)^{\frac{p}{1-p}} \BB 1_{F^\ep(w)}  \right]^{1-1/p} \quad \text{(by~\eqref{eqn-use-rn-mean})} \notag \\
&\preceq    \ep^{ \zeta}  \BB E\left[  \op{DA}^\ep(w)^{\frac{p}{1-p}}   \BB 1_{F^\ep(w)}\right]^{1-1/p} \quad \text{(by condition~\ref{item-deg-reg-local} in the definition of $F^\ep(w)$)} \notag \\
&\preceq   \ep^{ \zeta}    \quad \text{(by~\eqref{eqn-moment-given-circle-avg-DA})} ,
\alle
with implicit constant depending only on $\rho$ and $\gamma$ provided we choose $p$ and $p'$ in a manner which depends only on $\gamma$.  
Hence
\allb \label{eqn-diam-area-deg-offdiag}
\iint_{W_2^\ep} f(z)f(w) \BB E\left[ \BB 1_{E^\ep}  \ol u_*^\ep(z) \ol u_*^\ep(w) \right] \,dz \, dw 
\preceq \ep^{\zeta} \|f\|_\infty^2 \op{area}(D)^2   .
\alle
\medskip

\noindent\textit{Step 3: conclusion.}
By plugging the estimates~\eqref{eqn-diam-area-deg-diag} and~\eqref{eqn-diam-area-deg-offdiag} into~\eqref{eqn-diam-area-deg-var-decomp}, we get
\eqb \label{eqn-lln-var-bound}
\BB E\left[ \BB 1_{E^\ep}  \left(  \int_{D} f(z) \ol u_*^\ep(z) \,dz  \right)^2 \right] \preceq \ep^{\alpha_0} \| f\|_\infty^2 \op{area}(D) \preceq \ep^{\alpha_0} \|f\|_\infty^2 
\eqe 
where $\alpha_0 = \alpha_0(\gamma)  = ( 2q-10\zeta) \wedge \zeta$ (recall that we have chosen $q$ and $\zeta$ depending on $\gamma$ above). 
By applying~\eqref{eqn-lln-var-bound} and the Chebyshev inequality to bound the right side of~\eqref{eqn-add-int-to-var}, we obtain
\eqbn
\BB P\left[ \BB 1_{E^\ep}  \int_{D} f(z)  u_*^\ep(z) \,dz  >   A \int_{D} f(z) \,dz +    \ep^{\alpha_0/4} \| f\|_\infty  \right] \preceq \ep^{\alpha_0/2} .
\eqen
Since $\BB P[(E^\ep)^c]$ decays like a positive power of $\ep$ (Proposition~\ref{prop-diam-reg}) we obtain~\eqref{eqn-area-diam-deg-lln} in the case of a whole-plane GFF. 

The case of a 0-quantum cone follows from the case of a whole-plane GFF and the fact that the restrictions of a 0-quantum cone and a whole-plane GFF to $\BB D$ agree in law together with Lemma~\ref{lem-u_*-msrble} and Lemma~\ref{lem-max-cell-diam} (the latter is used to make sure the balls $\ol B_z^\ep$ for $z\in B_{\rho}(0)$ are contained in $\BB D$ with high probability). 
\end{proof}

\subsection{Proof of Proposition~\ref{prop-area-diam-deg-gamma}}
\label{sec-gamma-cone-proof}
 
Proposition~\ref{prop-area-diam-deg-lln} gives an analog of Proposition~\ref{prop-area-diam-deg-gamma} in the case of the whole-plane GFF or the 0-quantum cone. In this subsection we will transfer from the case of the whole-plane GFF to the case of the $\gamma$-quantum cone.
In fact, it will be convenient to work with quantities which are locally determined by the field, so we will actually transfer Proposition~\ref{prop-area-diam-deg-lln*} instead of Proposition~\ref{prop-area-diam-deg-lln}.

\begin{prop} \label{prop-area-diam-deg-gamma*}
In the case when $h$ is the circle-average embedding of a $\gamma$-quantum cone into $(\BB C , 0 , \infty)$, 
the statement of Proposition~\ref{prop-area-diam-deg-gamma} is true with $u_*^\ep(z)$, defined as in~\eqref{eqn-area-diam-deg-def*}, in place of $u^\ep(z)$. 
\end{prop}

We will deduce Proposition~\ref{prop-area-diam-deg-gamma*} from Proposition~\ref{prop-area-diam-deg-lln*} and the following basic fact about the $\gamma$-LQG measure: if $h^0$ is a whole-plane GFF and we sample $\BB z$ uniformly from Lebesgue measure on $\BB D$, independently from everything else, and set $h^1 = h^1  - \gamma \log|\cdot-\BB z|  + \gamma \log \max(|\cdot|, 1)$, then the laws of $h^0$ and $h^1$ are mutually absolutely continuous, with an explicit Radon-Nikodym derivative (see, e.g.,~\cite[Lemma A.10]{wedges}). Roughly speaking, the reason for this is that $h^0$ a.s.\ has a $\gamma$-log singularity at a point sampled uniformly from its $\gamma$-LQG measure.

The field $h^1(\cdot - \BB z)$ is close to a $\gamma$-quantum cone on a neighborhood of 0 (modulo normalization), so we can transfer statements about a whole-plane GFF to statements about a $\gamma$-quantum cone. The details of the proof will involve several steps in which we control the Radon-Nikodym derivatives between the laws of successive fields.
 
Fix $q\in \left( 0, \tfrac{2}{(2+\gamma)^2} \right)$, chosen in a manner depending only on $\gamma$. For a given random distribution $h$ on $\BB C$, let $u_*^\ep(z)$ for $z\in \BB C$ be defined as in~\eqref{eqn-area-diam-deg-def*}. For $\rho \in (0,1]$, define the regularity event 
\eqb \label{eqn-lln-E_0}
 E^\ep_0(\rho) := \left\{ \op{diam}(\eta([\tau_z-\ep , \tau_z+\ep]) \leq \ep^q ,\: \forall x \in \mcl V\mcl G^\ep(B_{\rho}(0)) \right\} .
\eqe 
For $\alpha >0$, $A > 0$ and a bounded measurable function $f : B_{\rho}(0) \rta [0,\infty)$, also define
\eqbn
G_f^\ep(h ; \rho) = G_f^\ep(h;  \rho ,  A , \alpha , q   ) 
:= E_0^\ep(\rho) \cap \left\{ \int_{B_{\rho}(0)} f(z)  u_*^\ep(z) \, dz \leq A \int_{B_{\rho}(0)} f(z)\,dz + \ep^\alpha \|f\|_\infty \right\} .
\eqen
We observe that $G_f^\ep(h)$ is a.s.\ determined by $h|_{B_{\rho + 4\ep^q}(0)}$ and the imaginary geometry field $h^{\op{IG}}$. 
We will prove the following statement for several different choices of $h$. 
\allb \label{eqn-lln-statement}
&\text{There exists $\alpha = \alpha(\gamma) > 0$ and and $A = A(\rho,\gamma) > 0$ such that for each bounded measurable} \notag \\
&\text{function $f : B_{\rho}(0) \rta [0,\infty)$, we have $\BB P[G_f^\ep(h ; \rho) ] \geq 1 - O_\ep( \ep^\alpha )$,} \notag\\
&\text{at a rate depending only on $\rho$ and $\gamma$.}
\alle

\noindent\textit{Case 0.} Let $h^0$ be a whole-plane GFF normalized so that the circle average $h^0_1(0)$ is $0$. 
By Lemma~\ref{lem-max-cell-diam} and Proposition~\ref{prop-area-diam-deg-lln*}, we know that~\eqref{eqn-lln-statement} is true with $h = h^0$. 
\medskip

\noindent\textit{Case 1.} Let $\BB z$ be sampled uniformly from Lebesgue measure on $\BB D$, independently from everything else, and define the field
\eqbn
h^1  := h^0  - \gamma \log |\cdot - \BB z| +  \gamma \log \max(|\cdot|,1)  .
\eqen 
By~\cite[Lemma A.10]{wedges}, the law of $h^1$ is the same as the law of $h^0$ weighted by $\mu_{h^0}(\BB D) / \BB E[\mu_{h^0}(\BB D)]$. 
Since $\mu_{ h^0}(\BB D)$ has a finite $p$th moment for some $p >1$~\cite[Theorem 2.11]{rhodes-vargas-review} and by H\"older's inequality,~\eqref{eqn-lln-statement} for $h= h^1$ and any $\rho \in (0,1)$ follows from~\eqref{eqn-lln-statement} for $h=h^0$ (note that here the value of $\alpha$ corresponding to $h^1$ is equal to $1-1/p$ times the value of $\alpha$ corresponding to $h^1$).   

In fact, for small enough $\xi = \xi(\gamma) >0$, we have $\BB P[\BB z \in B_{\ep^\xi}(0)] \succeq \ep^{\alpha/2}$ with universal implicit constant (where here the value of $\alpha$ is the one corresponding to $h^1$). Hence,~\eqref{eqn-lln-statement} also holds with $h$ sampled from the \emph{conditional} law of $h^1$ given $\{\BB z \in B_{ \ep^\xi}(0) \}$. Henceforth fix such a value of $\xi$; in what follows we will frequently condition on $\{\BB z \in B_{ \ep^\xi}(0) \}$. 
\medskip

\noindent\textit{Case 2.} We next establish~\eqref{eqn-lln-statement} for $\rho \in (0,1)$ and $h$ replaced by the field 
\eqbn
h^2 := h^1 - \gamma \log \max(|\cdot|,1) =  h^0 - \gamma \log |\cdot - \BB z| . 
\eqen
Indeed, since $\gamma \log \max(|\cdot|,1)$ is identically equal to 0 on $\BB D$, this case is immediate from the preceding one. In fact, by the last paragraph of the preceding case we get that~\eqref{eqn-lln-statement} is satisfied with $h$ sampled from the conditional law of $h^2$ given $\{ \BB z \in B_{\ep^\xi}(0) \}$, for any $\rho \in (0,1)$. 
\medskip

\noindent\textit{Case 3.} We next consider a value of $\rho \in (0,1/2)$ and the field
\eqbn
h^3 := h^2  - h^2_{1/2}(\BB z)  = h^0 - \gamma \log |\cdot - \BB z|    - h^0_{1/2}(\BB z)  + \gamma \log(1/2)  ,
\eqen
i.e., $h^2$ with the additive constant chosen so that $h^2_{1/2}(\BB z) =0$. 

For $\ep > 0$ and $i\in \{2,3\}$, let $u_*^{i ,  \ep}(\cdot)$ be as in~\eqref{eqn-area-diam-deg-def*} with $h = h^i$. 
Then for each $z\in\BB C$,
\eqb \label{eqn-u-23}
u_*^{3,S\ep}(z) =  u_*^{2, \ep}(z) \quad \op{for} \quad  S := e^{\gamma h^2_{1/2}(\BB z)} .
\eqe

We can bound integrals against $u_*^{2,\ep} $ by case 2, so we need to convert from integrals against $u_*^{3,\ep}$ to integrals against $u_*^{3,S\ep}$. 
To this end, we will compare the unconditional law of $h^3|_{B_{1/2}(\BB z)}$ given only $\BB z$ to its conditional law given $\BB z$ and $\frk a$. 
By Lemma~\ref{lem-big-circle-rn'} (applied with $\delta =1/2$, $w = \BB z$, and $h = h^0$), for $\frk a \in \BB R$ and $\frk z \in B_{\ep^\xi}(0)$, the conditional law of $h^3|_{B_{1/2}(\BB z)}$ given $\{h^2_{1/2}(\BB z) = \frk a\} \cap \{\BB z = \frk z\}$ is mutually absolutely continuous with respect to the unconditional law of $h^3|_{B_{1/2}(\BB z)}$ given only $\{\BB z = \frk z\}$. Furthermore, for $p > 1$ there exists $r_p > 0$ (depending only on $p$ and $\rho$) such that for small enough $\ep > 0$, the $-p$th moment of the Radon-Nikodym derivative $M_{\frk a , \frk z}$ is bounded above by a constant depending only on $p$ provided $\frk a \in [-r_p ,r_p]$ (which happens with uniformly positive probability). For each $\rho \in (0,1/2)$, we have $B_{\rho}(0) \subset B_{1/2}(\frk z)$ for small enough $\ep$. Consequently, for each such $\rho$ and each $\frk a \in [-r_p , r_p]$, 
\alb
\BB P\left[ G_f^\ep(h^3; \rho)^c \,|\, \BB z = \frk z \right] 
&\preceq \BB P\left[ G_f^\ep(h^3; \rho)^c \,|\, \BB z = \frk z ,\, h_{1/2}(\BB z) = \frk a \right] \quad \text{(by H\"older)} \\
&\preceq \BB P\left[ G_f^{e^{-\gamma \frk a} \ep} (h^2; \rho)^c \,|\, \BB z = \frk z ,\, h_{1/2}(\BB z) = \frk a \right] \quad \text{(by~\eqref{eqn-u-23})} .
\ale
We now integrate both sides of this inequality over $B_{\ep^\xi}(0)$ with respect to the law of $\BB z$ and over $[-r_p ,r_p]$ with respect to the law of $h_{1/2}(\BB z)$.
The right side is at most $O_\ep(\ep^\alpha)$ for appropriate $\alpha=\alpha(\gamma)$ since we know that $\BB P[h_{1/2}(\BB z) \in [-r_p , r_p]] \succeq 1$ and by~\eqref{eqn-lln-statement} in the case when $h$ is sampled from the conditional law  $h^2$ given $\{ \BB z \in B_{\ep^\xi}(0) \}$, but with $e^{\gamma \frk a} \ep$ in place of $\ep$. We thus obtain~\eqref{eqn-lln-statement} in the case when $h$ is sampled from the conditional law of $h^3$ given $\{ \BB z \in B_{\ep^\xi}(0) \}$ and $\rho \in (0,1/2)$.
\medskip

\noindent\textit{Case 4.} We next consider the field
\eqbn
h^4 := h^3\left( \frac12 \cdot + \BB z \right) ,
\eqen
which has the law of a whole-plane GFF plus $-\gamma \log |\cdot|$, normalized so that its circle average over $\bdy \BB D$ is 0 (even if we condition on $\BB z$). 
Fix $\rho \in (0,1)$. 
If we define $u_*^{3,\ep}(\cdot)$ and $u_*^{4,\ep}(\cdot)$ as in~\eqref{eqn-area-diam-deg-def*} with $h=h^3$ and $h = h^4$, respectively, then by the $\gamma$-LQG coordinate change formula $u_*^{3,\ep}(\cdot) = u_*^{4, 2^{-\gamma Q} \ep}( 2(\cdot + \BB z)    )$. Hence for a bounded measurable function $f : B_{\rho}(0) \rta [0,\infty)$, 
\eqb \label{eqn-0-to-gamma-coord}
\int_D f(z)  u_*^{4,\ep}(z) \, dz = \frac12 \int_{ D_{\BB z} } f\left( 2 (w + \BB z ) \right) u_*^{3, c \ep}(w) \, dw  \quad \op{for} \quad D_{\BB z} = \frac12 D + \BB z 
\quad \op{and} \quad c = 2^{-\gamma Q} .
\eqe
This does not immediately imply~\eqref{eqn-lln-statement} for $ h = h^4$ since $f\left( 2 (w+\BB z) \right)$ is not deterministic (it depends on $\BB z$). 
To get around this difficulty, we will compare $f(2(w+\BB z))$ to $f(2w)$. 
For this purpose we need to impose the continuity assumptions on $f$ and $D$ appearing in Proposition~\ref{prop-area-diam-deg-gamma}.  

Fix $\beta \in (0,\xi)$ and $C >0$ and assume that $D$ is such that $\op{area}(B_r(\bdy D)) \leq C r$ for each $r > 0$ and $f : D \rta [0,\infty)$ is $C \ep^{-\beta}$-Lipschitz and bounded above by $C\ep^{-\beta}$. 
We also extend $f$ to be identically equal to $C \ep^{-\beta}$ outside of $D$. 
Then if $\BB z \in B_{\ep^\xi}(0)$ and $w \in D_{\BB z} \setminus B_{\ep^\xi}(\bdy D_{\BB z} )$, 
\eqbn
\left| f\left( 2 (w +\BB z) \right)  - f\left( 2w \right) \right| \leq 2 C \ep^{-\beta+ \xi} 
\eqen 
Consequently,
\eqb \label{eqn-0-to-gamma-affine}
\int_{ D_{\BB z} } f\left( 2 (w + \BB z ) \right) u_*^{3, c\ep}(z) \, dz 
 \leq \int_{(1/2 + \ep^\xi) D} \left(   f(2w) + 2 C \ep^{-\beta + \xi} \right)  u_*^{3, c\ep}(w) \, dw  + C \ep^{-\beta} \int_{B_{2\ep^\xi}(\bdy D)} u_*^{3,c\ep}(w) \,dw .
\eqe 
Note that the second term comes from the integral over $D_{\BB z} \cap  B_{\ep^\xi}(\bdy D_{\BB z} )$. 
By~\eqref{eqn-lln-statement} in the case when $h$ is sampled from the conditional law of $h^3$ given $\{\BB z \in B_{\ep^\xi}(0)\}$ (applied with  $\rho/2$ in place of $\rho$ and to each of the function/domain pairs $(f(2\cdot) , (1/2 + \ep^\xi) D)$, $(1 , (1/2 + \ep^\xi) D)$, and $(1 , B_{2\ep^\xi}(\bdy D)$), there exists $\alpha_0 =\alpha_0 (\gamma ) > 0$ and $A_0 = A_0(\rho,\gamma) > 0$ such that with probability at least $1- O_\ep(\ep^{\alpha_0} )$, the event in~\eqref{eqn-lln-E_0} holds with $h = h^3$ and the right side of~\eqref{eqn-0-to-gamma-affine} is at most
\eqb \label{eqn-0-to-gamma-affine'}
A_0 \int_{ \frac12 D}  f(2w) \,dw + A_0 C \ep^{-\beta + \xi} +  C A_0 \ep^{-\beta} \op{area}\left( B_{2\ep^\xi}\left( \tfrac12 \bdy D \right) \right) + \ep^{\alpha_0} \| f\|_\infty .
\eqe
Note that here we bound the integral of $f$ over $(1/2 +\ep^\xi) D\setminus \frac12 D$ by $C \ep^{-\beta}  \op{area}\left( B_{2\ep^\xi}\left( \tfrac12 \bdy D \right) \right) $. 
By assumption, $\|f\|_\infty \leq C\ep^{-\beta}$ and $ \op{area}\left( B_{2\ep^\xi}\left( \tfrac12 \bdy D \right) \right)  \leq O_\ep(\ep^\xi)$, so (if we take $\beta < \alpha_0$) then the sum of the last three terms on right side of~\eqref{eqn-0-to-gamma-affine'} is bounded above by $O_\ep(\ep^\alpha)$ for an appropriate $\alpha=\alpha(\gamma) > 0$. 
From this,~\eqref{eqn-0-to-gamma-coord}, and~\eqref{eqn-0-to-gamma-affine}, we infer that the conclusion of Proposition~\ref{prop-area-diam-deg-gamma} is true with $h^4$ in place of $h$. 
\medskip

\noindent\textit{Proofs of Propositions~\ref{prop-area-diam-deg-gamma*} and~\ref{prop-area-diam-deg-gamma}.} If $h$ is the circle-average embedding of a $\gamma$-quantum cone in $(\BB C , 0,\infty)$, then $h^4|_{\BB D} \eqD h|_{\BB D}$. Furthermore, by Lemma~\ref{lem-u_*-msrble}, $u_*^\ep(z)$ is a.s.\ determined by $h^{\op{IG}}$ and $h|_{B_{4\ep^q}(z)}$ on the event $\{\op{diam}(H_{x_z^\ep}^\ep \leq \ep^q\}$. From this, we infer~\eqref{eqn-lln-statement} with this choice of $h$ and any $\rho \in (0,1)$, and hence Proposition~\ref{prop-area-diam-deg-gamma*}, from~\eqref{eqn-lln-statement} with $h^4$ in place of $h$. 

Proposition~\ref{prop-area-diam-deg-gamma} is an immediate consequence of Proposition~\ref{prop-area-diam-deg-gamma*} and~\eqref{eqn-u-u_*-compare}. \qed

\appendix

\section{Estimates for the GFF}
\label{sec-gff-abs-cont}

In this appendix, we record several facts about various types of Gaussian free field which are needed in the proofs of our main results. Many of these lemmas state that certain GFF-type distributions are absolutely continuous with respect to one another, often with quantitative bounds for the Radon-Nikodym derivatives. The results of this appendix are technical in nature and their proofs do not rely on any other results from the paper (actually, we use only standard formulas for the GFF), so we collect them here to avoid interrupting the flow of the main argument.

\subsection{Conditioning on the average over a large circle}
\label{sec-big-circle-rn}

In this subsection, we record a lemma which makes the following intuitively obvious statement precise. If $h$ is a whole-plane GFF normalized so that $h_1(0) = 0$, then conditioning on the circle-average of $h$ over a large circle $\bdy B_R(w)$ which surrounds $\BB D$ does not have a large effect on the conditional law of its restriction to $\BB D$. The main point of the lemma is that the circle $\bdy B_R(w)$ is not required to be centered at 0, even though we normalize the field so that its circle average over $\bdy\BB D$ is 0. In the case when $w = 0$, conditioning on $h_R(w)$ has no effect on $h|_{\BB D}$ since $t\mapsto h_{e^{-t}}(0)$ evolves as a standard linear Brownian motion~\cite[Proposition 3.3]{shef-kpz}.

\begin{lem} \label{lem-big-circle-rn}
Let $h$ be a whole-plane GFF normalized so that $h_1(0) =0$. 
Also fix $\rho  \in (0,1)$ and let $R \geq (1-\rho)^{-1}$ and $w\in\BB C$ be such that $|w| \leq \rho R$, so that $\BB D\subset B_R(w)$. For $\frk a \in \BB R$, the conditional law of $h|_{\BB D}$ given $\{h_R(w) = \frk a\}$ is absolutely continuous with respect to the unconditional law of $h|_{\BB D}$. Furthermore, for each $p_0 > 0$ there are constants $R_* \geq (1-\rho)^{-1}$ and $c > 0$ depending only on $p_0$ and $\rho$ such that for each $p \in [-p_0, \infty)$ and $R\geq R_*$, the Radon-Nikodym derivative $M_{\frk a}$ satisfies
\eqb \label{eqn-big-circle-rn}
\BB E\left[ M_{\frk a}^p \right] \preceq \exp\left(  \frac{c |p| \frk a^2 | w|^2 }{R^4  (\log R)^2} \right) ,
\eqe
with the implicit constant depending only on $\rho$ and $p$. 
\end{lem}
\begin{proof}
Since $t\mapsto h_{e^{-t}}(0)$ evolves as a standard linear Brownian motion~\cite[Proposition~3.3]{shef-kpz}, the statement of the lemma in the case $w=0$ is immediate. Henceforth assume $w\not=0$. 
We will compute the conditional law of $h_R(w)$ given $h|_{\BB D}$ and apply Bayes' rule.

By Lemma~\ref{lem-whole-plane-markov}, we can write $h|_{\BB C \setminus \ol{\BB D}} = \frk h   + \rng h$, where $\frk h$ is a random harmonic function on $\BB C\setminus \ol{\BB D}$ which is determined by $h|_{\ol{\BB D}}$ and $\rng h$ is a zero-boundary GFF on $\BB C\setminus \ol{\BB D}$ which is independent from $h|_{\ol{\BB D}}$.  
The image of the circle $B_R(w)$ under the inversion map $z\mapsto 1/z$ is the circle of radius $1/\wt R$ and center $\wt w$, where
\eqb  \label{eqn-big-circle-rn-invert}
\wt R: = \frac{R^2-|w|^2}{R}   \quad\op{and} \quad \wt w :=  - \frac{w}{R^2 - |w|^2} 
\eqe 
Note that $\wt R \in [(1-\rho^2) R , R]$ and $|\wt w| \leq   \rho (1-\rho^2)^{-1} R^{-1}$. 

By applying~\cite[Proposition~3.2]{shef-kpz} to the inverted GFF $h(1/\cdot)$, we see that $\rng h_R(w)$ is centered Gaussian with variance $\log \wt R + \log(1-|\wt w|^2)$. Hence the conditional law of $h_R(w)$ given $h|_{\BB D}$ is Gaussian with mean equal to the circle average $\frk h_R(w)$ and variance $\log \wt R + \log(1-|\wt w|^2)$. 

Furthermore, $\frk h(1/\cdot)$ is the harmonic part of $h(1/\cdot)|_{\BB D}$, so by the mean value property of harmonic functions, $\frk h_R(w) = \frk h(1/\wt w)$. By~\cite[Lemma~6.4]{ig1} applied to $h(1/\cdot)$, we infer that $\frk h_R(w)$ is centered Gaussian with variance $ \log( (1-|\wt w|^2)^{-1} ) \preceq 1$. 
Since $h_R(w) = \frk h_R(w) + h_R^0(w)$ and the two summands are independent, it follows that the marginal law of $h_R(w)$ is Gaussian with mean $\log \wt R$. 

By combining the above descriptions of the laws of $h_R(w)$ and $\rng h_R(w)$ and applying Bayes' rule for conditional densities, we get the absolute continuity in the statement of the lemma with Radon-Nikodym derivative
\alb
M_{\frk a}
&=   \frac{\sqrt{\log \wt R}}{\sqrt{ \log \wt R + \log(1-|\wt w|^2) }} \exp\left(  \frac{\frk a^2}{2\log \wt R}  - \frac{(\frk a - \frk h_R(w))^2}{2( \log \wt R + \log(1-|\wt w|^2) )}  \right) \\
&\asymp   \exp\left(  \frac{\frk a^2}{2\log \wt R}  - \frac{(\frk a - \frk h_R(w))^2}{2 (\log \wt R + \log(1-|\wt w|^2) )}  \right) .
\ale

We now estimate the moments of $M_{\frk a}$. 
Integrating against the law of $\frk h_R(w)$ and evaluating a Gaussian integral shows that if $p\in\BB R$ is such that
\eqb \label{eqn-circle-avg-min-p}
\frac{1}{\log((1-|\wt w|^2)^{-1})} + \frac{p}{\log \wt R + \log(1-|\wt w|^2)} > 0,
\eqe 
then
\allb  \label{eqn-circle-avg-compute}
\BB E\left[ M_{\frk a}^p    \right] 
&\asymp \frac{\exp\left(  \frac{\frk a^2 p }{2\log \wt R} - \frac{\frk a^2 p}{2\left( \log \wt R + \log(1-|\wt w|^2)   - p\log(1-|\wt w|^2)   \right)} \right)}{ \sqrt{\log((1-|\wt w|^2)^{-1} ) } \sqrt{ \frac{1}{\log((1-|\wt w|^2)^{-1} )    } + \frac{p}{ \log \wt R + \log(1-|\wt w|^2)    }   } } \notag \\
&\asymp  \exp\left(  \frac{\frk a^2 p }{2 }\left(  \frac{  (p-1) \log((1-|\wt w|^2)^{-1})  }{\log \wt R ( \log \wt R + (p-1) \log((1-|\wt w|^2)^{-1}) )}    \right)   \right) ,
\alle
where in the second proportionality we use that $ \log((1-|\wt w|^2)^{-1} )  \preceq 1$. 

By~\eqref{eqn-big-circle-rn-invert}, $\log( ( 1 - |\wt w|^2 )^{-1} ) \asymp |\wt w|^2 \asymp |w|^2/R^4$.  
Moreover, if we are given $p_0  >0$ and $R$ is sufficiently large, depending only on $\rho$ and $p_0$, then for each $p\geq -p_0$ the relation~\eqref{eqn-circle-avg-min-p} holds and in fact $\log \wt R + (p-1) \log((1-|\wt w|^2)^{-1}) \succeq \log R $,
 with the implicit constant depending only on $\rho$. Plugging these two estimates into~\eqref{eqn-circle-avg-compute} shows that~\eqref{eqn-big-circle-rn} holds for $p \geq -p_0$. 
\end{proof}

By translating and scaling, we deduce from Lemma~\ref{lem-big-circle-rn'} an estimate for a whole-plane GFF normalized so that $h_1(0) = 0$. 

\begin{lem} \label{lem-big-circle-rn'}
Let $h$ be a whole-plane GFF normalized so that its circle average over $\bdy\BB D$ is 0. Also let $\rho \in (0,1]$ and $\delta \in (0, 1-\rho)$. 
For $\frk a \in \BB R$ and $z\in B_{ \rho}(0)$, the conditional law of $h|_{B_\delta(z)}$ given $\{h_\delta(z) = \frk a\}$ is absolutely continuous with respect to the law of a whole-plane GFF in $B_\delta(z)$ normalized so that its circle average over $\bdy B_\delta(z)$ is $\frk a$. 
Furthermore, for each $p_0 > 0$ there are constants $c>0$ and $\delta_*  \in (0,\rho^2]$, depending only on $\rho$ and $p_0$, such that for each $p \in [-p_0,\infty)$ and each $\delta \in (0,\delta_*]$ the Radon-Nikodym derivative $M_{\frk a}$ satisfies
\eqb \label{eqn-big-circle-rn'}
\BB E\left[ M_{\frk a}^p \right] \preceq \exp\left(  \frac{c |p|  \delta^2 \frk a^2 | z|^2}{(\log \delta^{-1} )^2} \right)  
\eqe 
with the implicit constant depending only on $\rho$ and $p$.  
\end{lem} 
\begin{proof} 
Let $h^{z,\delta} := h(\delta \cdot + z) - h_\delta(z) $. Then $h^{z,\delta}$ has the law of a whole-plane GFF normalized so that its circle average over $\bdy \BB D$ is 0 and $h_\delta(z) = -h^{z,\delta}_{\delta^{-1}}(-\delta^{-1} z)$. The conditional law of $h|_{B_\delta(z)}$ given $\{h_\delta(z) = \frk a\}$ is the same as the conditional law of $h^{z,\delta}(\delta^{-1}(\cdot - z)) + \frk a$ given $\{ h^{z,\delta}_{\delta^{-1}}(-\delta^{-1} z) =  - \frk a\}$. The statement of the lemma therefore follows from the invariance of the law of the whole-plane GFF under complex affine transformations, modulo additive constant, together with Lemma~\ref{lem-big-circle-rn} applied with $R=\delta^{-1}$ and $w = -\delta^{-1} z$.
\end{proof}

\subsection{Long-range independence}
\label{sec-long-range-ind}

The goal of this subsection is to prove the following lemma, which tells us that (roughly speaking) for a whole-plane GFF $h$ and a small $\delta \in (0,1)$, the only information we need about $h|_{\BB C\setminus \BB D}$ to determine most of the information about $h|_{B_\delta(0)}$ is the circle average $h_1(0)$. 
We will eventually need to apply the analogous fact for a pair of GFFs, namely an embedding of a $\gamma$-quantum cone together with the independent whole-plane GFF used to generate an independent whole-plane space-filling SLE$_{\kappa'}$ as in Section~\ref{sec-wpsf-prelim}. 
So, we state the lemma for an $N$-tuple of independent whole-plane GFFs rather than a single GFF.

\begin{lem} \label{lem-rn-mean} 
Let $N\in\BB N$ and let $B_R(w)$ be a ball which contains $\BB D$. Let $  \bd h  = ( h^1,\dots , h^N)$ be an $N$-tuple of i.i.d.\ whole-plane GFFs, each normalized so that its circle average over $\bdy B_R(w)$ is 0. For $z\in \BB C$ and $r >0$, let $ \bd h_r(z) := ( h^1_r(z) ,\dots ,   h_r^N(z))$ be the $N$-tuple of radius-$r$ circle averages at $z$. 
For $p' > p > 1$ and $s \in (0,1)$, there is are constants $a = a(p,p' ,N)>0$ and $b = b(p,p',s,N) > 0$ (which do not depend on $w,R$) such that for $\delta\in \left(0, a  \right]$, the following is true. Let $X$ be a random variable which is a.s.\ determined by $\bd h|_{B_\delta(0)}$ and $\bd h_1(0)$. Then
\allb \label{eqn-rn-mean}
\BB E\left[ \bigg| \BB E\left[ X   \,|\,   \bd h|_{\BB C\setminus \BB D}    \right] - \BB E\left[ X  \,|\, \bd h_1(0) \right]    \bigg|^p \right] 
\preceq   \delta^{s p} \BB E\left[ |X|^p   \right]  +  e^{-b/\delta^{1-s}} \BB E\left[ |X|^{p'} \right]^{p/p'}
\alle 
with implicit constant depending only on $p,p',s,$ and $N$. 
\end{lem}

The reason why we normalize the fields in Lemma~\ref{lem-rn-mean} so that their circle averages over $\bdy B_R(w)$ are 0 is because $\bdy B_R(w) $ is disjoint from $\BB D$, so we can apply the last assertion of Lemma~\ref{lem-whole-plane-markov} (with $h(R\cdot+w)$ in place of $h$). The proof of Lemma~\ref{lem-whole-plane-markov}, and hence also the proof of Lemma~\ref{lem-rn-mean}, works verbatim if we replace the circle average over $\bdy B_R(w)$ by, e.g., the distributional pairing $(h,\psi)$ for some fixed test function $\psi$ which is supported on $h|_{\BB C\setminus \BB D}$ and whose inverse Laplacian has finite Dirichlet energy (the constants $a$ and $b$ and the implicit constants in~\eqref{eqn-rn-mean} do not depend on $\psi$). 

To prove Lemma~\ref{lem-rn-mean}, we will bound for a whole-plane GFF $ h$ the Radon-Nikodym derivative of the conditional law of $ h|_{B_\delta(0) }$ given $ h|_{\BB C\setminus \BB D}$ with respect to its conditional law given only $h_1(0)$. For this purpose we need the following estimate for the harmonic part of $h|_{\BB D}$.

\begin{lem} \label{lem-harmonic-part-law}
Let $ h$ be a whole-plane GFF normalized (with any choice of additive constant) and let $\frk h$ be the harmonic part of $ h|_{\BB D}$ as defined just after Lemma~\ref{lem-whole-plane-markov'}. There is a universal constant $a_0 > 0$ such that for $\delta \in (0,1/4)$, 
\eqbn
\BB E\left[ \sup_{z\in B_{\delta  }(0)} \exp\left( \frac{a_0}{\delta^2}    (\frk h(z) - \frk h(0))^2 \right) \right] \preceq 1
\eqen
with universal implicit constant.  
\end{lem}
\begin{proof}
By the mean value property of harmonic functions,  
\eqbn
\sup_{z\in B_{\delta}(0) } |\frk h(z) -\frk h(0) |   \preceq \frac{1}{\delta^2 } \int_{B_{2\delta }(0) } |\frk h(z)-\frk h(0) |  \, dz  ,
\eqen 
with universal implicit constant. 
By combining this with Jensen's inequality, applied to the convex function $x\mapsto e^{a_0 x^2}$, we find that for $a_0  > 0$, 
\eqbn
\sup_{z\in B_{\delta }(0) } \exp\left( \frac{a_0}{\delta^2}  (\frk h(z)-\frk h(0))^2   \right) 
\preceq  \frac{1}{  \delta^2 } \int_{B_{2\delta }(0) } \exp\left( \frac{4\pi a_0}{\delta^2} (\frk h(z) -\frk h(0))^2 \right) \, dz   
\eqen 
with universal implicit constant. 
By~\cite[Lemma~6.4]{ig1}, for $z\in B_\delta(0)$ the random variable $\frk h(z) -\frk h(0) $ is centered Gaussian with variance $-\log \left( 1 - \left| z \right|^2 \right)$. This variance is bounded above by a universal constant times $\delta^2$ for $z\in B_{\delta}(0)$. Hence for a small enough universal choice of $a_0 > 0$, 
\eqbn
\BB E\left[\frac{1}{  \delta^2 } \int_{B_{2\delta }(0) } \exp\left( \frac{4\pi a_0}{\delta^2} (\frk h(z) -\frk h(0))^2 \right) \, dz     \right]  \preceq 1 . \qedhere
\eqen 
\end{proof}

The following Radon-Nikodym derivative estimate, which compares the conditional law of $h|_{B_\delta(0)}$ given $h|_{\BB C\setminus \BB D}$ to its conditional law given only $h_1(0)$, is the key input in the proof of Lemma~\ref{lem-rn-mean}. 
In the statement, we will actually compare the conditional law of $h|_{B_\delta(0)}$ given $h|_{\BB C\setminus \BB D}$ to the conditional law of $\ol h|_{B_\delta(0)}$ given $h_1(0)$, where here $\ol h$ is another field coupled with $h$ in such a way that $h_1(0) = \ol h_1(0)$; we find that this  makes our moment estimate for the Radon-Nikodym derivative more clear.

\begin{lem}  \label{lem-harmonic-part-rn}   
Fix $R > 1$ and $w \in\BB C$ such that $\BB D \subset B_R(w)$. 
Let $ h$ and $\ol h$ be whole-plane GFFs, normalized so that their circle averages over $\bdy B_R(w)$ are 0, coupled together so that the circle averages $h_1(0)$ and $\ol h_1(0)$ agree and $h$ and $\ol h$ are conditionally independent given these circle averages (note that the circle averages of our GFFs over both $\bdy B_R(w)$ and $\bdy \BB D$ agree).
For $\delta \in (0,1)$, the conditional law of $h|_{B_{ \delta}(0)}$ given $h|_{\BB C\setminus \BB D}$ is a.s.\ absolutely continuous with respect to the conditional law of $\ol h|_{B_{\delta  }(0)}$ given $\ol h_1(0)$.
Let $M_\delta  = M_\delta\left( h|_{\BB C\setminus \BB D}  , \ol h|_{B_\delta(0) }\right)$ be the Radon-Nikodym derivative. There is a universal constant $a \in (0,1)$ such that for $\delta \in (0,a ]$, 
\eqb \label{eqn-rn-moment}
\BB E\left[ M_\delta^{a/\delta }   \right] \preceq 1 \quad\op{and}\quad \BB E\left[ M_\delta^{-a/\delta }   \right] \preceq 1 
\eqe 
with universal implicit constants. 
\end{lem}
\begin{proof}
Since $\bdy B_R(w) \cap\BB D =\emptyset$, by Lemma~\ref{lem-whole-plane-markov} (applied with $h(R\cdot+w)$ in place of $h$) we can write $h|_{\BB D} = \frk h   + \rng h$, where $\frk h$ is a random harmonic function on $\BB D$ which is determined by $h|_{\BB C\setminus \BB D}$ and $\rng h$ is a zero-boundary GFF on $\BB D$ which is independent from $h|_{\BB C\setminus \BB D}$.  
Decompose $\ol h|_{\BB D} = \ol{\frk h} +  \rng{\ol h}$ analogously. 
By our choice of coupling, $\frk h(0) = \ol{\frk h}(0) = h_1(0) = \ol h_1(0)$.
Furthermore, conditional on $(h,\ol h)|_{\BB C\setminus \BB D}$ (which a.s.\ determines $\frk h$ and $\ol{\frk h}$) the fields $\rng h  $ and $\rng{\ol h}$ are conditionally independent zero-boundary GFFs on $\BB D$. 
 
Let $\phi_1$ be a deterministic smooth bump function taking values in $[0,1]$ which equals 1 on $B_1(0)$ and 0 on $\BB C\setminus B_2(0)$. Let $\phi_\delta(z) := \phi_1(z/\delta)$, so that $\phi$ is supported on $B_{2\delta }(0)$ and is identically equal to 1 on $B_{\delta }(0)$. Also let $g_\delta := (\frk h - \ol{\frk h}) \phi_\delta $. 
If we condition on $(h,\ol h)|_{\BB C\setminus \BB D}$, then the proof of~\cite[Proposition~3.4]{ig1} shows that the conditional laws of $h|_{B_{\delta }(0)}$ and $\ol h |_{B_{\delta  }(0)}$ are mutually absolutely continuous, and the Radon-Nikodym derivative of the former with respect to the latter is given by 
\eqbn
 \BB E\left[  \exp\left( (h , g_\delta )_\nabla - \frac12 (g_\delta ,g_\delta)_\nabla \right)  \,|\,  \frk h , \ol{\frk h}, \ol h|_{B_{\delta}(0)} \right] 
\eqen
where $(\cdot,\cdot)_\nabla$ denotes the Dirichlet inner product. Averaging over the possible realizations of $\ol{\frk h}$ shows that that the Radon-Nikodym derivative of the conditional law of $h|_{B_{\delta }(0)}$ given $h|_{\BB C\setminus \BB D}$ with respect to the conditional law of $\ol h|_{B_\delta(0)}$ given $h|_{\BB C\setminus \BB D}$ is equal to
\eqbn
M_\delta = \BB E\left[ \exp\left( (h , g_\delta)_\nabla - \frac12 (g_\delta , g_\delta)_\nabla \right)  \,|\, \frk h , \ol h|_{B_{\delta}(0)} \right] .
\eqen
Note that $M_\delta$ is also the Radon-Nikodym derivative of the conditional law of $h|_{B_{\delta }(0)}$ given $h|_{\BB C\setminus \BB D}$ with respect to the conditional law of $\ol h|_{B_\delta(0)}$ given $\ol h_1(0)$ since $h|_{\BB C\setminus \BB D}$ determines $h_1(0) = \ol h_1(0)$ and $\ol h$ is conditionally independent from $h$ given $\ol h_1(0)$. 

We now estimate $M_\delta $. By Jensen's inequality, for $\delta\in (0, a ]$, 
\eqbn
M_\delta^{a /\delta  } \leq \BB E\left[ \exp\left( \frac{a}{\delta } (h , g_\delta)_\nabla - \frac{a}{2\delta } (g_\delta, g_\delta)_\nabla \right)  \,|\, \frk h , \ol h|_{B_{\delta}(0)} \right] 
\eqen
so
\eqb \label{eqn-rn-moment-to-gff}
\BB E\left[ M_\delta^{a/\delta  }   \right] \leq \BB E\left[ \exp\left( \frac{a}{\delta } (h , g_\delta)_\nabla - \frac{a}{2\delta } (g_\delta,g_\delta)_\nabla \right)    \right] .
\eqe 
If we condition on $\frk h$ and $\ol{\frk h}$, the conditional law of $(h , g_\delta)_\nabla$ is centered Gaussian with variance $(g_\delta,g_\delta)_\nabla$ (note that $(\frk h , g_\delta)_\nabla = 0$ since $\frk h$ is harmonic in $\BB D$ and $g_\delta$ is compactly supported in $\BB D$). Hence, by first taking the conditional expectation given $\frk h$ and $\ol{\frk h}$,
\allb\label{eqn-rn-gff-moment}
\BB E\left[ \exp\left( \frac{a}{\delta } (h, g_\delta)_\nabla - \frac{a}{2\delta } (g_\delta,g_\delta)_\nabla \right)   \right]
&= \BB E\left[ \exp\left( \left( \frac{a^2}{2\delta^2}  - \frac{a}{2\delta } \right) (g_\delta,g_\delta)_\nabla \right)  \right] \notag  \\
&\leq  \BB E\left[ \exp\left( \frac{a^2}{2\delta^2}   (g_\delta,g_\delta)_\nabla \right)   \right] .
\alle 
 By integration by parts,
\eqb \label{eqn-rn-parts}
(g_\delta,g_\delta)_\nabla 
= \int_{\BB D}  \left( \frac12 \Delta(\phi_\delta^2)(z) - \phi_\delta(z) \Delta \phi_\delta(z) \right)    (\frk h(z) - \ol{\frk h}(z) )^2 \,dz .
\eqe 
Since $\phi_1$ is smooth and supported on $B_{2\delta}(0)$, the function $\frac12 \Delta(\phi_\delta^2)  - \phi_\delta  \Delta \phi_\delta $ is bounded above by a constant $c_0$ (depending only on $\phi_1$) times $\delta^{-2}$ and is supported on $B_{2\delta}(0)$. By combining this with~\eqref{eqn-rn-moment-to-gff},~\eqref{eqn-rn-gff-moment}, and~\eqref{eqn-rn-parts}, we get
\allb \label{eqn-rn-moment'} 
\BB E\left[ M_\delta^{a/\delta }  \right]
&\leq \BB E\left[ \exp\left( \frac{c_0 a^2  }{ 2 \delta^4} \int_{B_{2\delta  }(w)}  (\frk h(z) - \ol{\frk h}(z) )^2 \,dz    \right) \right] \notag\\
&\leq \BB E\left[ \sup_{z\in B_{2\delta}(0) }   \exp\left( \frac{2\pi c_0 a^2  }{ \delta^2}(\frk h(z) - \ol{\frk h}(z) )^2 \right)  \right].
\alle
The random functions $\frk h -  h_1(0)$ and $\ol{\frk h} -\ol h_1(0)$ are i.i.d.\ and centered Gaussian so $\frk h - \ol{\frk h} \eqD \sqrt 2( \frk h - h_1(0) )$. Hence the first estimate in~\eqref{eqn-rn-moment} for a small enough universal choice of $a > 0$ follows from Lemma~\ref{lem-harmonic-part-law}.
We similarly obtain the second estimate in~\eqref{eqn-rn-moment}. 
\end{proof}

\begin{proof}[Proof of Lemma~\ref{lem-rn-mean}]
By considering the positive and negative parts $X\BB 1_{(X\geq 0)}$ and $X\BB 1_{(X\leq 0)}$ separately, we can assume without loss of generality that $X$ is non-negative. We make this assumption throughout the proof.

In order to apply Lemma~\ref{lem-harmonic-part-rn}, we let $\ol{\bd h} = (\ol h^1,\dots, \ol h^N)$ be another $N$-tuple of independent GFFs with the same law as $\bd h$, coupled together with $\bd h$ in such a way that $\bd h_1(0) = \ol{\bd h}_1(0)$ and $\bd h$ and $\ol{\bd h}$ are conditionally independent given this circle average. Let $\ol X = X(\ol{\bd h}|_{B_\delta(0)} , \bd h_1(0) )$ be determined by $\ol{\bd h}|_{B_\delta(0)}$ and $\bd h_1(0)$ in the same manner that $X$ is determined by $\bd h|_{B_\delta(0)}$ and $\bd h_1(0)$.  

For $k\in [1,N]_{\BB Z}$, let $M_\delta^k = M_\delta^k(h^k|_{\BB C\setminus \BB D} , \ol h^k|_{B_\delta(0) } )$ be the Radon-Nikodym derivative of the conditional law of $h^k|_{B_\delta(0)}$ given $h^k|_{\BB C\setminus \BB D}$ with respect to the law of $\ol h^k |_{B_\delta(0)}$, as in Lemma~\ref{lem-harmonic-part-rn}.
Then the $M_\delta^k$'s are independent and
\eqb \label{eqn-prod-rn}
\bd M_\delta := \prod_{k=1}^N M_\delta^k 
\eqe 
is the Radon-Nikodym derivative of the conditional law of $\bd h|_{B_{ \delta}(0)}$ given $\bd h|_{\BB C\setminus \BB D}$ w.r.t.\ the conditional law of $\ol{\bd h}|_{B_\delta(0) }$ given $ \bd h_1(0)$ (equivalently, by conditional independence, w.r.t.\ the conditional law of $\ol{\bd h}_{B_\delta(0)}$ given $\bd h|_{\BB C\setminus \BB D}$). Hence
\allb \label{eqn-rn-mean-pos-decomp+}
\BB E\left[ X    \,|\,   \bd h|_{\BB C\setminus \BB D}   \right]
&= \BB E\left[ \bd M_\delta \ol X    \,|\, \bd h|_{\BB C\setminus \BB D} \right] \notag \\
&\leq (1+\delta^s) \BB E\left[ \ol X  \BB 1_{(\bd M_\delta \leq 1 + \delta^s)}   \,|\, \bd h|_{\BB C\setminus \BB D} \right] 
+ \BB E\left[ \bd M_\delta \ol X   \BB 1_{(\bd M_\delta > 1 + \delta^s)} \,|\, \bd h|_{\BB C\setminus \BB D} \right] \notag \\
&\leq (1+\delta^s) \BB E\left[ X  \,|\,  \bd h_1(0)    \right]
+ \BB E\left[ \bd M_\delta \ol X   \BB 1_{(\bd M_\delta > 1 + \delta^s)} \,|\, \bd h|_{\BB C\setminus \BB D} \right]   
\alle
where in the last line we use that the conditional law of $\ol X$ given ${\bd h}|_{\BB C\setminus \BB D}$ is the same as the conditional law of $X$ given $\bd h_1(0)$ (by our choice of coupling). Similarly,
\allb \label{eqn-rn-mean-pos-decomp-}
\BB E\left[ X  \,|\,   \bd h|_{\BB C\setminus \BB D}    \right]
&\geq  (1-\delta^s) \BB E\left[ \ol X  \BB 1_{(\bd M_\delta \geq 1 - \delta^s)}   \,|\, \bd h|_{\BB C\setminus \BB D} \right]   \notag \\
&\geq (1-\delta^s) \BB E\left[ X  \,|\,  \bd h_1(0)    \right]
- \BB E\left[ \ol X \BB 1_{(\bd M_\delta < 1 - \delta^s)} \,|\, \bd h|_{\BB C\setminus \BB D} \right]  . 
\alle
By~\eqref{eqn-rn-mean-pos-decomp+},~\eqref{eqn-rn-mean-pos-decomp-}, and Jensen's inequality,
\eqb \label{eqn-rn-mean-pos-abs}
\BB E\left[ \left| \BB E\left[ X   \,|\,   \bd h|_{\BB C\setminus \BB D}    \right] - \BB E\left[ X  \,|\,  \bd h_1(0)    \right]    \right|^p \right] 
\preceq   \delta^{s p} \BB E\left[ X^p       \right]  + \BB E\left[ \bd M_\delta^p \ol X^p   \BB 1_{(\bd M_\delta > 1 + \delta^s)} \right] 
   +   \BB E\left[   \ol X^p  \BB 1_{(\bd M_\delta < 1- \delta^s)} \right]  ,
\eqe 
with implicit constant depending only on $p$. 
 
By Lemma~\ref{lem-harmonic-part-rn}, there is a universal constant $a' > 0$ such that for each $\delta \in (0,a']$ and each $k\in [1,N]_{\BB Z}$, we have $\BB E[(M_\delta^k)^{a'/\delta}]\preceq 1$, with a universal implicit constant. 
By this and the Chebyshev inequality,
\eqb \label{eqn-rn-tail}
\BB P\left[  \bd M_\delta >  1 + \delta^{ s}  \right] 
\leq \sum_{k=1}^N \BB P\left[ M_\delta^k > (1+\delta^{ s} )^{1/N} \right] 
\preceq (1+\delta^{ s} )^{-  a' / (N\delta ) } \preceq e^{ - b' s /\delta^{1- s} } 
\eqe
with $b'  >0$ a constant which depends only on $N$ and the implicit constant in $\preceq$ also depends only on $N$.
Note that in the last inequality, we used that $(1+\delta^s)^{1/\delta}  = [(1+\delta^s)^{1/\delta^s}]^{1/\delta^{1-s}} \preceq e^{-1/\delta^{1-s}}$. 

By~\eqref{eqn-rn-tail} and H\"older's inequality (recall that $\bd M_\delta$ has constant-order moments up to order $a'/\delta$ by our choice of $a'$), if we choose $a \in (0,a']$ sufficiently small, in a manner depending only on $p,p',N,$ and $s$, then for $\delta \in (0,a]$,
\eqb \label{eqn-rn-mean-error+}
\BB E\left[ \bd M_\delta^p \ol X^p   \BB 1_{(\bd M_\delta > 1 + \delta^s)} \right]  \preceq e^{-b /\delta^{1-s} } \BB E\left[   X^{p'} \right]^{p/p'}
\eqe
for a constant $b= b(p,p',s,N) > 0$. Similarly,
\eqb \label{eqn-rn-mean-error-}
\BB E\left[  \ol X^p   \BB 1_{(\bd M_\delta < 1- \delta^s)} \right]  \preceq e^{-b  /\delta^{1-s} } \BB E\left[  X^{p'} \right]^{p/p'}
\eqe
for a possibly smaller choice of the constant $b$. 
Combining~\eqref{eqn-rn-mean-pos-abs},~\eqref{eqn-rn-mean-error+}, and~\eqref{eqn-rn-mean-error-} yields~\eqref{eqn-rn-mean} with the above choice of $a$ and $b$. 
\end{proof}

\section{Index of notation}
\label{sec-index}

Here we record some commonly used symbols in the paper, along with their meaning and the location where they are first defined. Other symbols not listed here are only used locally.

\begin{multicols}{2}
\begin{itemize}
\item $\gamma$: LQG parameter; Section~\ref{sec-overview}.
\item $\mcl G^\ep$: mated-CRT map; Sections~\ref{sec-overview}.
\item $h$: main GFF-type distribution; Section~\ref{sec-standard-setup}.
\item $Q = 2/\gamma + \gamma/2$: LQG coordinate change constant; \eqref{eqn-lqg-coord}.
\item $\kappa'  =16/\gamma^2$; SLE parameter; Section~\ref{sec-sle-lqg-prelim}.
\item $\eta$: space-filling SLE$_{\kappa'}$; Section~\ref{sec-wpsf-prelim}. 
\item $h^{\op{IG}}$: imaginary geometry GFF used to construct $\eta$; Section~\ref{sec-wpsf-prelim}.  
\item $o_C^\infty(C)$: a quantity decaying faster than any negative power of $C$; Section~\ref{sec-basic-notation}. 
\item $\mcl V(G)$ and $\mcl E(G)$; vertex and edge sets; Section~\ref{sec-basic-notation}. 
\item $H_x^\ep :=  \eta([x-\ep , x] )$, structure graph cell; \eqref{eqn-cell-def}.
\item $x_z^\ep$: element of $\mcl V\mcl G^\ep$ with $z\in H_x^\ep$; \eqref{eqn-pt-vertex}.
\item $\mcl G^\ep(D)$: subgraph of $\mcl G^\ep$ corresponding to domain $D\subset \BB C$; \eqref{eqn-structure-graph-domain}.
\item $h_r(z)$: circle average;~\cite[Section 3.1]{shef-kpz}.
\item $u^\ep(z)$: diameter$^2$/area times degree of cell $H_{x_z^\ep}^\ep$ containing $z$; \eqref{eqn-area-diam-deg-def}.
\item $\op{Energy}$: Dirichlet energy; Definitions~\ref{def-discrete-dirichlet} and~\ref{def-continuum-dirichlet}.
\item $\tau_z$: time when $\eta$ hits $z$; \eqref{eqn-sle-hit-time}. 
\item $\op{DA}^\ep(z)$: localized version of\\ $ \op{diam}(H_{x_z^\ep}^\ep)^2 / \op{area}(H_{x_z^\ep}^\ep)$; \eqref{eqn-area-diam-localize}.
\item $\ol B_z^\ep$: ball of radius $4 \op{diam}\left( \eta([\tau_z-\ep , \tau_z+\ep]) \right)$ centered at $z$; \eqref{eqn-localize-set}. 
\item $\op{deg}_{\op{in}}^\ep(z)$, $\op{deg}_{\op{out}}^\ep(z)$: localized versions of $\op{deg}(x_z^\ep ; \mcl G^\ep)$; Section~\ref{sec-area-diam-deg-localize}
\item $u_*^\ep(z)$: localized version of $u^\ep(z)$;~\eqref{eqn-area-diam-deg-def*}.
\end{itemize}
\end{multicols}

\bibliography{cibiblong,cibib}

\def\cprime{$'$}
\begin{thebibliography}{ABGGN16}

\bibitem[ABGGN16]{abgn-bdy}
O.~Angel, M.~T. Barlow, O.~Gurel-Gurevich, and A.~Nachmias.
\newblock Boundaries of planar graphs, via circle packings.
\newblock {\em The Annals of Probability}, 44(3):1956--1984, 2016,
  \arxiv{1311.3363}.

\bibitem[AHNR16]{angel-hyperbolic}
O.~Angel, T.~Hutchcroft, A.~Nachmias, and G.~Ray.
\newblock Unimodular hyperbolic triangulations: Circle packing and random walk.
\newblock {\em Inventiones mathematicae}, 206(1):229--268, 2016,
  \arxiv{1501.04677}.

\bibitem[Ald91a]{aldous-crt1}
D.~Aldous.
\newblock The continuum random tree. {I}.
\newblock {\em Ann. Probab.}, 19(1):1--28, 1991. \MR{1085326 (91i:60024)}

\bibitem[Ald91b]{aldous-crt2}
D.~Aldous.
\newblock The continuum random tree. {II}. {A}n overview.
\newblock In {\em Stochastic analysis ({D}urham, 1990)}, volume 167 of {\em
  London Math. Soc. Lecture Note Ser.}, pages 23--70. Cambridge Univ. Press,
  Cambridge, 1991. \MR{1166406 (93f:60010)}

\bibitem[Ald93]{aldous-crt3}
D.~Aldous.
\newblock The continuum random tree. {III}.
\newblock {\em Ann. Probab.}, 21(1):248--289, 1993. \MR{1207226 (94c:60015)}

\bibitem[Ber07a]{bernardi-dfs-bijection}
O.~Bernardi.
\newblock Bijective counting of {K}reweras walks and loopless triangulations.
\newblock {\em J. Combin. Theory Ser. A}, 114(5):931--956, 2007.

\bibitem[Ber07b]{bernardi-maps}
O.~Bernardi.
\newblock Bijective counting of tree-rooted maps and shuffles of parenthesis
  systems.
\newblock {\em Electron. J. Combin.}, 14(1):Research Paper 9, 36 pp.
  (electronic), 2007, \arxiv{math/0601684}. \MR{2285813 (2007m:05125)}

\bibitem[Ber17]{berestycki-gmt-elementary}
N.~Berestycki.
\newblock An elementary approach to {G}aussian multiplicative chaos.
\newblock {\em Electron. Commun. Probab.}, 22:Paper No. 27, 12, 2017,
  \arxiv{1506.09113}. \MR{3652040}

\bibitem[BHS18]{bhs-site-perc}
O.~{Bernardi}, N.~{Holden}, and X.~{Sun}.
\newblock {Percolation on triangulations: a bijective path to Liouville quantum
  gravity}.
\newblock {\em ArXiv e-prints}, July 2018, \arxiv{1807.01684}.

\bibitem[BS01]{benjamini-schramm-topology}
I.~Benjamini and O.~Schramm.
\newblock Recurrence of distributional limits of finite planar graphs.
\newblock {\em Electron. J. Probab.}, 6:no. 23, 13 pp. (electronic), 2001,
  \arxiv{0011019}. \MR{1873300 (2002m:82025)}

\bibitem[Che17]{chen-fk}
L.~Chen.
\newblock Basic properties of the infinite critical-{FK} random map.
\newblock {\em Ann. Inst. Henri Poincar\'e D}, 4(3):245--271, 2017,
  \arxiv{1502.01013}. \MR{3713017}

\bibitem[Cur15]{curien-glimpse}
N.~Curien.
\newblock A glimpse of the conformal structure of random planar maps.
\newblock {\em Comm. Math. Phys.}, 333(3):1417--1463, 2015, \arxiv{1308.1807}.
  \MR{3302638}

\bibitem[DKRV16]{dkrv-lqg-sphere}
F.~David, A.~Kupiainen, R.~Rhodes, and V.~Vargas.
\newblock Liouville quantum gravity on the {R}iemann sphere.
\newblock {\em Comm. Math. Phys.}, 342(3):869--907, 2016, \arxiv{1410.7318}.
  \MR{3465434}

\bibitem[DMS14]{wedges}
B.~{Duplantier}, J.~{Miller}, and S.~{Sheffield}.
\newblock {Liouville quantum gravity as a mating of trees}.
\newblock {\em ArXiv e-prints}, September 2014, \arxiv{1409.7055}.

\bibitem[DS11]{shef-kpz}
B.~Duplantier and S.~Sheffield.
\newblock Liouville quantum gravity and {KPZ}.
\newblock {\em Invent. Math.}, 185(2):333--393, 2011, \arxiv{1206.0212}.
  \MR{2819163 (2012f:81251)}

\bibitem[Dub09]{dubedat-duality}
J.~Dub{\'e}dat.
\newblock Duality of {S}chramm-{L}oewner evolutions.
\newblock {\em Ann. Sci. \'Ec. Norm. Sup\'er. (4)}, 42(5):697--724, 2009,
  \arxiv{0711.1884}. \MR{2571956 (2011g:60151)}

\bibitem[GGN13]{gn-recurrence}
O.~Gurel-Gurevich and A.~Nachmias.
\newblock Recurrence of planar graph limits.
\newblock {\em Ann. of Math. (2)}, 177(2):761--781, 2013, \arxiv{1206.0707}.
  \MR{3010812}

\bibitem[GH18]{gh-displacement}
E.~{Gwynne} and T.~{Hutchcroft}.
\newblock {Anomalous diffusion of random walk on random planar maps}.
\newblock {\em ArXiv e-prints}, July 2018, \arxiv{1807.01512}.

\bibitem[GHM15]{ghm-kpz}
E.~{Gwynne}, N.~{Holden}, and J.~{Miller}.
\newblock {An almost sure KPZ relation for SLE and Brownian motion}.
\newblock {\em ArXiv e-prints}, December 2015, \arxiv{1512.01223}.

\bibitem[GHS17]{ghs-map-dist}
E.~{Gwynne}, N.~{Holden}, and X.~{Sun}.
\newblock {A mating-of-trees approach for graph distances in random planar
  maps}.
\newblock {\em ArXiv e-prints}, November 2017, \arxiv{1711.00723}.

\bibitem[GHS19]{ghs-dist-exponent}
E.~{Gwynne}, N.~{Holden}, and X.~{Sun}.
\newblock {A distance exponent for Liouville quantum gravity}.
\newblock {\em {Probability Theory and Related Fields}}, 173(3):931--997, 2019,
  \arxiv{1606.01214}.

\bibitem[GKMW18]{gkmw-burger}
E.~Gwynne, A.~Kassel, J.~Miller, and D.~B. Wilson.
\newblock Active {S}panning {T}rees with {B}ending {E}nergy on {P}lanar {M}aps
  and {SLE}-{D}ecorated {L}iouville {Q}uantum {G}ravity for {$\kappa > 8$}.
\newblock {\em Comm. Math. Phys.}, 358(3):1065--1115, 2018, \arxiv{1603.09722}.
  \MR{3778352}

\bibitem[GM17]{gm-spec-dim}
E.~{Gwynne} and J.~{Miller}.
\newblock {Random walk on random planar maps: spectral dimension, resistance,
  and displacement}.
\newblock {\em ArXiv e-prints}, November 2017, \arxiv{1711.00836}.

\bibitem[GM19]{gm-uniqueness}
E.~Gwynne and J.~Miller.
\newblock Existence and uniqueness of the {L}iouville quantum gravity metric
  for {$\gamma \in (0,2)$}.
\newblock 2019.

\bibitem[GMS17]{gms-tutte}
E.~{Gwynne}, J.~{Miller}, and S.~{Sheffield}.
\newblock {The Tutte embedding of the mated-CRT map converges to Liouville
  quantum gravity}.
\newblock {\em ArXiv e-prints}, May 2017, \arxiv{1705.11161}.

\bibitem[GMS18]{gms-random-walk}
E.~{Gwynne}, J.~{Miller}, and S.~{Sheffield}.
\newblock {An invariance principle for ergodic scale-free random environments}.
\newblock {\em ArXiv e-prints}, July 2018, \arxiv{1807.07515}.

\bibitem[GR13]{gill-rohde-type}
J.~T. Gill and S.~Rohde.
\newblock On the {R}iemann surface type of random planar maps.
\newblock {\em Rev. Mat. Iberoam.}, 29(3):1071--1090, 2013, \arxiv{1101.1320}.
  \MR{3090146}

\bibitem[HS18]{hs-euclidean}
N.~Holden and X.~Sun.
\newblock S{LE} as a mating of trees in {E}uclidean geometry.
\newblock {\em Comm. Math. Phys.}, 364(1):171--201, 2018, \arxiv{1610.05272}.
  \MR{3861296}

\bibitem[Kah85]{kahane}
J.-P. Kahane.
\newblock Sur le chaos multiplicatif.
\newblock {\em Ann. Sci. Math. Qu\'ebec}, 9(2):105--150, 1985. \MR{829798
  (88h:60099a)}

\bibitem[KMSW15]{kmsw-bipolar}
R.~{Kenyon}, J.~{Miller}, S.~{Sheffield}, and D.~B. {Wilson}.
\newblock {Bipolar orientations on planar maps and {SLE}$_{12}$}.
\newblock {\em Annals of {P}robability}, to appear, 2015, \arxiv{1511.04068}.

\bibitem[KMT76]{kmt}
J.~Koml\'os, P.~Major, and G.~Tusn\'ady.
\newblock An approximation of partial sums of independent {RV}'s, and the
  sample {DF}. {II}.
\newblock {\em Z. Wahrscheinlichkeitstheorie und Verw. Gebiete}, 34(1):33--58,
  1976. \MR{0402883}

\bibitem[Law05]{lawler-book}
G.~F. Lawler.
\newblock {\em Conformally invariant processes in the plane}, volume 114 of
  {\em Mathematical Surveys and Monographs}.
\newblock American Mathematical Society, Providence, RI, 2005. \MR{2129588
  (2006i:60003)}

\bibitem[{Lee}17]{lee-conformal-growth}
J.~R. {Lee}.
\newblock {Conformal growth rates and spectral geometry on distributional
  limits of graphs}.
\newblock {\em ArXiv e-prints}, January 2017, \arxiv{1701.01598}.

\bibitem[Lee18]{lee-uniformizing}
J.~R. Lee.
\newblock Discrete {U}niformizing {M}etrics on {D}istributional {L}imits of
  {S}phere {P}ackings.
\newblock {\em Geom. Funct. Anal.}, 28(4):1091--1130, 2018, \arxiv{1701.07227}.
  \MR{3820440}

\bibitem[LP16]{lyons-peres}
R.~Lyons and Y.~Peres.
\newblock {\em Probability on Trees and Networks}, volume~42 of {\em Cambridge
  Series in Statistical and Probabilistic Mathematics}.
\newblock Cambridge University Press, New York, 2016.
\newblock Available at \url{http://pages.iu.edu/~rdlyons/}. \MR{3616205}

\bibitem[LSW17]{lsw-schnyder-wood}
Y.~{Li}, X.~{Sun}, and S.~S. {Watson}.
\newblock {Schnyder woods, SLE(16), and Liouville quantum gravity}.
\newblock {\em ArXiv e-prints}, May 2017, \arxiv{1705.03573}.

\bibitem[MS15]{lqg-tbm1}
J.~{Miller} and S.~{Sheffield}.
\newblock {Liouville quantum gravity and the Brownian map I: The QLE(8/3,0)
  metric}.
\newblock {\em ArXiv e-prints}, July 2015, \arxiv{1507.00719}.

\bibitem[MS16a]{ig3}
J.~{Miller} and S.~{Sheffield}.
\newblock {Imaginary geometry III: reversibility of SLE$_\kappa$ for $\kappa
  \in (4,8)$}.
\newblock {\em Annals of Mathematics}, 184(2):455--486, 2016,
  \arxiv{1201.1498}.

\bibitem[MS16b]{lqg-tbm2}
J.~{Miller} and S.~{Sheffield}.
\newblock {Liouville quantum gravity and the Brownian map II: geodesics and
  continuity of the embedding}.
\newblock {\em ArXiv e-prints}, May 2016, \arxiv{1605.03563}.

\bibitem[MS16c]{lqg-tbm3}
J.~{Miller} and S.~{Sheffield}.
\newblock {Liouville quantum gravity and the Brownian map III: the conformal
  structure is determined}.
\newblock {\em ArXiv e-prints}, August 2016, \arxiv{1608.05391}.

\bibitem[MS16d]{ig1}
J.~Miller and S.~Sheffield.
\newblock Imaginary geometry {I}: interacting {SLE}s.
\newblock {\em Probab. Theory Related Fields}, 164(3-4):553--705, 2016,
  \arxiv{1201.1496}. \MR{3477777}

\bibitem[MS16e]{ig2}
J.~Miller and S.~Sheffield.
\newblock Imaginary geometry {II}: {R}eversibility of {$\operatorname{SLE}\sb
  \kappa(\rho\sb 1;\rho\sb 2)$} for {$\kappa\in(0,4)$}.
\newblock {\em Ann. Probab.}, 44(3):1647--1722, 2016, \arxiv{1201.1497}.
  \MR{3502592}

\bibitem[MS17]{ig4}
J.~Miller and S.~Sheffield.
\newblock Imaginary geometry {IV}: interior rays, whole-plane reversibility,
  and space-filling trees.
\newblock {\em Probab. Theory Related Fields}, 169(3-4):729--869, 2017,
  \arxiv{1302.4738}. \MR{3719057}

\bibitem[Mul67]{mullin-maps}
R.~C. Mullin.
\newblock On the enumeration of tree-rooted maps.
\newblock {\em Canad. J. Math.}, 19:174--183, 1967. \MR{0205882 (34 \#5708)}

\bibitem[Pol81a]{polyakov-qg1}
A.~M. Polyakov.
\newblock Quantum geometry of bosonic strings.
\newblock {\em Phys. Lett. B}, 103(3):207--210, 1981. \MR{623209 (84h:81093a)}

\bibitem[Pol81b]{polyakov-qg2}
A.~M. Polyakov.
\newblock Quantum geometry of fermionic strings.
\newblock {\em Phys. Lett. B}, 103(3):211--213, 1981. \MR{623210 (84h:81093b)}

\bibitem[RS05]{schramm-sle}
S.~Rohde and O.~Schramm.
\newblock Basic properties of {SLE}.
\newblock {\em Ann. of Math. (2)}, 161(2):883--924, 2005, \arxiv{math/0106036}.
  \MR{2153402 (2006f:60093)}

\bibitem[RV14]{rhodes-vargas-review}
R.~Rhodes and V.~Vargas.
\newblock Gaussian multiplicative chaos and applications: {A} review.
\newblock {\em Probab. Surv.}, 11:315--392, 2014, \arxiv{1305.6221}.
  \MR{3274356}

\bibitem[Sch00]{schramm0}
O.~Schramm.
\newblock Scaling limits of loop-erased random walks and uniform spanning
  trees.
\newblock {\em Israel J. Math.}, 118:221--288, 2000, \arxiv{math/9904022}.
  \MR{1776084 (2001m:60227)}

\bibitem[She07]{shef-gff}
S.~Sheffield.
\newblock Gaussian free fields for mathematicians.
\newblock {\em Probab. Theory Related Fields}, 139(3-4):521--541, 2007,
  \arxiv{math/0312099}. \MR{2322706 (2008d:60120)}

\bibitem[She16a]{shef-zipper}
S.~Sheffield.
\newblock Conformal weldings of random surfaces: {SLE} and the quantum gravity
  zipper.
\newblock {\em Ann. Probab.}, 44(5):3474--3545, 2016, \arxiv{1012.4797}.
  \MR{3551203}

\bibitem[She16b]{shef-burger}
S.~Sheffield.
\newblock Quantum gravity and inventory accumulation.
\newblock {\em Ann. Probab.}, 44(6):3804--3848, 2016, \arxiv{1108.2241}.
  \MR{3572324}

\bibitem[SS13]{ss-contour}
O.~Schramm and S.~Sheffield.
\newblock A contour line of the continuum {G}aussian free field.
\newblock {\em Probab. Theory Related Fields}, 157(1-2):47--80, 2013,
  \arxiv{math/0605337}. \MR{3101840}

\bibitem[Ste03]{stephenson-circle-packing}
K.~Stephenson.
\newblock Circle packing: a mathematical tale.
\newblock {\em Notices of the AMS}, 50(11):1376--1388, 2003.

\bibitem[Wer04]{werner-notes}
W.~Werner.
\newblock Random planar curves and {S}chramm-{L}oewner evolutions.
\newblock In {\em Lectures on probability theory and statistics}, volume 1840
  of {\em Lecture Notes in Math.}, pages 107--195. Springer, Berlin, 2004,
  \arxiv{math/030335}. \MR{2079672 (2005m:60020)}

\bibitem[Zai98]{zaitsev-kmt}
A.~Y. Zaitsev.
\newblock Multidimensional version of the results of {K}oml\'os, {M}ajor and
  {T}usn\'ady for vectors with finite exponential moments.
\newblock {\em ESAIM Probab. Statist.}, 2:41--108, 1998. \MR{1616527}

\bibitem[Zha08]{zhan-duality1}
D.~Zhan.
\newblock Duality of chordal {SLE}.
\newblock {\em Invent. Math.}, 174(2):309--353, 2008, \arxiv{0712.0332}.
  \MR{2439609 (2010f:60239)}

\bibitem[Zha10]{zhan-duality2}
D.~Zhan.
\newblock Duality of chordal {SLE}, {II}.
\newblock {\em Ann. Inst. Henri Poincar\'e Probab. Stat.}, 46(3):740--759,
  2010, \arxiv{0803.2223}. \MR{2682265 (2011i:60155)}

\end{thebibliography}
\bibliographystyle{hmralphaabbrv}

\end{document}